\documentclass{article}
\usepackage{fullpage}
\usepackage{fancyhdr}
\usepackage{amsmath, amssymb, latexsym, amscd, amsthm,amsfonts,amstext}
\usepackage{subfigure}
\usepackage{enumerate}
\usepackage{xcolor}
\usepackage{textcomp}
\usepackage{mathtools}
\usepackage{relsize}
\usepackage{cleveref}
\usepackage{graphicx}
\usepackage{enumerate}
\usepackage{graphicx}
\usepackage{color}
\usepackage{amssymb}
\usepackage{algorithm}
\usepackage{amsmath}
\usepackage{mathtools}
\usepackage{fancyhdr}
\usepackage{pgf,tikz}
\usetikzlibrary{arrows}

\DeclarePairedDelimiter\abs{\lvert}{\rvert}%
\DeclarePairedDelimiter\norm{\lVert}{\rVert}%
\makeatletter
\let\oldabs\abs
\def\abs{\@ifstar{\oldabs}{\oldabs*}}
\let\oldnorm\norm
\def\norm{\@ifstar{\oldnorm}{\oldnorm*}}
\makeatother

\makeatletter
\DeclareRobustCommand\widecheck[1]{{\mathpalette\@widecheck{#1}}}
\def\@widecheck#1#2{%
    \setbox\z@\hbox{\m@th$#1#2$}%
    \setbox\tw@\hbox{\m@th$#1%
       \widehat{%
          \vrule\@width\z@\@height\ht\z@
          \vrule\@height\z@\@width\wd\z@}$}%
    \dp\tw@-\ht\z@
    \@tempdima\ht\z@ \advance\@tempdima2\ht\tw@ \divide\@tempdima\thr@@
    \setbox\tw@\hbox{%
       \raise\@tempdima\hbox{\scalebox{1}[-1]{\lower\@tempdima\box
\tw@}}}%
    {\ooalign{\box\tw@ \cr \box\z@}}}
\makeatother

\def\curl{\operatorname{curl}}

\newtheorem{theorem}{Theorem}[section]
\newtheorem{lemma}[theorem]{Lemma}
\newtheorem{proposition}[theorem]{Proposition}

\numberwithin{equation}{section}

\numberwithin{equation}{section}

\newcommand\LEF[6]{{{\mathcal{L}({#1^{#2}}({#3}),{#4^{#5}}({#6}))}}}

\newcommand\n{\text{inc}}
\newcommand\s{\text{sca}}
\newcommand\e{+}
\newcommand\Ll{\mathnormal{L}}
\newcommand\ta{\text{t}}

\newcommand\R{\mathcal{R}}
\newcommand\dB{{\partial B}}
\newcommand\dD{{\partial D}}
\newcommand\Div{Div\,}
\newcommand\p{\partial}
\newcommand\Hh{\mathnormal{H}}

\newcommand\Polt{\bigl[\mathcal{P}_{\p D_i} \bigr]}
\newcommand\Vmt{\bigl[\mathcal{T}_{\p D_i} \bigr]}

\newcommand\SB[2]{{S^{#1}_{{#2}{#2},_B}}}
\newcommand\SD[2]{{S^{#1}_{{#2}{#2},_D}}}
\newcommand\KD[2]{{K^{#1}_{{#2}{#2},_D}}}
\newcommand\KB[2]{{K^{#1}_{{#2}{#2},_B}}}
\newcommand\MD[2]{{M^{#1}_{{#2}{#2},_D}}}
\newcommand\MB[2]{{M^{#1}_{{#2}{#2},_B}}}
\newcommand\Varepsilon{\mathlarger{\mathlarger{\epsilon}}}
\newcommand{\Sp}{\mathcal{S}}

\usepackage[mathscr]{eucal}
\usepackage{graphicx}
\usepackage{subfigure}
\usepackage{cite}

\usepackage{xcolor}

\def\curl{\operatorname{curl}}

\usepackage{url}
\def\curl{\operatorname{curl}}

\allowdisplaybreaks
\title{The Foldy-Lax Approximation for the Full Electromagnetic Scattering by Small Conductive Bodies of Arbitrary Shapes} 

\author{
 Ali Bouzekri
\thanks{Laboratoire de Mathematiques Pures et Appliquees, Universite Mouloud Mammeri, Tizi-Ouzou, Algeria
(Email: bouzekri.ali.lmpa@gmail.com).}
\and  Mourad Sini
\thanks{Radon institute (RICAM), Austrian Academy of Sciences, 69 Altenbergerstrasse, A4040, Linz, Austria
(Email: mourad.sini@oeaw.ac.at). This author is partially supported by the Austrian Science Fund (FWF): P28971-N32.}
}

\begin{document}
\graphicspath{{Figures-eps/}}
\maketitle

\begin{abstract}
  We deal with the electromagnetic waves propagation in the harmonic regime. We derive
  the Foldy-Lax approximation of the scattered fields generated by a cluster of small conductive inhomogeneities arbitrarily distributed in a bounded domain $\Omega$ of $\mathbb{R}^3$.
  This approximation is valid under a sufficient but general condition on the number of such inhomogeneities $m$, their maximum radii $\epsilon$ and the minimum distances between them $\delta$,
  of the form 
  \begin{equation}
  (\ln m)^{\frac{1}{3}}\frac{\epsilon}{\delta} \leq C, \nonumber
  \end{equation} 
  where $C$ is a constant depending only on the Lipschitz characters of the scaled inhomogeneities. In addition, we provide explicit error estimates
  of this approximation in terms of aforementioned parameters, $m, \epsilon, \delta$ but also the used frequencies $k$ under the Rayleigh regime. Both the far-fields and the near-fields (stated at a 
  distance $\delta$ to the cluster) are estimated. In particular, for a moderate number of small inhomogeneities $m$, the derived expansions are valid in the mesoscale regime where $\delta \sim \epsilon$.
  
  At the mathematical analysis level and based on integral equation methods, we prove a priori estimates of the densities in the $\Ll^{2,\Div}_t$ spaces instead of the usual $\Ll^2$ spaces (which are not enough).
  A key point in such a proof is a derivation of a particular Helmholtz type decomposition of the densities. Those estimates allow to obtain the needed qualitative as well as quantitative estimates while refining the approximation.
Finally, to prove the invertibility of the Foldy-Lax linear algebraic system, we reduce the coercivity inequality to the one related to the scalar
Helmholtz model. As this linear algebraic system comes from the boundary conditions, such a reduction is not straightforward.
\end{abstract}


\bigskip


\pagestyle{myheadings}
 \thispagestyle{plain}

\section{Introduction and main results}
Let $(B_i)_{i=1}^m$ be $m$ open, bounded and simply connected sets containing the origin, with Lipschitz boundaries. 
To these sets, we correspond the small bodies $(D_i)_{i=1}^m$ which are defined as the translations and contractions of the $m$ bodies $({B}_i)_{i=1}^m$, that is 
\begin{equation}
D_i=\epsilon {B}_i+z_i , i=1,...,m
\end{equation}
where $z_i, i=1, ..., m$ are given positions in $\mathbb{R}^3$ and $\epsilon$ a small parameter.

We consider the scattering of a time-harmonic electromagnetic plane wave 
by the perfectly conducting small bodies $(D_i)_{i=1}^m$ formulated as follows (see  \cite{cltn-krss(2013)})
\begin{equation}
 \begin{aligned}
 &\nabla \times E-ik H=0~~\text{in}~D^\e:=\mathbb{R}^3\setminus \cup_{i=1}^m \overline{D_i}, \label{Maxwell-equation}\\
 &\nabla \times H+ik E=0~~\text{in}~ D^\e,\\
 \end{aligned}
\end{equation}
with the boundary condition
\begin{equation}
\nu\times E=0 ~~~\text{on}~ \partial D=\cup_{i=1}^{m}\partial{D_i}. \label{Bondary-condition}
\end{equation}

The total electromagnetic fields are expressed as $E=E^\n+E^\s$, $H=H^\n+H^\s$ where  the indices ``$\n$'' and ``$\s$'' indicate
the incident wave and the  the scattered wave, respectively. The condition \eqref{Bondary-condition} corresponds to the case of perfectly conducting obstacle 
and $\nu$ expresses the unit outward normal vector to the boundary of $D=\cup_{i=1}^{M}\; D_i$. Here the wave number $k$ is defined through the relation $k^2=(\xi\omega+i\sigma)\mu\omega$ where  $\xi,\,\sigma,\,\mu$  
are respectively the electric permittivity, electric conductivity and the magnetic permeability. In the case where $\sigma =0$, and then $k$ is 
real, the scattered wave $(E^\s,H^\s)$ must satisfy the outgoing radiation condition 
\begin{equation}        
        \begin{aligned}
        \lim_{\abs{x}\rightarrow\infty}\left(H^\s(x)\times x-\abs{x}E^\s(x)\right)=0. \label{Rad-Condition}
         \end{aligned}
\end{equation}
Motivated by applications, we restrict ourselves to incident waves of the form $E^\n (x)=Pe^{ikx\cdotp \theta}$,
where $\theta$ is the incident direction, $P\in \mathbb{R}^3$ is the polarization that is orthogonal to $\theta$.\\
We introduce the diameters 
$\epsilon_i={\max}_{x,y \in D_i}d(x,y),~ i\in \{1,...,m\},$
and the distance between two bodies $D_i, D_j, i\neq j$, as $\delta_{i,j}=\min_{x\in D_i, y\in D_j} {{d(x,y)}},$ for every ${i,j\in \{1,...,m\};i\neq j}$
where $d(\cdotp,\cdotp)$ stands for the Euclidean distance. We set \begin{align}
               &\epsilon:=\max_{i\in \{1,...,m\}}{\epsilon_i},
               &\delta:=\min_{i\neq j\in \{1,...,m\}} {\delta_{i,j}}. 
              \end{align} 
We suppose in addition that $\cup_{i=1}^m\overline{D_i}\subset \Omega,$ where $\Omega$ is a  bounded Lipschitz domain such that 
              $$d(\p \Omega,\cup_{i=1}^m\overline{D_i})\geq \delta.$$ 
Let us recall that a bounded open connected domain $B$, is said to be a Lipschitz domain with character $(l_\dB,L_\dB)$ 
if for  each $x\in\partial D$ there exist a coordinate system $(y_i)_{i=1,2,3}$, a truncated cylinder $\mathfrak{C}$ centered at $x$ whose axis is parallel to $y_3$ with
length $l$ satisfying $l_\dB\leq l\leq 2l_\dB$, and a Lipschitz function $f$ that is $\abs{f(s_1)-f(s_2)}\leq L_\dB\abs{s_1-s_2}$ for 
 every $s_1,s_2\in \mathbb{R}^2$ , such that $B\cap \mathfrak{C}=\{(y_i)_{i=1,2,3}: y_3>f(y_1,y_2) \}$ 
 and $\dB\cap \mathfrak{C}=\{(y_i)_{i=1,2,3}: y_3=f(y_1,y_2) \}.$ In this work, we assume that the sequence of Lipschitz characters $(l_{\partial B_i},L_{\partial B_i})^m_{i=1}$ 
 of the bodies $B_i, i=1,..., m,$ is bounded from above and below.
 \bigskip
 
 The scattering problem \eqref{Maxwell-equation} under the boundary condition \eqref{Bondary-condition}  and the radiating condition \eqref{Rad-Condition} is well posed in appropriate spaces under appropriate conditions 
 (see \cite{Ned:Spring2001, cltn-krss(2013)}) which will described later. In addition, when $\Im k$ is different from zero, the scattered electromagnetic fields are fastly decaying at infinity as we have attenuation.
 But when $\Im k=0$, i.e.~in the absence of attenuation, we have the following
 behavior (as spherical-waves) of the scattered electric fields far away from the sources $D_i$'s
 \begin{equation}\label{electric-farfield}
  E^\s(x)=\frac{e^{ik\vert x\vert}}{\vert x\vert} \{ E^{\infty}(\tau) +O(\vert x\vert^{-1}) \}, ~~~ \vert x\vert \longmapsto \infty,
 \end{equation}
and we have a similar behavior for the scattered magnetic field as well:
\begin{equation}\label{magnetic-farfield}
  H^\s(x)=\frac{e^{ik\vert x\vert}}{\vert x\vert}\{H^\infty(\tau) +O(\vert x\vert^{-1})\}, ~~~ \vert x\vert \longmapsto \infty.
 \end{equation}
where $(E^\infty(\tau), H^\infty(\tau))$ is the electromagnetic far field pattern in the direction of propagation $\tau:=\frac{x}{\vert x\vert}$.
 \bigskip

 The goal of this work is to derive the Foldy-Lax approximation (also called the point interaction approximation) of the electromagnetic fields taking into account the whole parameters
 defining the model, i.e.~the three parameters $\epsilon, \delta$ and $m$ defining the conductors 
 but also the wave number $k$. In addition, the error of these approximations are uniform in terms of these parameters where the uniform bounds depend only on the a priori bounds of the Lipschitz character,
of the set of conductors, described above.
 In particular, we deal with the mesoscale regime where $\epsilon \sim \delta$. \\
To describe the results precisely, let us recall some properties and notations. For $i=1, ..., m$, we recall the single layer operator $[S_{ii,_D}^k]: L^2(\partial D_i)\rightarrow H^1(\partial D_i)$, defined as
        \begin{equation}
         [S_{ii,_D}^k](\psi)(x):= \int_{\p D_i}\Phi_k(x,y)\psi(y)~ds(y),\,~ x\in\p D_i,
        \end{equation}
and the double layer operator $[K^k_{ii,_D}]: L^2(\partial D_i)\rightarrow L^2(\partial D_i)$,
$$[K^k_{ii,_D}](\psi)(x)=\int_{\p D_i}\frac{\p \Phi_k}{\p \nu_y}(x,y)\psi(y)~ds(y),\,~ x \in \p D_i,$$
with its adjoint $$[(K^k_{ii,_D})^*](\psi)(x)=\int_{\p D_i}\frac{\p \Phi_k}{\p \nu_x}(x,y)\psi(y)~ds(y),\,~ x \in \p D_i,$$ where $\Phi_k(x,y)=\frac{1}{4\pi}\frac{e^{ik\abs{x-y}}}{\abs{x-y}}$ 
is the Green function for the Helmholtz equation at the wave number $k$.\medskip
                         
The operator $[\lambda I+(K^0_{ii,_D})^*]$ is invertible from $\Ll^2_0(\p D_i)$ onto itself for any complex number $\lambda$ such that $\vert \lambda \vert \geq \frac{1}{2}$, see \cite{Ammaripola, MD:InteqnsaOpethe1997} for instance. The following two quantities will play an important role in the sequel:
                            \begin{align}\label{Polarization-Tensor}
                           [\mathcal{P}_{\p D_i}]:=\int_{\p D_i}[-\frac{1}{2}I+(K^0_{ii,_D})^*]^{-1}(\nu)(y)y^Tds(y),
                           \end{align}
                    and    \begin{align}\label{Virtual-Mass-Tensor}
                           \Vmt:= \int_{\p D_i}[\frac{1}{2}I+(K^0_{ii,_D})^*]^{-1}(\nu)(y)y^Tds(y).
                           \end{align} 
The tensor $[\mathcal{P}_{\p D_i}]$ is  negative-definite symmetric matrix and $\Vmt$ is positive-definite symmetric matrix, 
                     (see {Lemma 5} and {Lemma 6} in \cite{Cedio-Fengya-1998} or Theorem 4.11 in \cite{Ammaripola}). Further, we have the following scales:
                     \begin{align}\label{scalingtensors}
                      [\mathcal{P}_{\p D_i}]=\epsilon^3[\mathcal{P}_{\p B_i}],~\text{and }~[\mathcal{T}_{\p D_i}]=\epsilon^3[\mathcal{T}_{\p B_i}].
                     \end{align} 
             Indeed, \footnote{Recall that $\int_{\p D_i}[-\frac{1}{2}I+(K^0_{ii,_D})^*]^{-1}(\nu)(y)~ z_i^Tds(y)=\int_{\p D_i}[-\frac{1}{2}I+(K^0_{ii,_D})^*]^{-1}(\nu)(y)~~ds(y)~=0$.} 
                     \begin{align*}
                     [\mathcal{P}_{\p D_i}]&=\int_{\p D_i}[-\frac{1}{2}I+(K^0_{ii,_D})^*]^{-1}(\nu)(y)~ (y-z_i)^Tds(y),\\
                                           &=\int_{\p B_i}[-\frac{1}{2}I+(K^0_{ii,_B})^*]^{-1}(\nu)(s_i)~ (\epsilon s+z_i-z_i)^T \epsilon^2~ds(y),\\
                                           &=\epsilon^3 [\mathcal{P}_{\p B_i}],
                    \end{align*} and we get it in the same way for the second identity, after noticing that\footnote{We have $[\frac{1}{2}I+K^{0}_{ii,_{D}}](z_i)=z_i[\frac{1}{2}I+K^{0}_{ii,_{D}}](1)=0$.}  
                    $$\Vmt:= \int_{\p D_i} \nu_y{\biggl[ [\frac{1}{2}I+K^{0}_{ii,_{D}}]^{-1}~( x) \biggr]}^Tds(y)=\int_{\p D_i} \nu_y{\biggl[ [\frac{1}{2}I+K^{0}_{ii,_{D}}]^{-1}~( x-z_i) \biggr]}^Tds(y).$$
                    
          For $i\in\{1,...,m\}$ let $(\mu_i^{\mathcal{T}})^+$, $(\mu_i^{\mathcal{P}})^+$ be the respective maximal eigenvalues of $[\mathcal{T}_{\p B_i}]$, $-[\mathcal{P}_{\p B_i}]$,  
      and let $(\mu_i^{\mathcal{T}})^{-}$, $(\mu_i^{\mathcal{P}})^{-}$ be their minimal ones. We define
                \begin{equation}\label{mudefinition}\begin{aligned}
                 \mu^+=\max_{i\in\{1,...,m\}}((\mu_i^{\mathcal{T}})^{+}, (\mu_i^{\mathcal{P}})^{+}),\\
                 \mu^-=\min_{i\in\{1,...,m\}}((\mu_i^{\mathcal{T}})^{-}, (\mu_i^{\mathcal{P}})^{-}).
                \end{aligned}
                \end{equation} Hence for every vector $\mathcal{C}$, we get 
          \begin{equation} \label{Tensor-Inequalities}  
                 \begin{aligned}
                 \mu^- \abs{\mathcal{C}}^2\epsilon^3\leq [\mathcal{T}_{\p D_i}]~\mathcal{C}\cdotp\mathcal{C}\leq\epsilon^3 \mu^+ \abs{\mathcal{C}}^2,\\
                 \mu^- \abs{\mathcal{C}}^2\epsilon^3\leq [-\mathcal{P}_{\p D_i}]~\mathcal{C}\cdotp\mathcal{C}\leq \epsilon^3\mu^+ \abs{\mathcal{C}}^2.
                \end{aligned}
          \end{equation} 
          The dyadic Green's function is given by  \begin{equation}\label{Dyadic-Green-function}
                                                   \Pi(x,y):=k^2\Phi_k(x,y)I+\nabla_x\nabla_x \Phi_k(x,y),\\
                                                   =k^2\Phi_k(x,y)I-\nabla_x\nabla_y \Phi_k(x,y). 
                                                   \end{equation}
We introduce generic functions $\Varepsilon(\delta^{s},\abs{k}^{l})$ and $\Varepsilon_{k,\delta,m}$ which express error functions as follows 
          \begin{equation}\label{Varepsilon-Generic-Function}
          \begin{aligned}
          &\Varepsilon(\delta^{s},\abs{k}^{l})=O(\frac{\max(1,\abs{k}^{l})}{\delta^{s}})\\
          &\Varepsilon_{k,\delta,m}:=O\Bigl(\frac{(\abs{k}+1)~\ln(m^\frac{1}{3})}{\delta^3}+ \frac{(\abs{k}+1)^2~m^\frac{1}{3}}{\delta^2}+\frac{(\abs{k}+1)^3~m^\frac{2}{3}}{\delta}\Bigr).\\
          \end{aligned}
          \end{equation}
Finally, we set $$C_{Ls}:={\bigl[ \frac{(1+\abs{k}^2)2^6 \Bigl((1+\frac{\abs{k}}{2}) D(\Omega)^\frac{1}{3}+\frac{\abs{k}}{2}D(\Omega)^\frac{2}{3}\Bigr) }{8\pi} +\frac{12^2\norm*{S_{_{B^1}}} }{\pi}+
                                            \frac{63^\frac{1}{2} \abs{k}^2D(\Omega)^{\frac{2}{3}}}{4\pi}\bigr]}$$ where $\norm*{S_{_{B^1}}}:=\norm*{S_{_{B^1}}}_{\mathcal{L}(\Ll^2(\p B^1),\Hh^1(\p B^1))}$ is the norm of the usual single layer potential 
on the unit sphere and $D(\Omega)$ is the diameter of $\Omega$. 
Now, we are ready to state the main result of this work.  
 \begin{theorem}\label{letheorem}
        Let $(\mathcal{A}_i)_{i=1}^m$ and $(\mathcal{B}_i)_{i=1}^m$ be the solutions of the following linear system
                                   \begin{equation}\label{linearsystem}
                                         \begin{aligned}
                                        \mathcal{A}_i&=-\Polt \sum_{(j\neq i)\geq1}^m \left(\Pi_k(z_i,z_j){\mathcal{A}}_j-k^2\nabla\Phi_k(z_i,z_j)\times{\mathcal{B}}_j\right) -\Polt\curl E^\n(z_i),\\
                                        \mathcal{B}_i&=\Vmt\sum_{(j\neq i)\geq1}^m\left(~ -\nabla_x\Phi_k(z_i,z_j)\times {\mathcal{A}}_j~+ \Pi_k(z_i,z_j){\mathcal{B}}_j\right)-\Vmt E^\n(z_i), 
                                        \end{aligned}
                                   \end{equation} which is invertible under the following sufficient condition \begin{align}\label{general-condition}
                                                                                                              C_{Li}:=1-C_{Ls}\frac{\mu^+\epsilon^3}{\delta^3}>0.
                                                                                                              \end{align}
 There exists a constant $C$ depending only on the Lipschitz character of the $B_i$'s such that if
                     \begin{align}\label{General-Condition}
                     \abs{k}^2\epsilon~+ (1+\abs{k}^2)\mu^+\frac{\epsilon^3}{\delta^3}+ ~\left(\frac{\ln({m^\frac{1}{3}})}{\delta^3}+\frac{2\abs{k} m^\frac{1}{3}}{\delta^2} +
                     \frac{m^\frac{2}{3}}{2\delta}\abs{k}^2\right)\epsilon^3 < C,
                    \end{align} then
                                         \begin{enumerate} 
                                          \item the scattered electric field has the following expansion, for $\Im k \geq 0$, for $x\in \mathbb{R}^3\setminus (\cup^m_{i=1}D_i)$ such that $\min_{1\leq i \leq m}d(x,D_i)=\delta$, 
                                  
\begin{equation}\label{Scattered-electric-field}
                                         \begin{aligned}
                                           E^\s(x)=&\sum_{i=1}^m \left(  \nabla\Phi_k(x,z_i)\times\widehat{\mathcal{A}}_i+
                                                    \curl \curl (\Phi_k(x,z_i) \widehat{\mathcal{B}}_i)\right)\\
                                                   &+\frac{(C_{L_i^2}\mu^-)^{-1}}{\mu^+}\times O^\Varepsilon(\frac{\epsilon^4}{\delta^4})+(C_{L_i^2}\mu^-)^{-1}\times O^\Varepsilon\bigl(\frac{\epsilon^7}{\delta^7}\bigr),
\end{aligned}
\end{equation} with $C_{L_i^2}={1-4\mu^+\left(\frac{\ln({m^\frac{1}{3}})}{\delta^3}+\frac{2\abs{k} m^\frac{1}{3}}{\delta^2} +\frac{m^\frac{2}{3}}{2\delta}\abs{k}^2\right)\epsilon^3}$.
                                          \item the far field pattern has the following expansion for $\Im k=0$
                                  \begin{equation}\label{Farfieldpattern}
                                         \begin{aligned}
                                            E^\infty(\tau)=&\frac{ik}{4\pi}\sum_{i=1}^m e^{-ik\tau.z_i}\tau\times 
                                            \bigl(\widehat{\mathcal{A}}_i-ik\tau\times \widehat{\mathcal{B}}_i\bigr) +O\left(~ (\abs{k}^3+\abs{k}^2)~ m\epsilon^4\right)\\
                                            &+\frac{\abs{k}}{2\pi}~\frac{\max(1,\abs{k})}{C_{Li}\mu^-}~O^\varepsilon\Bigl(\frac{\epsilon^4}{\delta^4}\Bigr)m\epsilon^3
\end{aligned}
\end{equation}
\end{enumerate}
and the errors in \eqref{Scattered-electric-field} and \eqref{Farfieldpattern} correspond to
                                        \begin{align}
                                            O^\Varepsilon\bigl(\frac{\epsilon^7}{\delta^7}\bigr):=&
                                                       O\Bigl(\frac{\epsilon^7}{\delta^7}+\Varepsilon(\delta^6,\abs{k}+\abs{k}^2+\abs{k}^3)\epsilon^7+\Varepsilon(\delta^5,\abs{k}^2)\epsilon^7\Bigr),\\
                                             O^\Varepsilon\bigl(\frac{\epsilon^4}{\delta^4}\bigr):=&
                                               O\Bigl(\frac{\epsilon^4}{\delta^4}+(1+\abs{k})\Varepsilon_{k,\delta,m} \epsilon^4+\max(1+\abs{k},\abs{k}^2 )\epsilon\Bigr).
                                        \end{align}

\end{theorem}


The approximation in (\ref{Scattered-electric-field}) and (\ref{Farfieldpattern}) are called the point-interaction approximations or the Foldy-Lax approximations as the dominant field is reminiscent to the field 
describing the interactions between the points $z_i, i=1, .., m$, with the scattering coefficients given by the polarization tensors $\Polt$ and $\Vmt$, see \cite{Martin:2006, C-H-S:2014} for particular situations. 
Since the pioneering works of Rayleigh till Foldy, the first and original goal of such approximations was to reduce the computation of the fields generated by a cluster of small bodies to inverting an algebraic system 
(called the Foldy linear algebraic system), see \cite{Martin:2006} for more information. With our approximations above, and regarding the full Maxwell system, such  a goal is reached with high generality as we take into account all the parameters, $m, \epsilon, \delta$ and $k$, modeling 
the scattering by the cluster of small 
conductors $D_i$'s.  
\bigskip

In the recent twenty years or so, there were different and highly important fields where such kind of approximations are key tools. 
Let us mention few of them which are of particular interest to us:
\begin{enumerate}
 \item Let us start with the mathematical imaging field where
the small bodies can model impurities or small tumors, for instance, that one should localize and estimate the sizes from the measured fields (either near fields or far-fields), see \cite{A-I-L-P:2007}. In this case, based on our
 approximations, we can indeed localize the small bodies by reconstructing the points $z_i$, via MUSIC type algorithms \cite{A-I-L-P:2007, C-H-S:2014}, and then estimate the polarization tensors. 
 From these polarization tensors, one can derive lower and upper estimates of the small bodies' sizes, see \cite{A-C-S:2016}. The small bodies can also model electromagnetic nanoparticles. Imaging using electromagnetic nanoparticles as contrast agent is a recent and highly attractive imaging modality that uses 
 special properties of the nanoparticles
 to create high contrasts in the tissue and then enhance the resolution of permittivity reconstruction for instance. There are at least two types of such special properties: one is related to the plasmonic nanoparticles 
 (which are nearly resonant nanoparticles but might create high dissipation) and the other one related to the all dielectric nanoparticles (which are characterized by their high refraction indices), see \cite{K-C-B-A:2017},
 for instance. 
 
\bigskip
 
\item A second field where this kind of approximations are useful is the material sciences. Indeed, arranging appropriately the small bodies in a given bounded domain, the whole cluster will generate electromagnetic 
fields which are close to the fields generated by related indices of refraction (or permittivities and permeabilities). These indices of refraction are dependent on the properties of the small bodies, as the size, 
geometry and their own possible contrasts in addition to the used frequencies. This opens the door to the possibility of creating desired and new materials. Such ideas are already tested and justified to some 
extent mathematically in the framework of the homogenization theory. However, this theory is based on the periodicity (or randomness) in distributing the small bodies. As we can see it from the approximations we provide
above, we can achieve similar goals but without assuming the periodicity. In addition, and as far the electromagnetic waves are concerned, we can handle in a unified way, the generation of volumetric metamaterials, 
Gradient metasurfaces and also metawires,  \cite{Y-L:2014, T:2017}. These properties will be quantified and justified in a future work were we plan to handle more general type of inhomogeneities than the conductive ones described in this work.
\end{enumerate}

\bigskip

Our contribution in this work is to have succeeded in handling the full Maxwell model by taking into account (explicitly) all the parameters modeling the small conductors $m, \epsilon$ and $\delta$ but also the used frequency $k$. 
To our best knowledge, there is no result in the literature where such approximations are provided with such generality and precision. At the mathematical analysis level, and as we are using integral equations methods,
we needed to derive an a priori estimate of the related densities. The first key observation here is to derive it in the $\Ll^{2,\Div}_t$ spaces instead of the usual $\Ll^2$ spaces. As a second observation, to derive such estimates,
we used a particular decomposition of the densities, see \Cref{Key-Decomposition} or Theorem \ref{Decomposition-density-estimate}, which allows to obtain the needed qualitative as well as quantitative estimates while refining the approximation.
Finally, to prove the invertibility of the algebraic system (\ref{linearsystem}) under the general condition (\ref{linearsystem}), the key point is to have reduced the coercivity inequality to the one related to scalar
Helmholtz model. Let us emphasize here that as this linear algebraic system comes form the boundary conditions, such a reduction of Maxwell to Helmholtz is not a trivial one (even, a priori, not a natural one). 

The only restriction we have in our condition (\ref{General-Condition}) is the appearance of the factor $\ln(m)$. 
At the technical level, see the last part of the proof of Theorem \ref{Decomposition-density-estimate}, its appearance is due to the fact the singularity of the dyadic Green's function (\ref{Dyadic-Green-function}) 
is of the order 3 (in contrast to the ones of the Green's functions for the Laplace or Lam\'e related models). 
We believe that this factor can be removed using more pde tools to invert the Calderon operator, see (\ref{Lin-Sys-Departur-Equation}), and hope to report on this in the near future.

The closest published works (i.e. deriving the Foldy-Lax type of approximations) are those by  A. Ramm in one side and those by V. Maz'ya, A. Movchan and M. Nieves in another side.
The several works by A. Ramm on Maxwell are derived more in a formal way, see \cite{Ramm-1, Ramm-2}. In addition, condition of the type $\frac{\epsilon}{\delta}<<1$ are used (meaning at least that the mesoscale 
regimes where $\epsilon \sim \delta$ are not handled) and without clarifying the rates. Finally, and unfortunately, to our opinion the form of the derived algebraic system is unclear and questionable. 
In a series of works dedicated to the Laplace and Lam\'e models (assuming $k=0$), V. Maz'ya, A. Movchan and M. Nieves
proposed a method which indeed takes into account the parameter and state the results in the mesoscale regimes too, see \cite{M-M:MathNach2010, M-M-N:MMS:2011, M-M-N:Springerbook:2013}. 
One possible limit to their approach is the need of the maximum principle in handling the link between the system on the boundary to fields outside.
This might be a handicap for tackling the Maxwell system for which such maximum principles are not at our hands.
\bigskip

The remaining part of the paper is dedicated to the proof of Theorem \ref{letheorem} and it is arranged as follows. In section \ref{Section-a-priori-estimate}, we recall and discuss the well posedness of the scattering 
problem via integral equation method and then derive the key a priori estimates of the densities. In section \ref{Section-first-approximation}, we derive the fields approximations with the corresponding non homogeneous
linear system. In section \ref{Section-invertibility-algebraic-system}, we justify and quantify the invertibility of the algebraic system and in section \ref{Section-end-of-the-proof}, we combine the estimates in the last two sections to complete the proof of Theorem \ref{letheorem}. 

\section{Existence, unique solvability and an a priori estimation of the density}\label{Section-a-priori-estimate}
\subsection{Preliminaries}

   Let us recall few properties of the surface divergence which will be important in our later analysis, see (Section 4 in \cite{DMM96} and Chapter 2 in \cite{cltn-krss(1983)}) for more details. 
   First, we recall the surface gradient of a smooth function $\phi$ on $\partial D$, $\nabla_\ta$, as  $\nabla_\ta \phi:=\nabla \phi-(\nu\cdotp\nabla\phi)\nu$  where $\nu$ is the exterior 
   unit normal to $\partial D$. Then the (weak) surface divergence for a tangential field $a$ is defined using the duality
                                                                     \begin{equation}\label{Div-definition}
                                                                      \int_\dD \phi \Div a~ds=-\int_\dD\nabla_\ta\phi\cdotp a~ds,  
                                                                     \end{equation} for every $\phi\in\mathcal{C}^\infty(\p D)$.

If $a$ is a tangential field for which $\Div a$ exists in the sense above, and it is in $L^1(\partial D)$ for instance, then, taking $\phi(x)=x_i$ in \eqref{Div-definition}, we have
                                                  \begin{equation}
                                                   \int_\dD x\Div a(x)~ds(x)=-\int_\dD a(x)~ds(x) \label{Divrel2}
                                                  \end{equation}  
and taking $\phi(x)=1$, we have \begin{equation}
                                  \int_\dD \Div a(x)~ds(x)=0. \label{Divrel-average-zero}
                                  \end{equation}  
When $a:=\nu\times u$ for a certain sufficiently smooth vector field $u$, we get \begin{align}\label{npcurl}
                                                                                   \Div a= -\nu\cdotp\curl u.
                                                                                   \end{align} 
Further, for a scalar function $\psi,$ being $\psi~a$ tangential (i.e.~$\psi~a\cdotp\nu=\psi~(a\cdotp \nu)=0$), the following identity holds, $\Div(\psi a)=\nabla_t\psi \cdotp a+ \psi  \Div a,$ and
hence \begin{equation}\label{Div-Identity-transform}
      \int_{\p D_i}\psi~a~ds
      =\int_{\p D_i}(x-z_i)(\nabla\psi(x)\cdotp a(x)+\psi(x) \Div~a(x))~ds(x).
      \end{equation}
Indeed, $\int_{\p D_i}\psi~a~ds=\int_{\p D_i}\nabla(x-z_i)~\bigl(\psi~a\bigr)(x)~ds(x)=\int_{\p D_i}(x-z_i)\Div (\psi~a)(x)~ds(x),$ here $\nabla(x-z_i)=I$ where $I$ stands for the identity matrix of $\mathbb{R}^3\times \mathbb{R}^3.$
\bigskip

 The spaces $\Ll^{p}_\ta(\dD)$ and $\Ll^{p,Div}(\p D)$ denote respectively the space of all tangential fields of 
of $\Ll^p(\p D);$ and the subspace of $\Ll^p_\ta(\p D)$ that have an $\Ll^p$ weak surface divergence, precisely 
$$\Ll^{p}_\ta(\dD)=\{a\in\Ll^{{p}}(\p D);\, a\cdotp\nu=0\},$$
$$\Ll^{p,Div}_\ta(\dD)=\{a\in\Ll^{{p}}_\ta(\p D);\, \Div a\in \Ll^p_0(\p D)  \}$$
where $\Ll^p_0(\p D):=\{u\in \Ll^p(\p D), \mbox{ such that } \int_{\partial D}u~ds=0\} $.
\bigskip
                                                  
\subsection{Existence and uniqueness of the solution} 

The solution to the problem \eqref{Maxwell-equation} under the boundary condition \eqref{Bondary-condition}  and the radiating conditions \eqref{Rad-Condition} can be expressed 
 in terms of boundary integral equation (see \cite{cltn-krss(2013), Ned:Spring2001}), under certain conditions in appropriate spaces that will be 
 specified later, using either one of the representations 
                         \begin{align}
                          E(x)=E^i(x)+\curl\int_{\partial D^{\e}}\Phi_k(x,y)a(y)~ds(y), \hspace{2cm}x\in\mathbb{R}^3\setminus (\bar{D}:=\cup_{i=1}^m \bar{D_i}) \label{cas1}
                         \end{align} 
                         \begin{align}
                         E(x)=E^i(x)+\curl\curl\int_{\partial D^{\e}}\Phi_k(x,y)a(y)~ds(y),  \hspace{2cm}x\in\mathbb{R}^3\setminus \bar{D} \label{cas2}     
                         \end{align} 
or a linear combination of the two, where $a$ is the unknown vector density to be found to solve the problem.\\
Let us consider the representation \eqref{cas1}, i.e. for a tangential field $a$ and $s\in D^\e$
\begin{equation}
 E^\s(s)=\curl\int_{\partial D^{\e}}\Phi_k(s,y)a(y)~ds(y),
\end{equation} 
and let be $\bigl\{\Gamma_{+}(x),\Gamma_{-}(x),\,~ x\in\cup_{i=1}^m\p D_i\bigr\}$  a family of doubly truncated cones with a vertex at $x$ such that $\Gamma_{\pm}(x)\cap D^\mp=\emptyset$,
we have for almost  every $x\in \cup_{i=1}^m\p D_i$ 
(see \cite{DMM96}) \begin{equation}
                           \lim_{{\substack{{s \rightarrow x}\\{s\in \Gamma_{\pm}}(x)}}} E^\s(s)=
                           (\mp\frac{1}{2}\nu\times a+ \curl\int_{\partial D}\Phi_k(x,y)a(y)~ds(y)
                           \end{equation} 
where the integral is taken in the principal value of Cauchy sense, and the identity must be understood in the sens of trace operator, 
                        then using the condition \eqref{Bondary-condition}, we get
                        \begin{equation}
                         \nu\times\lim_{{\substack{{s \rightarrow x}\\{s\in \Gamma_{\pm}}(x)}}} E^\s(s)=[\pm\frac{1}{2}I+M_{\p D}^k](a)(x):=
                         \pm\frac{1}{2}a+ \nu\times\curl\int_{\partial D} \Phi_k(x,y)a(y)~ds(y)
                        \end{equation} 
where $[\pm\frac{1}{2}I+M_{\p D}^k]$ is called the magnetic dipole operator. Consequently, to solve the scattering problem we need to solve the integral equation \begin{equation}
                                           [\frac{1}{2}I+M_{\p D}^k](a)(x)=-\nu\times E^\n,\; ~~ \mbox{ on } \cup^m_{i=1} \partial D_i\label{Equation-to-solve}
                                          \end{equation}
                                          or,
                                          \begin{equation}\label{equation-to-be-solved}
                                           [\frac{1}{2}I+M_{{\p D_i}}^k](a)(x_i)+[\sum_{(j\neq i)\geq1}^m M_{ij,{_D}}^k](a)=
                                           -\nu\times E^\n(x_i),\; ~~ x_i \in \partial D_i.
                                          \end{equation} 
In this case the magnetic field is represented by  
\begin{equation*} 
H^\s(x)=-ik\curl\curl\int_{\partial D^{\e}}\Phi_k(x,y)a(y)~ds(y).
\end{equation*} 

Further it satisfies \begin{equation}\label{Continuity-magnetic-Dipol-operator}
                                    \lim_{{\substack{{s \rightarrow x}\\{s\in \Gamma_{\pm}}(x)}}}\nu\times H^\s(s)=[N_{\p D}^k](a)=k^2[\nu\times S_{\p D}^k](a)+[\nu\times \nabla S_{\p D}^k](\Div a)
                                   \end{equation} which means that $\nu\times (H^\s)^+=(\nu\times H^\s)^-$ i.e~$H$ is continuous across the boundary, the operator $N_k$ is called electric dipole operator.

We can write the equations in \eqref{equation-to-be-solved} in a compact form as follows
\begin{equation}
(\frac{1}{2}I+\mathcal{M}_D+\mathcal{M}_N)A=-\nu\times E^I, \label{Lin-Sys-Departur-Equation} 
\end{equation}
where $\mathcal{A}=(a_1,...,a_m)^T$ is a column matrix which vectorial components are $a_i:=a/\partial {D_i}$. Similarly, $E^I=(E^\n_1,......,E^\n_m)$ with $E^\n_i=E^\n/\p D_i$ 
and $\mathcal{M}_D$ is the diagonal matrix operator given by
\begin{equation*}
 \mathcal{M}_D:=\left\{\begin{aligned}
             &M_{ij,_{D}}^k \,\,\text{if }~i=j\\
             &0 \,\,\text{ otherwise}
            \end{aligned}
\right.
\end{equation*}
and finally $\mathcal{M}_N$ is the matrix operator with null diagonal
\begin{equation*}
 \mathcal{M}_N=\left\{\begin{aligned}
             &0 \,\,\text{ if }~ i=j\\
             &M_{ij,_{D}}^k \,\,\text{ if }~i\neq j
            \end{aligned}
\right.
\end{equation*}
where
\begin{equation*}
 M_{ij,_{D}}^ka(x):=\nu\times\int_{\partial D_j}\nabla_x\Phi_k(x,y)\times a(y)~ds(y),\, \text{for every } x\in \partial D_i.
\end{equation*}
For $k\in\mathbb{C}\setminus{\{0\}}$ such that $\Im k\geq 0$, the operators $(\pm\frac{1}{2}I+M_{ii,_D}^k)$ are Fredholm with index zero 
from $\Ll_\ta^{2,\Div}(\dD_i)$  into it self, furthermore $(\pm\frac{1}{2}I+M_{ii,_D}^k)$ are Fredholm from $\Ll_\ta^{2,\Div}(\p D_i)/\Ll_\ta^{2,0}(\p D_i)$ into it 
self, likewise for $\Ll_\ta^{2}(\p D_i)$. Moreover if $k$  is not a Maxwell eigenvalue for $D_i$ then  the operators are in fact isomorphisms, see Theorem 5.3 \cite{DMM96} and the remark after, and \cite{MM2004}.

Let us notice that when $\Im k >0,$ then $k$ is not a Maxwell eigenvalue. In addition, when $\Im k=0$ and as the radius of $D_i$ is small, by a scaling argument, this condition on $k$ is obviously fulfilled. 
As $\pm\frac{1}{2}I+\mathcal{M}_D$ is an isomorphism and $\mathcal{M}_N$ is compact (since the kernel of each component is of class $\mathcal{C}^{\infty}$), 
the operators  \begin{equation}
                \pm\frac{1}{2}I+\mathcal{M}_D+\mathcal{M}_N:
                \prod_{i=1}^{m}\mathcal{E}(\partial D_i)\longrightarrow\prod_{i=1}^{m}\mathcal{E}(\partial D_i),\label{operateur1}
                \end{equation} where $\mathcal{E}:= \Ll^{2}_\ta,\,~ \Ll^{2,Div}_\ta,$
 are Fredholm with zero index, so to show that the operator above is in fact an isomorphism it is enough to show that the homogeneous problem (i.e~ $E^\n=0$), has the unique 
 identically null solution density that is $\mathcal{A}=0$.
 We derive from \eqref{Equation-to-solve}, for $E^\n\equiv0 $ and the uniqueness to the exterior boundary problem that 
 \begin{align}
  E^\s(x)=\curl\int_{\partial D^{\e}}\Phi_k(x,y)a(y)~ds(y)\equiv 0,\,~ x\in\mathbb{R}^3\setminus {\overline{D}},
 \end{align} taking the rotational, we obtain \begin{equation}
                                           H^\s(x)=-ik \curl\curl\int_{\partial D^{\e}}\Phi_k(x,y)a(y)~ds(y)\equiv0,\,~x \in\mathbb{R}^3\setminus {\overline{D}},
                                           \end{equation} then going to the boundary using \eqref{Continuity-magnetic-Dipol-operator} we get
                                                                                      \begin{equation*}
                                                                                       \nu\times {H^\s}^-=0,\,~x\in \dD=\cup_{i=1}^m\dD_i.
                                                                                      \end{equation*} 
                                                                                      As the homogeneous interior boundary problem admits the unique identically null solution, we get $H^\s(x)=0,\; x \in D$ and hence taking the rotational gives us $E^\s(x)=0,\; x \in D$ and ultimately $(\nu\times E^\s)^-=0$ on $\p D$. Finally we have                                                                                      \begin{equation}
                                                                                       [\pm\frac{1}{2}I+M_{ii,_D}^k+\sum_{(j\neq i)\geq1}^m M_{ij,{_D}}^k](a)=0
                                                                                      \end{equation} and taking the difference of the two identities we 
                                                                                       get $a_i=0$ for every $i\in\{1,...,m\},$ and yield that $a\equiv0$ on $\p D=\cup_{i=1}^m\p D_i$.
This shows that we have existence of the solution of our original scattering problem and it can be represented as \eqref{cas1} with a unique tangential density $a$.
The uniqueness of the solution for the original scattering problem is deduced in the same way as it is done in Theorem 6.10 in \cite{cltn-krss(2013)} for instance. 

\subsection{A priori estimates of the densities}
 In order to derive suitable estimates of the densities $a_i, i=1, ..., m$, we need to use the Helmholtz decomposition based on the following operators, which are  isomorphism (see Theorem 5.1 and Theorem 5.3 in \cite{DMM96}),
  \begin{equation}\label{isomorphism-decomposition}
   \begin{aligned}
  \nu\cdotp\curl \SD{0}{i}&:\Ll^{2,0}_\ta(\p D_i)\longrightarrow\Ll^2_0(\p D_i),\\
  \nu\times\nabla \SD{0}{i}&:\Ll^2_0(\p D_i)\longrightarrow\Ll_{t}^{2,0}(\p D_i),\\
  \nu\times\SD{0}{i}&:\Ll^{2,0}_{t}(\p D_i)\longrightarrow(\Ll^{2,Div}_\ta\setminus\Ll^{2,0}_\ta)(\p D_i):=\Ll^{2,Div}_{t}(\p D_i)\setminus\Ll^{2,0}_{t}(\p D_i).
 \end{aligned}
 \end{equation}
The following decomposition holds, 
\begin{proposition}\label{Key-Decomposition}
 Each element $V$ of $\Ll^{p,Div}_\ta(\p D_i)$ can be decomposed as 
\begin{equation}\label{Helmohotz-decomposition}
V=\mathfrak{V} + \nu\times\nabla \mathfrak{v}
\end{equation} 
 where
\begin{equation}
 \begin{aligned}
\mathfrak{V}~&\in \Ll^{p,Div}_\ta(\p D_i)\setminus\Ll^{p,0}_\ta(\p D_i),\\ 
\nu\times\nabla \mathfrak{v}~ &\in \Ll^{p,0}_\ta(\p D_i),~ \mathfrak{v}\in\Hh^1(\p D_i)\setminus C;
\end{aligned} 
\end{equation}
which satisfy
\begin{align}
 \norm*{\mathfrak{V}}_{\Ll^{2}(\p D_i)}\leq C_1\epsilon\norm*{V}_{\Ll^{2,\Div}_\ta(\p D_i)},\label{Div-part-L2estimates}\\
 \norm*{\mathfrak{V}}_{\Ll^{2,\Div}_\ta(\p D_i)}\leq C_2\norm*{V}_{\Ll^{2,\Div}_\ta(\p D_i)},\label{Div-part-L2Divestimates}\\
 \norm*{\nu\times\nabla \mathfrak{v}}_{\Ll^{2,0}_\ta(\p D_i)}\leq C_3\norm*{V}_{\Ll^{2,\Div}_\ta(\p D_i)},\label{nu-Grad-part-L2estimates}
\end{align}  and \begin{equation}
                     \norm*{\mathfrak{v}}_{\Ll^{2}(\p D_i)} \leq C_4 \epsilon\norm*{a}_{\Ll^{2,\Div}_\ta(\p D_i)}\label{Grad-part-v-L2}
                     \end{equation}
where $(C_i)_{i=1,2,3,4.}$ are constants which depend only on ${B}_i's$,
\end{proposition} 

To prove Proposition \ref{Key-Decomposition}, we need the following identities.
\begin{lemma}
     For $x_i\in\p D_i$, $s_i\in\p B_i$, with $x_i=\epsilon s_i+z$, and $\hat{a}(s_i)=a(\epsilon s_i+z_i),$ we have
    the following scaling identities
   \begin{equation}\label{L2div-scaling}
       \epsilon\norm*{\hat{a}}_{\Ll^2_{t}(\p B_i)}=\norm*{a}_{\Ll^2_{t}(\p D_i)},\,~
       \norm*{\Div\hat{a}}_{\Ll^2_{0}(\p B_i)}=\norm*{\Div a}_{\Ll^2_{0}(\p D_i)},
    \end{equation}
    and 
    \begin{align}
      \norm*{[\nu\times\nabla \SD{0}{i}]^{-1}}_\LEF{\Ll_{t}}{2,0}{\p D_i}{\Ll_{0}}{2}{\p D_i}&= \norm*{[\nu\times\nabla S_{ii,{_B}}^0]^{-1}}_\LEF{\Ll_{t}}{2,0}{\p B_i}{\Ll_{0}}{2}{\p B_i},\label{nablaSD-scaling}\\ 
      \norm*{[\nu\cdotp\curl \SD{0}{i}]^{-1}}_\LEF{\Ll_{0}}{2}{\p D_i}{\Ll_{t}}{2,0}{\p D_i}&= \norm*{[\nu\cdotp\curl \SB{0}{i}^0]^{-1}}_\LEF{\Ll_{0}}{2}{\p B_i}{\Ll_{t}}{2,0}{\p B_i}\label{curlSD-scaling}.
     \end{align}

\end{lemma}

\begin{proof} We derive the identities in (\ref{L2div-scaling}) as follows
\begin{align*}
\norm*{a}_{\Ll^2_{t}(\p D_i)}=\left(\int_{\p D} \abs{a(x)}^2ds(x)\right)^\frac{1}{2}
                  =\left(\int_{\p B} \abs{\hat{a}(s)}^2\epsilon^2ds(s)\right)^\frac{1}{2}
                  &=\left(\epsilon^2\int_{\p B} \abs{\hat{a}(s)}^2~ds(s)\right)^\frac{1}{2}\\
                  &=\epsilon\norm*{\hat{a}}_{\Ll^2_{t}(\p B_i)}.
\end{align*} 
For the second identity, one has for $\phi \in \Hh^1(\p D_i),$ \footnote{The gradient stands for the surface gradient. } 
             \begin{align*}
             \int_{\p D_i}\Div_{x} a(x) \phi(x)~ds(x)&=\int_{\p D_i}a(x)\cdotp \nabla_x \phi(x)~ds(x)
             =\int_{\p D_i}\widehat{a}(s)\cdotp \frac{1}{\epsilon}\nabla_s \widehat{\phi}(s)\epsilon^2ds(x),\\
             &=\int_{\p D_i}\Div_s\widehat{a}(s)\frac{1}{\epsilon} \widehat{\phi}(s)~ds(x)
             =\int_{\p D_i}\frac{1}{\epsilon}\Div_s\widehat{a}(s) \phi(x)~ds(x),
             \end{align*} then $\Div_{x} a(x)= \frac{1}{\epsilon}\Div_s\widehat{a}(s)$. It remains to take the norm to conclude.
             
Concerning \eqref{nablaSD-scaling} we have,\begin{align*}
                                        [\nu\times\nabla \SD{0}{i}](u)(x)
                                        =
                                        \nu\times\nabla \frac{1}{4\pi}\int_{\p D_i} \frac{1}{\abs{x-y}}u(y)~ds(y)
                                        =
                                        \nu\times\frac{1}{4\pi}\int_{\p D_i} \frac{1}{\abs{x-y}^3} (x-y)u(y)~ds(y),
                                       \end{align*} hence, for $x_i=\epsilon s_i+z_i$ and $y_i=\epsilon t_i+z_i$ , and $\hat{u}(s_i)=u(\epsilon t_i+z_i)$
                                                          \begin{align*}
                                                           [\nu_{x_i}\times\nabla \SD{0}{i}](u)(x)
                                                           &=
                                                           \nu_{s_i}\times\frac{1}{4\pi}\int_{\p B_i} \frac{1}{\abs{\epsilon s_i+z_i-\epsilon t_i-z_i}^3} (\epsilon s_i+z_i-\epsilon t_i-z_i)u(\epsilon t_i+z_i)~\epsilon^2~ds(t),\\
                                                           &=
                                                           \nu_{s_i}\times\frac{1}{4\pi}\int_{\p B_i} \frac{1}{\epsilon^3\abs{ s_i-t_i}^3}~ \epsilon~(s_i- t_i)u(\epsilon t_i+z_i)~\epsilon^2~ds(t),\\
                                                           &=
                                                           \nu_{s_i}\times\frac{1}{4\pi}\int_{\p B_i} \frac{1}{\abs{ s_i-t_i}^3}~(s_i- t_i)\hat{u}(t_i)~ds(t)=[\nu_{x_i}\times\nabla \SB{0}{i}](\hat{u})(s_i).
                                                          \end{align*} 
From the equality $[\nu_{x_i}\times\nabla \SD{0}{i}]([\nu_{x_i}\times\nabla \SD{0}{i}]^{-1}(u))=u,$ we get replacing in the previous identity $u$ by $[\nu_{x_i}\times\nabla \SD{0}{i}]^{-1}(u)$, that 
$[\nu_{x_i}\times\nabla \SB{0}{i}](\widehat{[\nu_{x_i}\times\nabla \SD{0}{i}]^{-1}(u)})=\widehat{u}$. Finally inverting the left-hand side operator, we have the scales 
                                                                                                                \begin{align}
                                                                                                                 \widehat{[\nu_{x_i}\times\nabla \SD{0}{i}]^{-1}(u)}
                                                                                                                =
                                                                                                                [\nu_{x_i}\times\nabla \SB{0}{i}]^{-1}\widehat{u}.\label{no-scaling-1} 
                                                                                                                \end{align} 
As \begin{align*}
          \norm*{[\nu\times\nabla \SD{0}{i}]^{-1}}_\LEF{\Ll_{t}}{2,0}{\p D_i}{\Ll_{0}}{2}{\p D_i}
          =
          \sup_{\stackrel{u\in\Ll^2_0(\p D_i)}{u\neq0}}\frac{\norm*{[\nu\times\nabla \SD{0}{i}]^{-1}(u)}_{\Ll^2(\p D_i)}}{\norm*{u}_{\Ll^2(\p D_i)}}
         \end{align*} \eqref{L2div-scaling} implies 
         \begin{align*}
          \norm*{[\nu\times\nabla \SD{0}{i}]^{-1}}_\LEF{\Ll_{t}}{2,0}{\p D_i}{\Ll_{0}}{2}{\p D_i}
          =
          \sup_{\stackrel{\widehat{u}\in\Ll^2_0(\p B_i)}{\widehat{u}\neq0}}\frac{\epsilon\norm*{\widehat{[\nu\times\nabla \SD{0}{i}]^{-1}(u)}}_{\Ll^2(\p B_i)}}{\epsilon\norm*{\widehat{u}}_{\Ll^2(\p B_i)}},
         \end{align*} 
and using \eqref{no-scaling-1} \begin{align*}
                                    \norm*{[\nu\times\nabla \SD{0}{i}]^{-1}}_\LEF{\Ll_{t}}{2,0}{\p D_i}{\Ll_{0}}{2}{\p D_i}
                                    =
                                    \sup_{\stackrel{\widehat{u}\in\Ll^2_0(\p B_i)}{\widehat{u}\neq0}}\frac{\norm*{{[\nu\times\nabla \SB{0}{i}]^{-1}(\widehat{u})}}_{\Ll^2(\p B_i)}}{\norm*{\widehat{u}}_{\Ll^2(\p B_i)}},
                                    \end{align*} the conclusion follows immediately. Similarly, we can derive \eqref{curlSD-scaling}. 
\end{proof}
\begin{proof}(Of Proposition \ref{Key-Decomposition})
It suffices to seek for the solution of the following equation 
 \begin{align}\label{tosolve}
  \nu\times\nabla \SD{0}{i}(v)+\nu\times\SD{0}{i}(w)=V.
 \end{align} Taking the surface divergence we have, $\nu\cdotp\curl \SD{0}{i}(w)=\Div V$ and then using \eqref{isomorphism-decomposition},
    $w=[\nu\cdotp\curl \SD{0}{i}]^{-1}\Div V$. Using \eqref{nablaSD-scaling}
    we get the estimate \begin{align}\label{The-first-estimation-decomposition}
           \norm*{w}_{\Ll_{t}^{2,0}{(\p D_i)}}\leq \norm*{[\nu\cdotp\curl S_{\p B_i}^0]^{-1}}_\LEF{\Ll_{0}}{2}{\p B_i}{\Ll_{t}}{2,0}{\p B_i}
           \norm*{(\Div V)}_{\Ll_{0}^2{(\p D_i)}}.
           \end{align}  
Put $\mathfrak{V}=\nu\times\SD{0}{i}(w)$, \begin{align*}
                                          \norm*{\mathfrak{V}}_{\Ll^{2}{(\p D_i)}}
                                          &=\left(\int_{\p D_i}[\nu\times\SD{0}{i}(w)]^2~ds\right)^\frac{1}{2}
                                          =\left(\int_{\p B_i}\left(\nu\times\int_{\p B_i}\frac{1}{\epsilon\abs{s-t}}\widehat{w}(t)\epsilon^2~ds(t)\right)^2 \epsilon^2~ds(s)\right)^\frac{1}{2},\\
                                          &=\epsilon^2\left(\int_{\p B_i}\left(\nu\times\int_{\p B_i}\frac{1}{\abs{s-t}}\widehat{w}(t)~~ds(t)\right)^2ds(s)\right)^\frac{1}{2}
                                          =\epsilon^2 \norm*{(\nu\times \SB{0}{i}(\widehat{w}))}_{\Ll^{2}{(\p B_i)}}.
                                          \end{align*}                                             
Using \eqref{The-first-estimation-decomposition} for the last coming inequality, we have
                                     \begin{equation}\label{V-in-L20}
                                     \begin{aligned}         
                                     \norm*{\mathfrak{V}}_{\Ll^{2}{(\p D_i)}}
                                             &=\epsilon^2 \norm*{(\nu\times \SB{0}{i}(\widehat{w}))}_{\Ll_{t}^{2}{(\p B_i)}},\\
                                             &\leq\epsilon^2 \left(\norm*{(\nu\times \SB{0}{i}(\widehat{w}))}_{\Ll_{t}^{2}{(\p B_i)}}+
                                              \norm*{\Div(\nu\times \SB{0}{i}(\widehat{w}))}_{\Ll_{0}^{2}{(\p B_i)}}\right),\\
                                              &\leq\norm*{(\nu\times \SB{0}{i}(\widehat{w}))}_{\Ll_{t}^{2,\Div}{(\p B_i)}},\\
                                              &\leq\epsilon^2 \norm*{\nu\times S_{\p B_i}^0}_\LEF{\Ll_{t}}{2,0}{\p B_i}{\Ll_{t}}{2,\Div}{\p B_i} \norm*{\widehat{w}}_{\Ll_{t}^{2,0}{(\p B_i)}},\\
                                              &\leq C^{_{B_i}}_1 \epsilon^2 \norm*{\widehat{V}}_{\Ll_{t}^{2,\Div}{\setminus(\p B_i)}}=
                                                   C^{_{B_i}}_1 \epsilon \norm*{V}_{\Ll_{t}^{2,\Div}{(\p D_i)}},
                                    \end{aligned}\end{equation} here 
                                    $C^{_{B_i}}_1:=\norm*{\nu\times \SB{0}{i}}_\LEF{\Ll_{0}}{2}{\p B_i}{\Ll_{t}}{2,\Div}{\p B_i}
                                    \norm*{[\nu\cdotp\curl \SB{0}{i}]^{-1}}_\LEF{\Ll_{t}}{2,0}{\p B_i}{\Ll_{t}}{2,0}{\p B_i},$ 
                                    which gives \eqref{Div-part-L2estimates}. The inequality \eqref{Div-part-L2Divestimates} is an immediate consequence. Now, \eqref{tosolve} becomes $\nu\times\nabla \SD{0}{i}(v)=V-\mathfrak{V},$ with $\Div (V-\mathfrak{V})=0$, then we get successively 
                                     \begin{equation} \label{Estimate-v-V}\begin{aligned}
                                      \norm*{v}_{\Ll_{0}^{2}{(\p D_i)}}
                                      &\leq
                                        \norm*{[\nu\times\nabla \SD{0}{i}]^{-1}}_\LEF{\Ll_{t}}{2,0}{\p D_i}{\Ll_{0}}{2}{\p D_i}
                                      \norm*{V-\mathfrak{V}}_{\Ll^{2,0}_\ta{(\p D_i)}},\\
                                      &\leq \norm*{[\nu\times\nabla \SB{0}{i}]^{-1}}_\LEF{\Ll_{t}}{2,0}{\p B_i}{\Ll_{0}}{2}{\p B_i}
                                      \norm*{V-\mathfrak{V}}_{\Ll^{2,0}_\ta{(\p D_i)}}\leq C^{_{B_i}}_3   \norm*{V}_{\Ll^{2,Div}_\ta{(\p D_i)}}, 
                                      \end{aligned} \end{equation} where $C^{_{B_i}}_3:=2\norm*{[\nu\times\nabla \SB{0}{i}]^{-1}}_\LEF{\Ll_{t}}{2,0}{\p B_i}{\Ll_{0}}{2}{\p B_i}$. We put $\mathfrak{v}=\SD{0}{i}(v)$
                                      hence $\nu\times\nabla\mathfrak{v}=\nu\times\nabla\SD{0}{i}(v)$ with $\norm*{\nu\times\nabla \mathfrak{v}}_{\Ll^{2,0}_{t}(\p D_i)}
                                      =\norm*{\nu\times\nabla\SD{0}{i}(v)}_{\Ll^{2,0}_{t}(\p D_i)}=\norm*{V-\mathfrak{V}}_{\Ll^{2,0}_{t}(\p D_i)}.$ Then with \eqref{V-in-L20} in mind, we derive the estimate (\ref{nu-Grad-part-L2estimates}) 
                                      \begin{align*}
                                       \norm*{\nu\times\nabla \mathfrak{v}}_{\Ll^{2,0}_{t}(\p D_i)}
                                       &=\norm*{V-\mathfrak{V}}_{\Ll^{2,0}_{t}(\p D_i)}\leq \norm*{V}_{\Ll^{2}_{t}(\p D_i)}+\norm*{\mathfrak{V}}_{\Ll^{2}_{t}(\p D_i)},\\
                                       &\leq\norm*{V}_{\Ll^{2,\Div}_{t}(\p D_i)}+C^{_{B_i}}_1\epsilon\norm*{V}_{\Ll^{2,\Div}_{t}(\p D_i)},\\
                                       &\leq (C^{_{B_i}}_3:=(1+C^{_{B_i}}_1))\norm*{V}_{\Ll^{2,\Div}_{t}(\p D_i)}.
                                      \end{align*} Concerning \eqref{Grad-part-v-L2} we have $\norm*{\mathfrak{v}}_{\Ll^{2}(\p D_i)}=\norm*{\SD{0}{i}(v)}_{\Ll^{2}(\p D_i)},$
hence \begin{align*}
       \norm*{\mathfrak{v}}_{\Ll^{2}(\p D_i)}&=\left(\int_{\p B_i}\left(\int_{\p B_i}\frac{1}{\epsilon\abs{s-t}}\widehat{v}(t)\epsilon^2~ds(t)\right)^2 \epsilon^2~ds(s)\right)^\frac{1}{2},\\
                                            &=\epsilon^2\norm*{\SB{0}{i}(v)}_{\Ll^{2}(\p B_i)}\leq \epsilon^2\left(\norm*{\SB{0}{i}(\widehat{v})}_{\Ll^{2}(\p B_i)}+\norm*{\nabla_t\SB{0}{i}(\widehat{v})}_{\Ll^{2}(\p B_i)}\right),\\
                                            &\leq \epsilon^2 \norm*{\SB{0}{i}}_{\mathcal{L}({\Ll^{2}(\p B_i)},\Hh^1(\p B_i))}\norm*{\widehat{v}}_{\Ll^{2}(\p B_i)}=
                                            \epsilon \norm*{\SB{0}{i}}_{\mathcal{L}({\Ll^{2}(\p B_i)},\Hh^1(\p B_i))}\norm*{v}_{\Ll^{2}(\p D_i)},
      \end{align*} and with \eqref{Estimate-v-V} we get, $$\norm*{\mathfrak{v}}_{\Ll^{2}(\p D_i)}\leq \epsilon \bigl(C^{_{B_i}}_4:=\norm*{\SB{0}{i}}_{\mathcal{L}({\Ll^{2}(\p B_i)},\Hh^1(\p B_i))}C^{_{B_i}}_3)
      \norm*{V}_{\Ll^{2,Div}_\ta{(\p D_i)}},$$ to conclude we put for $l=1,2,3,4,$ $C_l=\max{i\in\{1,...,m\}C^{_{B_i}}_l}.$
\end{proof}

We have the following theorem
\begin{theorem} \label{Decomposition-density-estimate}
There exist constants $C_{B,2}$, $C_{B,1}$ and $C_e$ which depend only on $B_i'$s and independent of their number such that if
\begin{equation}                    
                    \begin{aligned}\label{NeumannSeries-Cond}
                     C_{B,1}    \abs{k}^2\epsilon < 1,\,~\,~C_{B,2}    \left(\frac{\ln{m^\frac{1}{3}}}{\delta^3}+\frac{2k m^\frac{1}{3}}{\delta^2} +\frac{m^\frac{2}{3}}{2\delta}k^2\right)\epsilon^3<1
                    \end{aligned}
\end{equation}
 then  \begin{equation}\label{Estimates-Of-The-Density-a}
       \norm*{a}_{\Ll^{2,\Div}_{t}}\leq C_e\epsilon
       \end{equation} 
Further, in view of \Cref{Key-Decomposition} each $a_i\in\Ll^{2,\Div}_{t}(\p D_i)$ can be decomposed as the sum of 
\begin{equation}\label{Decomposition-Dnsity-a}
 \begin{aligned}
  \begin{aligned}
a_i^{[1]}~\in \Ll^{p,Div}_\ta(\p D_i)\setminus\Ll^{p,0}_\ta(\p D_i), 
\end{aligned} 
\text{~and,~}\\
\begin{aligned}
 a_i^{[2]}=\nu\times\nabla u_i \in~\in \Ll^{p,0}_\ta(\p D_i),~ u\in\Hh^1(\p D_i)\setminus C;
\end{aligned}
\end{aligned} 
\end{equation}
which satisfy
\begin{align}
 \norm*{a^{[1]}}_{\Ll^{2}(\p D_i)}&\leq C_1\epsilon\norm*{a}_{\Ll^{2,\Div}_\ta(\p D_i)},\label{Decomposition-Estimate-a_1-L2}\\
 \norm*{a^{[1]}}_{\Ll^{2,\Div}_\ta(\p D_i)}&\leq C_2\norm*{a}_{\Ll^{2,\Div}_\ta(\p D_i)},\label{Decomposition-Estimate-a_1-L2Div}\\
 \norm*{a^{[2]}}_{\Ll^{2,0}_\ta(\p D_i)}&\leq C_3\norm*{a}_{\Ll^{2,\Div}_\ta(\p D_i)},\label{Decomposition-Estimate-a_2}\\
 \norm*{u_i}_{\Ll^{2}(\p D_i)} &\leq C_4 \epsilon\norm*{a}_{\Ll^{2,\Div}_\ta(\p D_i)},\label{Decomposition-Estimate-u-L2Div}
\end{align}where $(C_i)_{i=1,2,3,4.}$ are constants which depends only on the shape of the ${B}_i$'s (i.e~their Lipschitz character) and not on their number $m$.
\end{theorem}

   From \eqref{Lin-Sys-Departur-Equation}, we have successively \begin{align*}
                                            (\frac{1}{2}I+\mathcal{M}_D)(I+(\frac{1}{2}I+\mathcal{M}_D)^{-1}\mathcal{M}_N)A=-\nu\times E^I,\\
                                            (I+(\frac{1}{2}I+\mathcal{M}_D)^{-1}\mathcal{M}_N)A=-(\frac{1}{2}I+\mathcal{M}_D)^{-1}\nu\times E^I,
                                            \end{align*} and if
                                            \begin{equation}\label{Condition-if-imk=0}
                                            {\norm*{(\frac{1}{2}I+\mathcal{M}_D)^{-1}}\norm*{\mathcal{M}_N}}< 1,
                                            \end{equation}
we get $$\norm*{A}\leq \frac{\norm*{(\frac{1}{2}I+\mathcal{M}_D)^{-1}}}{1-{\norm*{(\frac{1}{2}I+\mathcal{M}_D)^{-1}}\norm*{\mathcal{M}_N}}}\norm*{\nu\times E^I},$$ where
\begin{equation*}
\begin{aligned}
\norm*{A}:=&\max_{m\in \{1,...,M\}} \norm*{a_m}_{\Ll_{t}^{2,\Div}(\partial D_m)},\\
\norm*{(\frac{1}{2}I+\mathcal{M}_D)^{-1}}:=&\max_{i\in \{1,...,m\}}\norm*{(\frac{1}{2}I+M_{ii,_D}^k)^{-1}}_\LEF{\Ll_{t}}{2,\Div}{\p D_m}{\Ll_{t}}{2,\Div}{\p D_m},\\
\norm*{\mathcal{M}_N}:=&\max_{i\in \{1,...,m\}}\sum_{\substack{ j=1\\j\neq m}}^M\norm*{M_{ij,_{D}}^k}_\LEF{\Ll_{t}}{2,\Div}{\p D_j}{\Ll_{t}}{2,\Div}{\p D_i}.
\end{aligned}
\end{equation*}

\begin{proof} We prove that under the condition \eqref{NeumannSeries-Cond}, we have 
(\ref{Estimates-Of-The-Density-a}). The properties (\ref{Decomposition-Dnsity-a})-(\ref{Decomposition-Estimate-u-L2Div}) are immediate conclusions of \Cref{Key-Decomposition} and do not rely on (\ref{NeumannSeries-Cond}).

We set $diam(B):=\max_{i\in\{1,..,m\}}diam({B}_i)$ and we suppose that $diam(B)\leq 1.$\\ 
For every $i\in\{1,...,m\}$ and $x_i=\epsilon s_i+z_i\in\dD_i,\, s_i\in \p B_i$ we write 
$\hat{a}(s_i):=a(\epsilon s_i+z_i),$ and $\breve{a}(x_i)=a(\frac{x_i-z_i}{\epsilon}),$
 and set $$M_{ii,_D}^0:=\nu\times\int_{\dD_i}\nabla\Phi_0(x_i,y)\times a(y)~ds(y).$$ 
 We have, knowing that $[\frac{1}{2}I+M_{ii,_D}^0]$ is invertible (see { Theorem 5.1 property (xix) in \cite{DMM96}}), the identity
\begin{align*}
 [\frac{1}{2}I+M_{ii,_D}^k]=[\frac{1}{2}I+M_{ii,_D}^0]\bigl(I+[\frac{1}{2}I+M_{ii,_D}^0]^{-1}[M_{ii,_D}^k-M_{ii,_D}^0]\bigr),
\end{align*} then, by inverting the operators in each side, under the condition that 
\begin{equation}\label{cdtinvmii0}
\norm*{[M_{ii,_D}^k-M_{ii,_D}^0][\frac{1}{2}I+M_{ii,_D}^0]^{-1}}\leq \norm*{[\frac{1}{2}I+M_{ii,_D}^0]^{-1}}\norm*{[M_{ii,_D}^k-M_{ii,_D}^0]}<1,
\end{equation}
 we get $[\frac{1}{2}I+M_{ii,_D}^k]^{-1}=\left(I+[\frac{1}{2}I+M_{ii,_D}^0]^{-1}[M_{ii,_D}^k-M_{ii,_D}^0]\right)^{-1}[\frac{1}{2}I+M_{ii,_D}^0]^{-1}$. As 
$$\left(I+[\frac{1}{2}I+M_{ii,_D}^0]^{-1}[M_{ii,_D}^k-M_{ii,_D}^0]\right)^{-1}=\sum_{n\geq 0}\left(-[\frac{1}{2}I+M_{ii,_D}^0]^{-1}[M_{ii,_D}^k-M_{ii,_D}^0]\right)^n$$
                we have finally \begin{equation} \label{but3}
                         \norm*{[\frac{1}{2}I+M_{ii,_D}^k]^{-1}}\leq
                         \frac{\norm*{[\frac{1}{2}I+M_{ii,_D}^0]^{-1}}}{1-\norm*{[\frac{1}{2}I+M_{ii,_D}^0]^{-1}}\norm*{[M_{ii,_D}^k-M_{ii,_D}^0]}}.
                        \end{equation}
 From here $\mathcal{L}(E):=\mathcal{L}(E,E)$ denotes the space of continuous linear operators which are defined from $E$ to $E$.                   
\bigskip

Hence the proof of (\ref{Estimates-Of-The-Density-a}), based on the condition \eqref{NeumannSeries-Cond}, is reduced to the following two estimates:

                        \begin{equation}\label{scal-norm-diag-operator}
                        \begin{aligned}
                         \norm*{[\frac{1}{2}I+M_{ii,_D}^0]^{-1}}_{\mathcal{L}\left(\Ll_{t}^{2,{_{\Div}}}({\p D_i})\right)}
                         \leq &\norm*{[\frac{1}{2}I+M_{ii,_B}^0]^{-1}}_{\mathcal{L}\left(\Ll_{t}^{2}(\p B_i)\right)}
                         \\
                              &+\norm*{[\frac{1}{2}I-(K_{ii,_B}^0)^*]^{-1}}_{\mathcal{L}\left(\Ll_{0}^{2}(\p B_i)\right)},
                        \end{aligned} 
                        \end{equation}
                        and 
                        \begin{equation}\label{scal-norm-interaction-matrix}
                        \norm*{\mathcal{M}_N}_{\mathcal{L}\left(\Ll_{t}^{2,{\Div}}(\p B_i)\right)}\leq 
                        \frac{2^6C_{_B}}{\delta}\left(\frac{\ln{(n+1)}}{\delta^2}+\frac{2k n}{\delta} +\frac{n}{2}(n+1)k^2\right)\epsilon^3.
                        \end{equation} where $n=O(m^\frac{1}{3})$, and $C_{_B}$ is a constant which depends exclusively on ${B}_i$'s. 
                        In some places of the next computations, we use the notation 
                        \begin{equation}\nonumber
                         C_0:=2^6.
                        \end{equation}
 
\bigskip

To justify (\ref{scal-norm-diag-operator}) and (\ref{scal-norm-interaction-matrix}), we need the following lemma.
 \begin{lemma}\label{lem2}
              For $x_i\in\p D_i$, $s_i\in\p B_i$, with $x_i=\epsilon s_i+z$ 
    the following scaling estimation
           \begin{equation}\label{est1}
                     \norm*{[\frac{1}{2}I+M_{ii,_D}^0]^{-1}}_{\mathcal{L}(\Ll_t^{2}(\p D_i))}
                     =
                     \norm*{[\frac{1}{2}I+M_{ii,B}^0]^{-1}}_{\mathcal{L}(\Ll_t^{2}(\p B_i))},
               \end{equation} 
                \begin{equation}\label{estiv1}
                 \norm*{[\frac{1}{2}I-K^0_{ii,_D}]^{-1}}_{\mathcal{L}(\Ll_0^{2}(\p D_i))}=
                 \norm*{[\frac{1}{2}I-K^0_{ii,_B}]^{-1}}_{\mathcal{L}(\Ll_0^{2}(\p B_i))},
                 \end{equation}

                and 
                \begin{equation}\label{estim2}
                \norm*{[M_{ii,_D}^k-M_{ii,_D}^0]}_\LEF{\Ll_\ta}{2,\Div}{\p D_i}{\Ll_\ta}{2,\Div}{\p D_i}\leq 2C_{_B}    \abs{\p B} \abs{k}^2\epsilon  \norm*{b}_{\Ll^2(\p D_i)}.
                \end{equation}
 \end{lemma}
\begin{proof} 
For \eqref{est1} and \eqref{estiv1}, we, first have
\begin{align*}
M_{ii,_D}^0(a)(x_i)&=\nu\times\int_{\dD_i}(4\pi)^2\Phi_0^3(x_i,y)(x_i-y)\times a(y)~ds(y),\\
                &=\nu\times\int_{\partial \mathcal{B}_i}\frac{(4\pi)^2}{\epsilon^3}\Phi_0^3(s_i,t)\epsilon(s_i-t)\times a(\epsilon t+z_i)\epsilon^2ds(t),\\
                &=\nu\times\int_{\partial \mathcal{B}_i}(4\pi)^2\Phi_0^3(s_i,t)(s_i-t)\times a(\epsilon t+z_i)~ds(t)=\{M_{ii,_B}^0{\hat{a}} \}^\widecheck{\,},
\end{align*} which leads to $\widehat{[\frac{1}{2}I+M_{ii,_D}^0]^{-1}(b)}=[\frac{1}{2}I+M_{ii,_B}^0]^{-1} \widehat{b}$. 
With this in mind, considering \eqref{L2div-scaling} we get 
                          \begin{align*}
  \norm*{[\frac{1}{2}I+M_{ii,_D}^0]^{-1}}_{\mathcal{L}(\Ll_t^{2}(\p D_i))}&
  =\sup_{{(b\neq0) \in \Ll^{2}_{t}(\p D_i)}}
  \frac{\norm*{[\frac{1}{2}I+M_{ii,_D}^0]^{-1}(b)}_{\Ll^{2}_{t}(\p D_i)}}{\norm*{b}_{\Ll^{2}_{t}(\p D_i)}},\\
  &=\sup_{{(b\neq0) \in \Ll^{2}_{t}(\p D_i)}}
  \frac{\epsilon\norm*{\widehat{[\frac{1}{2}I+M_{ii,_D}^0]^{-1}(b)}}_{\Ll^{2}_{t}(\p B_i)}}{\epsilon\norm*{\widehat{b}}_{\Ll^{2}_{t}(\p B_i)}},\\
  &=\sup_{{(\widehat{b}\neq0) \in \Ll^{2}_{t}(\p B_i)}}
  \frac{\norm*{{[\frac{1}{2}I+M_{ii,_D}^0]^{-1}(\widehat{b})}\,}_{\Ll^{2}_{t}(\p B_i)}}
  {\norm*{\widehat{b}}_{\Ll^{2}_{t}(\p B_i)}}.
                          \end{align*} 
We obtain \eqref{estiv1} in the same way. For \eqref{estim2}, by Mean-value-theorem, we have 
\begin{align}\label{Mean-Val-Thm}
(\Phi_k(x,y)-\Phi_0(x,y))=\frac{1}{4\pi}\frac{e^{ik\abs{x-y}}}{\abs{x-y}}-\frac{1}{4\pi}\frac{e^{0\abs{x-y}}}{\abs{x-y}}=
\int_{0}^{1}ik\frac{1}{4\pi}{e^{ikl\abs{x-y}}}dl.
\end{align} 
Taking the gradient gives  \begin{align}\label{meanvthe}
                            \nabla_x(\Phi_k(x,y)-\Phi_0(x,y))= \frac{(ik)^2}{4\pi}\int_{0}^{1}\frac{le^{ikl\abs{x-y}}}{\abs{x-y}}(x-y)dl,
                           \end{align} thus, being $\Im k=0$, $\abs{\nabla_x(\Phi_k(x,y)-\Phi_0(x,y))\times b(y)}\leq \frac{\abs{k}^2}{4\pi}\abs{ b(y)}\leq \frac{\abs{k}^2}{4\pi}\abs{ b(y)},$ and 
                           \begin{align}\label{forphi}\abs{(\Phi_k(x,y)-\Phi_0(x,y))b}\leq \frac{\abs{k}}{4\pi}\abs{ b(y)},\end{align}
                           then it comes \begin{align*}
                                    \abs{[M_{ii,_D}^k-M_{ii,_D}^0](b)}=\abs{\int_{\dD_i}\nabla (\Phi_k-\Phi_0)(x_i,y)\times b(y)~ds(y)}
                                    &\leq \int_{\dD_i} \frac{\abs{k}^2}{4\pi}\abs{ b(y)}~ds(y),\\
                                    &\leq  \frac{\abs{k}^2\abs{\p B_i}^\frac{1}{2}\epsilon}{4\pi}
                                    \left(\int_{\dD_i}\abs{ b(y)}^2ds(y)\right)^\frac{1}{2}.
                                    \end{align*} 

Taking the norm in both sides,\begin{align}\label{Norm-M_k-M_0}
                                                                               \norm*{[M_{ii,_D}^k-M_{ii,_D}^0](b)}_{\Ll^2(\p D_i)}^2\leq
                                                                               \left(\frac{\abs{k}^2(\abs{\p B_i}\abs{\p B_i})^\frac{1}{2}\epsilon^2}{4\pi}
                                                                               \norm*{b}_{\Ll^2(\p D_i)}\right)^2.
                                                                               \end{align} 
                                                                               We have the identities, $\Div [M_{ii,_D}^k-M_{ii,_D}^0](b)(x)=-[k^2\nu \cdotp\SD{k}{i}](b)-[(\KD{k}{i}-\KD{0}{i})^*](\Div b)$. To estimate $[(\KD{k}{i}-\KD{0}{i})^*](\Div b)$, we can reproduce the same steps as we did to obtain \eqref{Norm-M_k-M_0}. We obtain
                                                                               \begin{align}\label{Fnormcompact2}
                                                                               \norm*{[(\KD{k}{i}-\KD{0}{i})^*]}_{\Ll^2_0(\p D_i)}^2
                                                                               \leq\left(\frac{\abs{k}^2(\abs{\p {B}_j}\abs{\p B_i})^\frac{1}{2}\epsilon^2}{4\pi}
                                                                               \norm*{\Div b}_{\Ll^2_0(\p D_i)}\right).
                                                                               \end{align} 
                                                                               Using \eqref{forphi}, we deduce that
                                                                               \begin{align*}
                                                                               \abs{[k^2\nu \cdotp\SD{k}{i}](b)}
                                                                               &=
                                                                               \abs{[k^2\nu \cdotp(\SD{k}{i}-\SD{0}{i})](b)+
                                                                               [k^2\nu \cdotp\SD{0}{i}](b)},\\
                                                                               &\leq \frac{\abs{k}^3\abs{\p B_i}\epsilon}{4\pi}
                                                                               \norm*{b}_{\Ll^2(\p D_i)}+
                                                                               \abs{[k^2\epsilon \nu \cdotp\SB{0}{i}](\widehat{b})},
                                                                               \end{align*} 
                                                                               
and taking the norm, we get \footnote{Notice that $\epsilon\norm*{\widehat{b}}_{\Ll^2(\p B_i)}=\norm*{b}_{\Ll^2(\p D_i)}$. } 
                                                                                                     \begin{equation} \label{knusdii}   
                                                                                                             \begin{aligned}
                                                                                                              \norm*{[k^2\nu \cdotp\SD{k}{i}]b}_{\Ll^2(\p D_i)}\hspace{-1mm}
                                                                                                              &\leq
                                                                                                                    \frac{\abs{k}^3(\abs{\p {B}_j}\abs{\p B_i})^\frac{1}{2}\epsilon^2}{4\pi}
                                                                                                                    \norm*{b}_{\Ll^2(\p D_i)}\hspace{-1mm}+ 
                                                                                                                    \abs{k}^2\epsilon^2(\abs{\p {B}_j})^\frac{1}{2} \norm*{\SB{0}{i}}\norm*{\widehat{b}}_{\Ll^2(\p B_i)},\\
                                                                                                               &\leq \frac{\abs{k}^3(\abs{\p {B}_j}\abs{\p B_i})^\frac{1}{2}\epsilon^2}{4\pi}
                                                                                                                    \norm*{b}_{\Ll^2(\p D_i)}+ 
                                                                                                                    \abs{k}^2\epsilon(\abs{\p {B}_j})^\frac{1}{2} \norm*{\SB{0}{i}}\norm*{b}_{\Ll^2(\p D_i)},\\  
                                                                                                               &\leq C_{_B}    \abs{\p B} \abs{k}^2\epsilon  \norm*{b}_{\Ll^2(\p D_i)},
                                                                                                              \end{aligned} 
                                                                                                      \end{equation} where $\abs{\p B}=\max_i\abs{\p B} $ and $C_{_B}$ is the maximum of the constants that appear 
                                                                                                      in the inequalities \eqref{Norm-M_k-M_0} and \eqref{Fnormcompact2}. Hence
                                                                                                              \begin{align}
                                                                                                               \norm*{[M_{ii,_D}^k-M_{ii,_D}^0](b)}_{\Ll^{2,\Div}_\ta(\p D_i)}\leq 
                                                                                                               2C_{_B}    \abs{\p B} \abs{k}^2\epsilon  \norm*{b}_{\Ll^2(\p D_i)}.
                                                                                                              \end{align}\end{proof}


To prove (\ref{scal-norm-diag-operator}), let us recall that we have
\begin{align}\label{Norm-Mii-Diag-oneachB}
 \norm*{[\frac{1}{2}I+M_{ii,_D}^0]^{-1}}&=
 \sup_{b\neq0}\left(\frac{\norm*{[\frac{1}{2}I+M_{ii,_D}^0]^{-1}b}_{\Ll^2{(\p D_i)}}^2+
 \norm*{\Div[\frac{1}{2}I+M_{ii,_D}^0]^{-1}b}_{\Ll^2{(\p D_i})}^2 }{\norm*{b}_{\Ll^2{(\p D_i)}}^2+\norm*{\Div b}^2_{\Ll^2{(\p D_i)}}}\right)^\frac{1}{2}
 \end{align} 
 Considering the fact that\footnote{It suffices to write, for $b,c\in\Ll^{2,\Div}$ such that $b=[\frac{1}{2}I+M_{ii,_D}^0]^{-1}c,$ 
                                                $\Div c=\Div [\frac{1}{2}I+M_{ii,_D}^0]b=[\frac{1}{2}I-(K_{ii,_D}^0)^*]\Div b$ and inverting 
                                                $[\frac{1}{2}I-(K_{ii,_D}^0)^*]$ to get $\Div b=[\frac{1}{2}I-(K_{ii,_D}^0)^*]^{-1}\Div c$.}
$\Div [\frac{1}{2}I+M_{ii,_D}^0]^{-1}= [\frac{1}{2}I-(K_{ii,_D}^0)^*]^{-1}\Div,$ then \eqref{Norm-Mii-Diag-oneachB} gives 
                                                  \begin{align*}
                                                  \norm*{[\frac{1}{2}I+M_{ii,_D}^0]^{-1}}=&
                                                  \sup_{b\neq0}\left(\frac{\norm*{[\frac{1}{2}I+M_{ii,_D}^0]^{-1}b}_{\Ll^2{(\p D_i)}}^2+
                                                  \norm*{[\frac{1}{2}I-(K_{ii,_D}^0)^*]^{-1}\Div b}_{\Ll^2{(\p D_i})}^2 }{\norm*{b}_{\Ll^2{(\p D_i)}}^2+
                                                  \norm*{\Div b}^2_{\Ll^2{(\p D_i)}}}\right)^\frac{1}{2},\\
                                                  \leq&\sup_{b\neq0}\biggl(\frac{\norm*{[\frac{1}{2}I+M_{ii,_D}^0]^{-1}b}_{\Ll^2{(\p D_i)}}^2}
                                                  {\norm*{b}_{\Ll^2{(\p D_i)}}^2+
                                                  \norm*{\Div b}^2_{\Ll^2{(\p D_i)}}}\biggr)^\frac{1}{2}\\
                                                  &+\sup_{b\neq0}\biggl(\frac{\norm*{[\frac{1}{2}I-(K_{ii,_D}^0)^*]^{-1}\Div b}_{\Ll^2{(\p D_i})}^2 }
                                                  {\norm*{b}_{\Ll^2{(\p D_i)}}^2+\norm*{\Div b}^2_{\Ll^2{(\p D_i)}}}\biggr)^\frac{1}{2},
                                                  \end{align*} 
                                                  where the last inequality is due to the fact that 
                                                  $\alpha^2+\beta^2\leq \alpha^2+\beta^2+2\alpha\beta$ for $\alpha,\beta\geq0$.
                                                  
Finally, as $\alpha^2\leq \beta^2+\alpha^2$, we get \footnote{With $\Ll^2$-norm.}  
                                    \begin{align*}
                                    \norm*{[\frac{1}{2}I+M_{ii,_D}^0]^{-1}}\leq&
                                    \sup_{b\neq0}\left(\frac{\norm*{[\frac{1}{2}I+M_{ii,_D}^0]^{-1}b}^2}
                                    {\norm*{b}^2}\right)^\frac{1}{2}+
                                    \sup_{b\neq0}\left(\frac{\norm*{[\frac{1}{2}I-(K_{ii,_D}^0)^*]^{-1}\Div b}^2 }
                                    {\norm*{\Div b}^2}\right)^\frac{1}{2}.
                                    \end{align*}
                                    
To prove (\ref{scal-norm-interaction-matrix}), we will need the following lemma
\begin{lemma}\label{Inter-Lemma}
 For every $i, j\in\{1,...,m\}$, under the condition that $\epsilon\leq \delta<1$ we have 
 \begin{equation}
   \norm*{M_{ij,_{D}}^k}_\LEF{\Ll_{t}}{2,Div}{\p B}{\Ll_{t}}{2,Div}{\p B}\leq 
   \frac{4(\abs{\p B_i}\abs{\p {B}_j})^\frac{1}{2}}{\pi\delta_{i,j}}
   \bigl(\frac{1}{\delta_{i,j}}+\abs{k}\bigr)^2 \epsilon^3.
 \end{equation}

\end{lemma}
\begin{proof} For $i\neq j$, $x_i\in \p D_i$
we have, recalling \eqref{Div-Identity-transform}, that 
\begin{equation}\label{Mij-trasfomation}
        \begin{aligned}
         \int_{\p D_j} \Phi_k(x_i,y) b(y)~ds(y)=&\int_{\p D_j}(\nabla_y(y-z_j)) (\Phi_k(x_i,y) b(y)~ds(y),\\
                                              =&\int_{\p D_j}(y-z_j) \Div\left(\Phi_k(x_i,y) b(y)\right)~ds(y),\\
                                              =&-\int_{\p D_j}(y-z_j) \Phi_k(x_i,y)\Div b(y)~ds(y)\\
                                               &-\int_{\p D_j}(y-z_j) \nabla_y\Phi_k(x_i,y)\cdot b(y)~ds(y).
        \end{aligned} 
 \end{equation} 
 Being $-M_{ij,_{D}}^k(b)= -\nu \times\nabla_x\times\int_{\p D_j} \Phi_k(x_i,y)\times b(y)~ds(y)$, it comes from \eqref{Mij-trasfomation}
\begin{equation}\label{Transformation}
           \begin{aligned}
           -M_{ij,_{D}}^k(b)=&\nu \times\nabla_x\times\int_{\p D_j}(y-z_j) \Phi_k(x_i,y)\Div b(y)~ds(y)~\\
                        &+\nu \times\nabla_x\times\int_{\p D_j}(y-z_j) \nabla_y\Phi_k(x_i,y)\cdot b(y))~ds(y),\\
                     =&\nu \times\int_{\p D_j}(y-z_j)\times \nabla_x\Phi_k(x_i,y)\Div b(y)~ds(y)~\\
                        &+\nu \times\int_{\p D_j}(y-z_j)\times \nabla_x \nabla_y \Phi_k(x_i,y)\cdot b(y))~ds(y).   
           \end{aligned} 
\end{equation}
As \footnote{ Notice that $\left((x-y)(x-y)^T\right)b=\left(b\cdotp(x-y)\right)(x-y)$ and $(x-y)\times (x-y)=0$.} 
 \begin{align*}
 -\nabla_x\nabla_y\Phi_k(x,y)=(4\pi)^4\Phi_k(x,y)&\Phi_0^2\left\{(\Phi_0-ik)^2+(\Phi_0-ik)\Phi_0+\Phi_0^2\right\}(x,y)\left((x-y)(x-y)^T\right)\\
 &+(4\pi)^2\Phi_k\Phi_0(\Phi_0-ik)(x,y) ~I,
\end{align*} where $(x-y)^T$ stands for the transpose vector of $(x-y)$, we  get\footnote{As $e^{-\Im k~\delta_{i,j}}\leq 1.$}          
          \begin{align}\label{doublgrad}
           \abs{\nabla_x\nabla_y\Phi_k(x,y)}\leq \frac{3}{4\pi}\frac{1}{\delta_{i,j}}\bigl(\frac{1}{\delta_{i,j}}+\abs{k}\bigr)^2,
           \end{align} and \begin{align}
                             \abs{\nabla_y\Phi_k(x,y)}\leq \frac{1}{4\pi}\frac{1}{\delta_{i,j}}\left(\frac{1}{\delta_{i,j}}
                             +\abs{k}\right).
                             \end{align}
Hence, in view of \eqref{Transformation}, using Holder's inequality,
           \begin{align*}
           \abs{M_{ij,_{D}}^k(b)(x_i)}\leq& \abs{\nabla_x\Phi_k(x_i,y)} \norm*{(y-z_j)}_{\Ll^2(\p D_i)}\norm*{\Div b}_{\Ll^2(\p D_i)}\\
                        &+\abs{\nabla_y\nabla_x \Phi_k(x_i,y)}\norm*{(y-z_j)}_{\Ll^2(\p D_i)}  \norm*{b(y)}_{\Ll^2(\p D_i)},\\
                                \leq& \frac{2~\abs{\p {B}_j}^\frac{1}{2}}{\pi \delta_{i,j}}
                                \bigl(\frac{1}{\delta_{i,j}}+\abs{k}\bigr)^2\epsilon^2 \norm*{b(y)}_{\Ll^{2,\Div}_\ta(\p D_i)},
           \end{align*} from which it follows that               
                                    \begin{equation}\label{First-part-EstimationMij}
                                     \norm*{M_{ij,_{D}}^k(b)}_{\Ll^2(\p D_i)}
                                     \leq \frac{2(\abs{\p B_i}\abs{\p {B}_j})^\frac{1}{2}}{\pi\delta_{i,j}}
                                     \bigl(\frac{1}{\delta_{i,j}}+\abs{k}\bigr)^2 \epsilon^3
                                     \norm*{b}_{\Ll^{2,\Div}_\ta(\p D_i)}.
                                     \end{equation} 
Taking the surface divergence of $M_{ij,_{D}}^k(a)$ we have
                                            \begin{align*}
                                            \Div M_{ij,_{D}}^k(b)(x_i)=-\int_{\p D_j}\frac{\p \Phi_k(x_i,y)}{\p \nu_x}\Div b(y)~ds(y)-
                                            k^2\nu\cdot\int_{\p D_j}\Phi_k(x_i,y)b(y)~ds(y)
                                            \end{align*} and using \eqref{Mij-trasfomation} 
                                            \begin{align*}
                                            \abs{\Div M_{ij,_{D}}^k(b)(x_i)}&\leq\abs{\int_{\p D_j}\frac{\p \Phi_k(x_i,y)}{\p \nu_x}\Div b(y)~ds(y)}+
                                            \abs{\int_{\p D_j}(y-z_j)\nabla\Phi_k(x_i,y)\cdotp b(y)~ds(y)}\\
                                            &\hspace{3cm}+\abs{\int_{\p D_j}(y-z_j)\Phi_k(x_i,y)\Div b(y)~ds(y)}.
                                            \end{align*}
Using the Mean-value-theorem, we get\footnote{Notice that $$\int_{\p D_i}\nabla \Phi_k(x_i,z_j)\Div b~ds(y)
                                                          =\nabla \Phi_k(x_i,z_j)\int_{\p D_i}\Div b~ds(y)=0$$.} successively
                                 \begin{align*}
                                 \abs{\int_{\p D_j}\frac{\p \Phi_k(x_i,y)}{\p \nu_x}\Div b(y)~ds(y)}&=
                                 \abs{\int_{\p D_j}\frac{\p \left(\Phi_k(x_i,y)-\Phi_k(x_i,z_j)\right)}{\p \nu_x}\Div b(y)~ds(y)},\\
                                 &\leq 
                                 \frac{\abs{\p {B}_j}^\frac{1}{2}}{\pi\delta_{i,j}}\left(\frac{1}{\delta_{i,j}}+
                                 \abs{k}\right)^2 \epsilon^2\norm*{\Div b}_{\Ll^2(\p D_j)};
                                 \end{align*} and, 
                                 \begin{align*}
                                  \abs{\int_{\p D_j}(y-z_j)\nabla\Phi_k(x_i,y)\cdotp b(y)~ds(y)}&\leq \epsilon^2\abs{\p {B}_j}^\frac{1}{2}
                                                           \frac{1}{4\pi\delta_{i,j}}(\frac{1}{\delta_{i,j}}+\abs{k})\norm*{b}_{\Ll^2(\p D_j)},\\
                                  \abs{\int_{\p D_j}(y-z_j)\Phi_k(x_i,y)\Div b(y)~ds(y)}&\leq \epsilon^2\abs{\p {B}_j}^\frac{1}{2}
                                  \frac{1}{4\pi\delta_{i,j}}\norm*{\Div b}_{\Ll^2(\p D_j)}.
                                 \end{align*} 
The sum of the three last inequalities, gives us the estimates
                                                    \begin{align}\label{Div-Mij}
                                                    \abs{\Div M_{ij,_{D}}^k(b)(x_i)}&\leq 
                                                    \frac{3\abs{\p {B}_j}^\frac{1}{2}}{\pi\delta_{i,j}}\left(\frac{1}{\delta_{i,j}}+
                                                    \abs{k}\right)^2 \epsilon^2\left(\norm*{b}_{\Ll^{2,\Div}_\ta(\p D_i)}+\norm*{\Div b}_{\Ll^2(\p D_j)}\right),
                                                    \end{align} 
and then
                                                    \begin{align}\label{frstpartestim1}
                                                    \norm*{\Div M_{ij,_{D}}^k(b)}_{\Ll^2(\p D_i)}&\leq 
                                                    \frac{3(\abs{\p B_i}\abs{\p {B}_j})^\frac{1}{2}}{\pi\delta_{i,j}}\left(\frac{1}{\delta_{i,j}}+
                                                    \abs{k}\right)^2 \epsilon^3\norm*{b}_{\Ll^{2,\Div}_\ta(\p D_i)}.
                                                    \end{align} 
We conclude by combining \eqref{First-part-EstimationMij} and \eqref{frstpartestim1}.
                                                    \end{proof}

Based on \Cref{Inter-Lemma}, let us show the proof of (\ref{scal-norm-interaction-matrix}).   
 Draw $l$ spheres $(\Sp_{l{\delta}}(z_i))_{\{l=1,2,...,n\}}$ centered at $z_i$ with radius $l\delta$ , where $n$ will be determined later, let 
 $R_l=\Sp_{l+1}-\Sp_l$, and $R_0=\Sp_1$ the volume of each $R_l$ is given by \begin{align*}
                                                                             Vol(R_l)&=\frac{4\pi \left((l+1)\delta\right)^3}{3}-\frac{4\pi (l\delta)^3}{3}=\frac{4\pi (3l^2+3l+1)\delta^3}{3}.                                                                             
                                                                             \end{align*} 
  For $j\neq i$, we consider the spheres  $\Sp_{\frac{1}{2}\delta}(z_j)$. We claim that for $j_2\neq j_1,$ $int(\Sp_{\frac{1}{2}\delta}(z_{j_1}))\cap int(\Sp_{\frac{1}{2}\delta}(z_{j_2}))=\emptyset$, 
  where $int(\Sp_{\frac{1}{2}\delta})$ stands for the interior of $\Sp_{\frac{1}{2}\delta}$. Indeed, if $t$ was in the intersection, we would have $d_{j_1,j_2}
  =\min_{\{x\in D_{j_1}, y\in D_{j_2}\}}  {{d(x,y)}}\leq d(z_{j_1},z_{j_2})\leq d(z_{j_1},t)+d(t,z_{j_2})<\frac{\delta}{2}+\frac{\delta}{2}=\delta,$ which contradicts the fact that $\delta$ is the minimum distance.
Hence, each domain $D_{j_1}$ located in $z_{j_1}$ occupying the volume of $\Sp_{\frac{1}{2}\delta}(z_{j_2})$  which is not shared with 
  another $D_{j_i}$. Then the maximum number of $D_j$'s that could occupy each $R_l$, for $l=1,...,m$, corresponds to the maximum number of spheres 
  $\Sp_{\frac{1}{2}\delta}(z_{j})$ that could fit in the closure of $R_l$. Considering the case where $z_j$ is on $\p R_l$, only the half of the ball is in $R_l$, see the figure.\medskip

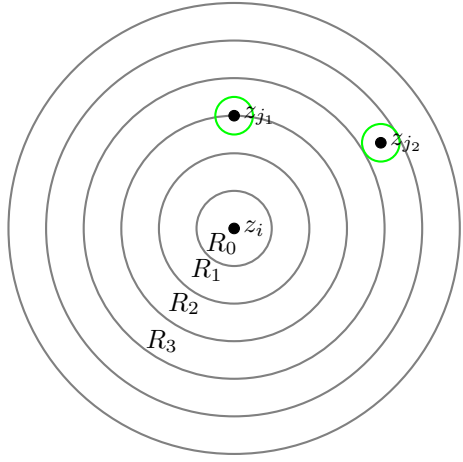
\begin{figure}\label{fig2}
\begin{center}
\begin{tikzpicture}
\draw[gray, thick] (2,0) circle (3);
\draw[gray, thick] (2,0) circle (2.5);
\draw[gray, thick] (2,0) circle (2);
\draw[gray, thick] (2,0) circle (1.5);
\draw[gray, thick] (2,0) circle (1);
\draw[gray, thick] (2,0) circle (0.5);
\draw[green, thick](2,1.5) circle (0.25);
\draw[green, thick](3.95,1.14) circle (0.25);
\filldraw[black] (2,0) circle (2pt) node[anchor=west] {$z_i$};
\filldraw[black] (2,1.5) circle (2pt) node[anchor=west] {$z_{j_1}$};
\filldraw[black] (3.95,1.14) circle (2pt) node[anchor=west] {$z_{j_2}$};
\filldraw[black] (1.5,-0.2) node[anchor=west] {$R_0$};
\filldraw[black] (1.3,-0.55) node[anchor=west] {$R_1$};
\filldraw[black] (1,-1) node[anchor=west] {$R_2$};
\filldraw[black] (0.7,-1.5) node[anchor=west] {$R_3$};
\end{tikzpicture}
\end{center}
\caption{Possible configuration for the scatterers.}
\end{figure}
If $m_l$ corresponds to the maximum amount of locations $z_j$ that are in the closure of $R_l$ then $m_l\leq {Vol(R_l)}/(\frac{Vol(\Sp_{\frac{1}{2}\delta})}{2})=2^4(3l^2+3l+1)\leq 2^6 l^2,$ 
whenever $l\geq 4$, then, being $\sum_{l=1}^n m_l=m\leq \sum_{l=1}^n 2^4(3l^2+3l+1)=2^4 n(n^2+3n+3),$ we get $n=O(m^\frac{1}{3}).$
Now, considering \cref{Inter-Lemma}, we have \begin{equation*}
                                      \norm*{M_{ij,_{D}}^k}_\LEF{\Ll_{t}}{2,Div}{\p B}{\Ll_{t}}{2,Div}{\p B}
                                      \leq \frac{4(\abs{\p B_i}\abs{\p {B}_j})^\frac{1}{2}}{\pi\delta_{i,j}} \bigl(\frac{1}{\delta_{i,j}}+\abs{k}\bigr)^2 \epsilon^3
                                                      \leq \frac{C_{i,j}}{\delta_{i,j}} \bigl(\frac{1}{\delta_{i,j}}+\abs{k}\bigr)^2\epsilon^3,
                                                      \end{equation*}hence, for $C_{_B}=\max_{i,j\in\{1,...,m\}}C_{i,j}$, we have \begin{equation}
                                                      \sum_{(j\neq i)\geq1}^m\norm*{M_{ij,_{D}}^k}_\LEF{\Ll_{t}}{2,Div}{\p B}{\Ll_{t}}{2,Div}{\p B}\leq
                                                      C_{_B}\sum_{l=1}^n\sum_{z_j\in R_l} \frac{1}{\delta_{i,j}}\bigl(\frac{1}{\delta_{i,j}}+\abs{k}\bigr)^2 \epsilon^3.
                                                      \end{equation} 
Since for every $z_j\in \bar{R_l}$,  $l\delta\leq \delta_{i,j}$
            we get\begin{equation}\label{Exhaustiv-counting}\begin{aligned}
                   \sum_{(j\neq i)\geq1}^m\norm*{M_{ij,_{D}}^k}_\LEF{\Ll_{t}}{2,Div}{\p B}{\Ll_{t}}{2,Div}{\p B}&\leq C_{_B}\sum_{l=1}^n\sum_{z_j\in R_l} \frac{1}{l\delta}\left(\frac{1}{l\delta}+k\right)^2 \epsilon^3,\\
                   &\leq C_{_B}\sum_{l=1}^n 2^4(3l^2+3l+1) \frac{1}{l\delta}\left(\frac{1}{l\delta}+k\right)^2 \epsilon^3,\\
                   &\leq 2^6C_{_B}\sum_{l=1}^n  l^2 \frac{1}{l\delta}\left(\frac{1}{l\delta}+k\right)^2 \epsilon^3,\\
                   &\leq \frac{2^6C_{_B}}{\delta}\left(\frac{\ln{(n+1)}}{\delta^2}+\frac{2k n}{\delta} +\frac{n}{2}(n+1)k^2\right)\epsilon^3.
                  \end{aligned}\end{equation}
 
\end{proof}

Considering \eqref{est1} and \eqref{estim2}, the condition \eqref{cdtinvmii0} is acquired if
                                        \begin{equation}\label{condition1}
                                        \frac{\abs{\p B_i}e^{C_{_B}}}{4\pi\, diam(B)^2}(k\epsilon)^2
                                        \norm*{[\frac{1}{2}I+M_{ii,_B}^0]^{-1}}_\LEF{\Ll_{t}}{2}{\p B_i}{\Ll_{t}}{2}{\p B_i}
                                         <1,
                                         \end{equation}
If we set $C_{B_i}=\frac{\abs{\p B_i}}{4\pi\, diam(B)^2}\norm*{[\frac{1}{2}I+M_{ii,_B}^0]^{-1}}_\LEF{\Ll_{t}}{2}{\p B_i}{\Ll_{t}}{2}{\p B_i}$ we get from \eqref{but3} 
                                         \begin{equation}
                                          \norm*{[\frac{1}{2}I+M_{ii,_D}^k]^{-1}}\leq C_{i,\epsilon}:=
 \frac{\norm*{[\frac{1}{2}I+M_{ii,_B}^0]^{-1}}_\LEF{\Ll_{t}}{2}{\p B_i}{\Ll_{t}}{2}{\p B_i}}{1-C_{B_i} \left(e^{k\epsilon}-1\right)k\epsilon},
                                         \end{equation} so we find $\norm*{(\frac{1}{2}I+\mathcal{M}_D)^{-1}}\leq C_\epsilon:= \max_{i\in\{1,...,m\}} C_{i,\epsilon},$ under the condition \eqref{condition1}.
 

 \section{Fields approximation and the linear algebraic systems}\label{Section-first-approximation}
 Based on the representation (\ref{cas1}), the expression of the  far field pattern is given by
                                    \begin{align*}
                                    E^\infty(\tau)&=\frac{ik}{4\pi}\tau\times\int_{\p D} a(y)e^{-ik\tau.y}ds(y),
                                    \end{align*} where $\tau=(x/\abs{x}) \in\mathbb{S}^2$. We put                                                                          
  \begin{align*}
  \mathcal{A}_i:=\int_{\p D_i}a_i^{[1]}ds,~\,~\, \mathcal{B}_i:=\int_{\p D_i}\nu u_ids,
  \end{align*} where $a_i^{[1]}$ and $u_i$ are defined in \Cref{Decomposition-density-estimate}, with the notation \eqref{Varepsilon-Generic-Function} 
  repeating the same calculations as in \eqref{Exhaustiv-counting}, we derive the estimate  
\begin{equation} \label{Error-Exhaustiv-Expression} \begin{aligned}
    \sum_{\underset{j\neq i}{j\geq1}}^{m}\frac{1}{\delta_{i,j}}\Bigl((\frac{1}{\delta_{i,j}}&+\abs{k})^3+(\frac{1}{\delta_{i,j}}+\abs{k})^2+(\frac{1}{\delta_{i,j}}+\abs{k})\Bigr)\\
       &=O\Bigl(\frac{1}{\delta^4}+\frac{(\abs{k}+1)~\ln(m^\frac{1}{3})}{\delta^3}+\frac{(\abs{k}+1)^2~m^\frac{1}{3}}{\delta^2}+\frac{(\abs{k}+1)^3~m^\frac{2}{3}}{\delta}\Bigr)
       &=:\Varepsilon_{k,\delta,m}.
  \end{aligned} \end{equation}
\begin{proposition}\label{Field-Approximation-First-Step}
        For $\Im k=0$, the far field pattern can be approximated by
                             \begin{align}\label{Far-Field-Approximation}
                              E^\infty(\tau)=\frac{ik}{4\pi}
                             \sum_{i=1}^m e^{-ik\tau.z_i}&\tau\times 
                             \left\{\mathcal{A}_i-ik\tau\times \mathcal{B}_i\right\} +O\left((\abs{k}^3+\abs{k}^2)~ m\epsilon^4\right).
                             \end{align}
         For $\Im k\geq 0$, and for every $x \in \R^3 \setminus\cup_{i=1}^m \overline{D_i},$ such that $d_x:=d(x,\cup_{i=1}^m \overline{D_i})\geq \delta $,
       the scattered electric field has the following expansion,\begin{equation}\label{Near-Field-Approximation}
                                                                 \begin{aligned}
                                                                 E^\s(x)=&\sum_{i=1}^m  \Bigl(\nabla_x\Phi_k(x,z_i)\times\mathcal{A}_i+  \nabla_y\times \nabla_x\times (\Phi_k(x,z_i) \mathcal{B}_i)\Bigr)
                                                                          +O\bigl(\frac{\epsilon^4}{\delta^4}+\Varepsilon_{k,\delta,m}~\epsilon^4 \bigr).
                                                                 \end{aligned}
                                                                 \end{equation} 
The elements $(\mathcal{A}_i)_{i=1}^m$ and $(\mathcal{B}_i)_{i=1}^m$ are solutions of the following linear algebraic system 
                               \begin{equation}\label{SYS1}
                                            \begin{aligned}
                                             \mathcal{A}_i
                                               =-\Polt &\sum_{(j\neq i)\geq1}^m \left(\Pi_k(z_i,z_j)\mathcal{A}_j-k^2\nabla\Phi_k(z_i,z_j)\times\mathcal{B}_j\right) -\Polt\curl E^\n(z_i)\\
                                               &+O\biggl(\frac{\epsilon^7}{\delta^4}+ \abs{k}\Varepsilon_{k,\delta,m}  \epsilon^7+\abs{k}^2\epsilon^4\biggr).
                                             \end{aligned}
                                 \end{equation} 
                                  \begin{equation} \label{SYS2}       
                                        \begin{aligned}
                                       \mathcal{B}_i=\Vmt&\sum_{(j\neq i)\geq1}^m\left(~ -\nabla_x\Phi_k(z_i,z_j)\times \mathcal{A}_j~+
                                      ~ \Pi_k(z_i,z_j)\mathcal{B}_j\right)-\Vmt E^\n(z_i) \\
                                  &+O\biggl(\frac{\epsilon^7}{\delta^4}+ \Varepsilon_{k,\delta,m}  \epsilon^7+(1+\abs{k})\epsilon^4\biggr).
                                        \end{aligned}                                                                      
                                \end{equation}  
\end{proposition} 

\subsection{Justification of \eqref{Far-Field-Approximation} and \eqref{Near-Field-Approximation}}
\begin{lemma}\label{Farfield-lemma-appro} 
 For $\Im k=0$, the far field pattern can be approximated by 
                  \begin{align}\label{Farfield-lemma-appro-estimate-1}
                  E^\infty(\tau)=\frac{ik}{4\pi}
                  \sum_{i=1}^m e^{-ik\tau.z_i}&\tau\times 
                  \left\{\int_{\p D_i}a_i~ds-\int_{\p D_i}\left(ik\tau.(y-z_i)\right) a_i(y)~ds(y)\right\} 
                  \end{align}
                   with an error estimate given by $O\left( e^{\abs{k}\epsilon}~  m~ (\abs{k}^3)\epsilon^4\right)$.
                   Precisely, in view of the decomposition \eqref{Decomposition-Dnsity-a} of \Cref{Decomposition-density-estimate} and \eqref{condition1} the far field  admits the following expansion
                  \begin{equation}\label{Farfield-lemma-appro-estimate-2}
                  \begin{aligned}
                   E^\infty(\tau)=\frac{ik}{4\pi}\sum_{i=1}^m e^{-ik\tau.z_i}\tau\times \left\{\mathcal{A}_i-ik\tau\times\mathcal{B}_i\right\}+
                                 O\left( (\abs{k}^3+\abs{k}^2)~ m\epsilon^4\right).
                  \end{aligned}
                  \end{equation}
\end{lemma}
\begin{proof} To prove (\ref{Farfield-lemma-appro-estimate-1}), we write
       \begin{align*}
       E^\infty(\tau)=\frac{ik}{4\pi}\sum_{i=1}^m e^{-ik\tau.z_i}&\tau\times\int_{\p D_i}a_m(y)~ds(y)+\frac{ik}{4\pi} \sum_{i=1}^m \tau\times\int_{\p D_i}\left(e^{-ik\tau.y}-e^{-ik\tau.z_i}\right)a_m(y)~ds(y)
       \end{align*}for every $i\in\{1,...,m\}$ and evaluate the term
 $$Q_{(\tau,z_m)}:=\tau\times\int_{\p D_i}\left(e^{-ik\tau.y}-e^{-ik\tau.z_m}\right)a_m(y)~ds(y).$$
 Developing the exponential in Taylor series, we obtain  
        \begin{align*}
        Q_{(\tau,z_m)}=&
        e^{-ik\tau.z_m}\tau\times\int_{\p D_i}\left(e^{-ik\tau.(y-z_i)}-1\right)a_m(y)~ds(y),\\
        =&e^{-ik\tau.z_m}\tau\times\int_{\p D_i}\sum_{n\geq1}\frac{\left(-ik\tau.(y-z_m)\right)^n}{n!}a_m(y)~ds(y),\\
        =&e^{-ik\tau.z_m}\tau\times\Bigl(\int_{\p D_i}\sum_{n\geq2}\frac{\left(-ik\tau.(y-z_m)\right)^n}{n!}a_m(y)~ds(y)+\int_{\p D_i}\left(-ik\tau.(y-z_m)\right) a_m(y)~ds(y)\Bigr).
       \end{align*} 
As $\abs{y-z_i}\leq \epsilon$, the first term gives us,\footnote{We have $\sum_{n\geq2} \frac{(\abs{k}\epsilon)^{n-2}}{n!}\leq \sum_{n\geq2} \frac{(\abs{k}\epsilon)^{n-2}}{(n-2)!}=e^{\abs{k}\epsilon}$.}
\begin{align*}
\Bigl|{e^{-ik\tau.z_i}\tau\times\int_{\p D_i}\sum_{n\geq2}\frac{\left(-ik\tau.(y-z_i)\right)^n}{n!}a_i(y)~ds(y)}\Bigr|
&\leq \abs{\p B_i}^\frac{1}{2} \sum_{n\geq2} \frac{(\abs{k}\epsilon)^n}{n!} \epsilon \norm*{a}_{\Ll^2(\p D_i)},\\
&\leq \abs{\p B_i}^\frac{1}{2} \sum_{n\geq2} \frac{(\abs{k}\epsilon)^{n-2}}{n!} \abs{k}^2\epsilon^4\leq \abs{\p B}^\frac{1}{2} e^{\abs{k}\epsilon} \abs{k}^2\epsilon^4.
\end{align*} 
Taking the sum over $i$, we get \begin{equation}\label{Cidessu2}  
                                                      E^\infty(\tau)=\frac{ik}{4\pi}
                                                      \sum_{i=1}^m e^{-ik\tau.z_i}\tau\times \Bigl(\mathcal{A}_i-\int_{\p D_i}\left(ik\tau.(y-z_i)\right) a_i(y)~ds(y)\Bigr)+O(e^{\abs{k}\epsilon} \abs{k}^3 m\epsilon^4).
                                             \end{equation} 
Now, considering the decomposition \eqref{Decomposition-Dnsity-a} of \cref{Decomposition-density-estimate}, we have 
                                                        \begin{align*}
                                                         \int_{\p D_i}\left(ik\tau.(y-z_i)\right) a_i(y)~ds(y)
                                                          &=\int_{\p D_i}ik\tau.(y-z_i) (a_i^{[1]}+a_i^{[2]})(y)~ds(y),\\
                                                          &= \int_{\p D_i}\left(ik\tau.(y-z_i)\right)a_i^{[2]}(y)~ds(y)+O(\abs{k}\epsilon^4),                                                   
                                                        \end{align*}
                                   where $O(\abs{k}\epsilon^4)$ comes from \begin{align*}
                                                                           \abs{\int_{\p D_i}ik\tau.(y-z_i) a_i^{[1]}(y)~ds(y)}
                                                                           \leq \abs{k}\epsilon^2\abs{\p B_i}\norm*{a^{[1]}}_{\Ll^2(\p D_i)}
                                                                           &\leq \abs{k}\epsilon^2\abs{\p B_i}C_1\epsilon \norm*{a}_{\Ll^{2,\Div}_\ta(\p D_i)},\\
                                                                           &\leq \abs{k}\abs{\p B_i}C_1C_e\epsilon^4.
                                                                           \end{align*}
Then $\int_{\p D_i}ik\tau.(y-z_i)~a_i~ds(y)=\int_{\p D_i}ik\tau.(y-z_i)~\nu\times\nabla u_i~ds(y)+O(\abs{k}\epsilon^4)$.                                                   
Multiplying by $ e^{-ik\tau.z_i} $ and taking the sum over $i$, we obtain \begin{align*}
                                                                \sum_{i=1}^m~e^{-ik\tau.z_i}\int_{\p D_i}\hspace{-0,2cm}ik\tau.(y-z_i) a_i~ds(y) 
                                                                &=\sum_{i=1}^m~e^{-ik\tau.z_i}\left(\int_{\p D_i}ik\tau.(y-z_i)\nu\times\nabla u_i~ds(y)+O(\abs{k}\epsilon^4)\right), 
                                                               \end{align*}
With this last approximation, \eqref{Cidessu2} gives \begin{align*}
                                                                   E^\infty(\tau)=&\frac{ik}{4\pi}
                                                                   \sum_{i=1}^m ~e^{-ik\tau.z_i}\tau\times 
                                                                   \Bigl(\mathcal{A}_i-\int_{\p D_i}ik\tau.(y-z_i)~\nu\times\nabla u_i(y)~ds(y)\Bigr)+ O\left((e^{\abs{k}\epsilon}\abs{k}^3+\abs{k}^2)~  m \epsilon^4\right).
                                                                  \end{align*} Finally, integrating by part the second term of the second member, we obtain  
                                                                  \begin{align*}
                                                                  \int_{\p D_i} ik\tau.(y-z_i)~\nu\times\nabla u_i~ds(y)
                                                                  &=-ik\int_{\p D_i}\left(\nu\times\nabla_y\tau.(y-z_i)\right) u_i~ds(y),\\
                                                                  &=+ik\int_{\p D_i}\left(\tau\times\nu\right) u_i(y)~ds(y)=ik\tau\times\mathcal{B}_i.
                                                                  \end{align*} \end{proof}
 \begin{lemma}\label{Field-Approximation-Lemma}
  The Electric field has the following asymptotic expansion
   \begin{equation}
   \begin{aligned}
    E^\s(x)=&\sum_{i=1}^m  \Bigl(\nabla_x\Phi_k(x,z_i)\times\mathcal{A}_i+  \nabla_y\times \nabla_x\times (\Phi_k(x,z_i) \mathcal{B}_i)\Bigr)\\
            &+O\biggl(\frac{1}{\delta}\biggl(\frac{1}{\delta^3}+\frac{3\abs{k}+1}{\delta^2}
              +\frac{5\abs{k}^2}{\delta_{i_0,i}}+\abs{k}^2+\abs{k}^3\biggr)~\epsilon^4 \biggr)\\
            &+O\biggl(\sum_{(i\neq i_0)\geq1}^m\frac{1}{\delta_{i_0,i}}\biggl(\frac{1}{\delta^3_{i_0,i}}+\frac{3\abs{k}+1}{\delta^2_{i_0,i}}
              +\frac{5\abs{k}^2}{\delta_{i_0,i}}+\abs{k}^2+\abs{k}^3\biggr)~\epsilon^4 \biggr). 
   \end{aligned}
  \end{equation}
\end{lemma}
  \begin{proof}
  For $x \in \R^3 \setminus\cup_{i=1}^m \overline{D_i},$ using Taylor formula with integral reminder, 
  we get from the representation \eqref{cas1}
   \begin{align*}
    E^\s(x)&=\sum_{i=1}^m \int_{\p D_i} \left(\nabla_x \Phi_k(x,z_i)+ \bigl(\nabla_y\nabla_x \Phi_k(x,z_i)\cdotp (y-z_i)\bigr)\right)\times a_i(y)~ds(y)\\
              &\hspace{2cm}+\sum_{i=1}^m \int_{\p D_i}\int_{0}^{1}D^3\Phi_k(x,ty+(1-t)z_i)\circ (y-z_i) (y-z_i)\times a_i(y)~ds(y),
   \end{align*}
which is\footnote{Recall that $\nabla_y\nabla_x\Phi_k(x,y)= -\nabla_x\nabla_x\Phi_k(x,y).$}
  \begin{equation} \label{Field-Firts-Approximation}
   \begin{aligned}
    E^\s(x)&=\sum_{i=1}^m  \left(\nabla_x \Phi_k(x,z_i)\times \mathcal{A}_i+\int_{\p D_i}\biggl(-\nabla_x \bigl(\nabla_x \Phi_k(x,z_i)\cdotp (y-z_i)\bigr)\biggr)\times a_i(y)~ds(y)\right)+\\
              &\hspace{2cm}\sum_{i=1}^m \int_{\p D_i}\int_{0}^{1}D^3\Phi_k(x,ty+(1-t)z_i)\circ (y-z_i) (y-z_i)\times a_i(y)~ds(y),\\
   \end{aligned}
  \end{equation}
As it was done in \eqref{gradientPhiappro}, with $d_{x,i}:=d{(x,{\p D_i})}$, we have 
                      \begin{align*}
                      \bigl|{\int_{\p D_i}\int_{0}^{1}D^3\Phi_k(x,ty+(1-t)z_i)\circ (y-z_i) (y-z_i)\times a_i(y)~ds(y)}\bigr|&\leq\\
                                &\hspace{-3cm}\frac{e^{-\Im k d_{x,i}}}{d_{x,i}}\Bigl(\frac{1}{d_{x,i}}+\abs{k}\Bigr)^3\epsilon^2\int_{\p D_i}\abs{a_i}ds.
                      \end{align*} 
For a fixed $x\in\R^3\setminus \Omega$, set $d_x:=\min_{i\in\{1,...,m\}} d_{x,i},$ hence there exists some $i_0$ such that $d_x=d(x,\p D_{i_0}),$ further $\delta_{i_0,i}=d(\p D_{i_0},\p D_i)\leq d_{x,i}+d_{x,i_0}$ 
    from which follows 
    \begin{equation}\label{Min-indices-Satisfying-d(x,D_i)} 
    \frac{1}{d_{x,i}}\leq \frac{2}{\delta_{i_0,i}}.
     \end{equation}
     Summing over $i$, the reminder, remain smaller then
                    \begin{equation}\label{Error-Field-Approximation}
                     O\Bigl(\sum_{(i\neq i_0)\geq1}^m\frac{e^{-\Im k\delta_{i_0,i}/2}}{\delta_{i_0,i}}
                     \bigl(\frac{1}{\delta_{i_0,i}}+\abs{k}\bigr)^3~\epsilon^4+\frac{e^{-\Im k\delta}}{\delta}\bigl(\frac{1}{\delta}+\abs{k}\bigr)^3~\epsilon^4\Bigr).
                     \end{equation} 
    The second term under the sum of \eqref{Field-Firts-Approximation} is precisely
                         \begin{equation}\label{Second-Term-Field-approximation}
                           \begin{aligned}
                           &\nabla_y\times \nabla_x\times (\Phi_k(x,z_i) \mathcal{B}_i)+
                           0\Bigl(\frac{e^{-\Im k\delta_{i_0,i}/2}}{\delta_{i_0,i}}\bigl(\frac{1}{\delta_{i_0,i}}+\abs{k}\bigr)^2~\epsilon^4\Bigr).
                           \end{aligned}
                           \end{equation} 
Indeed, 
                       \begin{align*}
                               \int_{\p D_i} \nabla_x \bigl(\nabla_x \Phi_k(x,z_i)\cdotp (y-z_i)\bigr)\times a_i(y)~ds(y)= \int_{\p D_i}\nabla_x\times~  \bigl[\bigl(\nabla_x \Phi_k(x,z_i)\cdotp (y-z_i)\bigr) a_i(y)\bigr]~ds(y).
                               \end{align*} 
   As we did for the far field approximation, we get in view of decomposition \eqref{Decomposition-Dnsity-a}
                                                        \begin{align}\label{FieldSecondTermappr1}
                                                         \int_{\p D_i}\left(\nabla_x\Phi_k(x,z_i).(y-z_i)\right) a_i(y)~ds(y)
                                                        = 
                                                        \int_{\p D_i}\left(\nabla_x\Phi_k(x,z_i).(y-z_i)\right) \left(a_i^{[1]}+a_i^{[2]}\right)(y)~ds(y),                                                   
                                                        \end{align} and being \begin{equation*}
                                                                                \abs{\int_{\p D_i}\nabla_x\left(\nabla_x\Phi_k(x,z_i).(y-z_i)\right)\times a_i^{[1]}(y)~ds(y)} \leq 
                                                                                C\frac{e^{-\Im k d_{x,i}}}{d_{x,i}}\left(\frac{1}{d_{x,i}}+\abs{k}\right)^2~\epsilon \int_{\p D_i}\abs*{a_i^{[1]}}ds,
                                                                               \end{equation*} we get, due to \eqref{Decomposition-Estimate-a_1-L2} 
                                                                               \begin{equation}\label{nomoreneeded}
                                                                               {\int_{\p D_i}\nabla\left(\nabla\Phi_k(x,z_i).(y-z_i)\right)\times a_i^{[1]}(y)~ds(y)}=
                                                                               0\Bigl(\frac{e^{-\Im k d_{x,i}}}{d_{x,i}}\left(\frac{1}{d_{x,i}}+\abs{k}\right)^2~\epsilon^4\Bigr).
                                                                               \end{equation} 
                                                                               Further, integrating by part, in the second step of the following identities  
                                                                                                       \begin{align*}
                                                                                                        {\int_{\p D_i}\left(\nabla_x\Phi_k(x,z_i).(y-z_i)\right) a_i^{[2]}(y)~ds(y)}
                                                                                                        &=
                                                                                                        \int_{\p D_i}\left(\nabla_x\Phi_k(x,z_i).(y-z_i)\right)  \nu\times\nabla u_i(y)~ds(y),\\
                                                                                                        &=
                                                                                                        -\int_{\p D_i}\nu\times\nabla_y\left(\nabla_x\Phi_k(x,z_i).(y-z_i)\right)  u_i(y)~ds(y),\\
                                                                                                        &=\int_{\p D_i}\nabla_x\Phi_k(x,z_i)\times\nu  u_i(y)~ds(y)=\nabla_x\Phi_k(x,z_i)\times \mathcal{B}_i,
                                                                                                        \end{align*} and differentiating, we get \begin{align*}
                                                                                                                                  -\nabla_x\times \int_{\p D_i}\left(\nabla_x\Phi_k(x,z_i).(y-z_i)\right) a_i^{[2]}(y)~ds(y)
                                                                                                                                  =
                                                                                                                                  \nabla_y \nabla_x (\Phi_k(x,z_i) \mathcal{B}_i).
                                                                                                                                 \end{align*} 

Hence, considering \eqref{nomoreneeded}, \eqref{Second-Term-Field-approximation} follows from \eqref{FieldSecondTermappr1}.

Replacing \eqref{Second-Term-Field-approximation} and \eqref{Error-Field-Approximation} in \eqref{Field-Firts-Approximation} gives
   \begin{align*}
    E^\s(x)=\sum_{i=1}^m & \biggl(  \nabla_x\Phi_k(x,z_i)\times\mathcal{A}_i+  \nabla_y\times \nabla_x\times (\Phi_k(x,z_i) \mathcal{B}_i)+O\Bigl(\frac{e^{-\Im k d_{x,i}}}{d_{x,i}}\bigl(\frac{1}{d_{x,i}}+\abs{k}\bigr)^2~m\epsilon^4\Bigr)\biggr)\\
                         &+O\Bigl( \sum_{i=1}^m\frac{e^{-\Im k\delta_{i_0,i}/2}}{\delta_{i_0,i}}\bigl(\frac{1}{\delta_{i_0,i}}+\abs{k}\bigr)^3~\epsilon^4\Bigr). 
   \end{align*}
Repeating the calculations done to get \eqref{Error-Field-Approximation}, we obtain
     \begin{align*}
    E^\s(x)=&\sum_{i=1}^m  \Bigl(\nabla_x\Phi_k(x,z_i)\times\mathcal{A}_i+  \nabla_y\times \nabla_x\times (\Phi_k(x,z_i) \mathcal{B}_i)\Bigr)\\
            &+O\Bigl(\frac{e^{-\Im k\delta}}{\delta}\Bigl[\Bigl(\frac{1}{\delta}+\abs{k}\Bigr)^2
             +\Bigl(\frac{1}{\delta}+\abs{k}\Bigr)^3\Bigr]~\epsilon^4\Bigr)\\ 
            &+O\Bigl(\sum_{(i\neq i_0)\geq 1}^m\frac{e^{-\Im k\delta_{i_0,i}/2}}{\delta_{i_0,i}}\Bigl[\Bigl(\frac{1}{\delta_{i_0,i}}+\abs{k}\Bigr)^2
             +\Bigl(\frac{1}{\delta_{i_0,i}}+\abs{k}\Bigr)^3\Bigr]~\epsilon^4\Bigr).
  \end{align*} 

The approximation \eqref{Near-Field-Approximation} follows using \eqref{Error-Exhaustiv-Expression}.
\end{proof}

\subsection{Justification of (\ref{SYS1}) and (\ref{SYS2})}
We provide the justification of (\ref{SYS2}) and then the one of (\ref{SYS1}).

\subsubsection{Justification of \eqref{SYS2}}

Let $\psi$ be any smooth enough vectorial function. Multiplying by \eqref{Equation-to-solve} and integrating over $\p D_i$, we get
                                       \begin{equation}\label{First-Linear-Sys-Approximation} 
                                              \begin{aligned}
                                              \int_{\p D_i}& \psi\cdotp[\frac{1}{2}I+M_{ii,_D}^k]a~ds+\sum_{(j\neq i)\geq1}^m\int_{\p D_i} \psi\cdotp [M_{ij,_D}^k](a_j)~ds
                                                =-\int_{\p D_i} \psi\cdotp\nu_i\times E^\n~ds.
                                               \end{aligned} 
                                      \end{equation} 
                                      
Recalling  the scaling \eqref{est1} and the estimate \eqref{Norm-M_k-M_0}\footnote{With ${\Ll^2(\p D_i)}$ norms.}, we have
                                                \begin{align*}
                                                \biggl|\int_{\p D_i}~ \psi\cdotp[\frac{1}{2}I+\MD{k}{i}]a_i^{[1]}~ds\Biggr|&\leq
                                                \norm*{\psi}\left(\norm*{[\frac{1}{2}I+\MD{0}{i}]a_i^{[1]}}+
                                                \norm*{[\MD{k}{i}-\MD{0}{i}]a_i^{[1]}}\right),\\
                                                &\leq 
                                                \norm*{\psi}_{\Ll^2(\p D_i)}\Bigl(C_{\scriptscriptstyle{_{[\MB{0}{i}]}}}
                                                +\frac{\abs{k}^2(\abs{\p B})\epsilon^2}{4\pi}
                                                 \Bigr)\norm*{a_i^{[1]}}_{\Ll^2(\p D_i)}.
                                                \end{align*} 
                                                
In view of the decomposition \eqref{Decomposition-Dnsity-a}, we obtain $\bigl|\int_{\p D_i} \psi\cdotp[\frac{1}{2}I+(\MD{k}{i})]a_i^{[1]}~ds\bigr|=
 O(\norm*{\psi}_{\Ll^2(\p D_i)}\epsilon^2),$ and then \eqref{First-Linear-Sys-Approximation} gives 
                                               \begin{align*}
                                              O(\norm*{\psi}_{\Ll^2(\p D_i)}\epsilon^2)+\int_{\p D_i} \psi\cdotp[\frac{1}{2}I+\MD{0}{i}
                                              +{(\MD{k}{i}-\MD{0}{i})}](a_i^{[2]})~&ds\\
                                              +\sum_{(j\neq i)\geq1}^m\int_{\p D_i} \psi\cdotp [M_{ij,_D}^k] (a_j^{[1]}+a_j^{[2]})~ds)~&ds=
                                               -\int_{\p D_i} \psi\cdotp\nu_i\times E^\n~ds.
                                               \end{align*} 
Using \eqref{Norm-M_k-M_0} and the estimates \eqref{Decomposition-Estimate-a_2} of the decomposition \eqref{Decomposition-Dnsity-a}, in the left hand side, we get 
                                                                             \begin{equation}\label{First-Linsys-approximation-step1} 
                                                                                   \begin{aligned}
                                                                                   O(\norm*{\psi}_{\Ll^2}\epsilon^2)+\int_{\p D_i} \psi\cdotp[\frac{1}{2}I+\MD{0}{i}](a_i^{[2]})~ds
                                                                                   +O(\abs{k\epsilon}^2\norm*{\psi}_{\Ll^2}\epsilon)&\\
                                                                                   +\sum_{(j\neq i)\geq1}^m\int_{\p D_i} \psi\cdotp [M_{ij}](a_j^{[1]}+a_j^{[2]})~ds&=\int_{\p D_i}-\psi\cdotp\nu_i\times E^\n~ds.
                                                                                   \end{aligned} 
                                                                                   \end{equation} 

Now, we show how we choose appropriate candidates $\psi$ to derive the estimates (\ref{SYS1}) and (\ref{SYS2}).

\begin{lemma}\label{Lemme-Def-Psi}
There are functions $(\psi_l)_{l=1,2,3.}$ such that 
$\nu\times\psi_l \in \Ll^{2,\Div}_\ta(\p D_i)$ and satisfying, for constants $C_{_{(\MB{0}{i})}}, C_{_{(\MB{0}{i},\KB{0*}{i})}}$ which depends only on $\abs{\p B},$
 \begin{equation}\label{psiestim}
        \begin{aligned}
        \norm*{\nu\times\psi_l}_{\Ll^2(\p Di)}
                                                             \leq C_{_{(\MB{0}{i})}} \abs{\p B}\epsilon^2,\,~ \,~ \norm*{\nu\times\psi_l}_{\Ll^{2,\Div}_\ta(\p Di)}
                                                             \leq C_{_{(\MB{0}{i},\KB{0*}{i})}} \abs{\p B}\epsilon,                            
                                                             \end{aligned} 
                                                        \end{equation}     and for which the approximation
                                                       \begin{equation}     
                                                           \int_{\p D_i} \psi_l\cdotp[\frac{1}{2}I+\MD{0}{i}](a_i^{[2]})~ds=O(\epsilon^4+\abs{k}^2\epsilon^5)+\int_{\p D_i} \nu_i^l u_i~ds, 
                                                       \end{equation} hold, here $\nu\times\nabla u_i=a_i^{[2]}$.
                                                  \end{lemma}
\begin{proof}
Let $(b_l)_{l=1,2,3}$ be the solution of the following equation       
                                                                \begin{align} \label{Def-Psi_l}
                                                                [-\frac{1}{2}I+\MD{0}{i}](b_l)=-\nu\times V_l, 
                                                                \end{align} where, $(e_l)_{l=1,2,3.}$ being the canonical base of $\mathbb{R}^3$   
                                                                  $V_1=(0,0,(x-z_i)\cdotp e_2),~V_2=((x-z_i)\cdotp e_3,0,0)$ and $V_3=(0,(x-z_i)\cdotp e_1,0)$.  
 It is evident that $\curl V_l=e_l,$ further $b_l  \in \Ll^{2,Div}_\ta(\p D_i),$ and\footnote{Having $\nu\cdotp b_l=0~ \text{ and } \nu\cdotp\nu=1.$}
\begin{equation}\label{tang} 
 \nu\times(\nu\times b_l)=(\nu\cdotp b_l)\nu-(\nu\cdotp\nu) b_l=-b_l.
 \end{equation} 
 We put $\psi_l:=-\nu\times b_l,$ hence  $\nu\times\psi_l=b_l,$ and $\Div (\nu\times\psi_l)=\Div b_l.$
 Solving \eqref{Def-Psi_l} is amount to solve the following problem (it suffices to take the surface divergence in the identity (\ref{tang}))
                                                                                                    \begin{align}\label{Surface-Divergence-of-psi_l}
                                                                                                    [\frac{1}{2}I+(\KD{0}{i})^*](\nu\cdotp\curl\psi_l)=-\nu_i^l.
                                                                                                    \end{align}  

Further, as it was done in \eqref{est1} and \eqref{scal-norm-diag-operator}, the following estimates hold
                                                                                      \begin{align*}
                                                                                       \norm*{\nu\times\psi_l}_{\Ll^2(\p Di)}
                                                                                       &\leq 
                                                                                       \norm*{[-\frac{1}{2}I+\MB{0}{i}]^{-1}}_{\mathcal{L}(\Ll_\ta^2(\p B_i))}
                                                                                       \norm*{\nu\times V_l}_{\Ll^2(\p Di)}\leq 
                                                                                       C_{_{(\MB{0}{i})}} 
                                                                                       \abs{\p B}\epsilon^2,\\
                                                                                       \norm*{\nu\times\psi_l}_{\Ll^{2,\Div}_\ta(\p Di)}
                                                                                       &\leq 
                                                                                       \norm*{[-\frac{1}{2}I+\MB{0}{i}]^{-1}}_{\mathcal{L}(\Ll_\ta^{2,\Div}(\p B_i))}
                                                                                       \norm*{V_l}_{\Ll^{2,Div}_\ta(\p Di)}\leq C_{\MB{0}{i},{\KD{0*}{i}}}
                                                                                       \abs{\p B}\epsilon.
                                                                                       \end{align*}

We have the following relations (see Lemma 5.11 \cite{DMM96})
\begin{align}\label{nu-time-nabla-interversion}
 [\frac{1}{2}I+\MD{0}{i}](a_i^{[2]})=[\frac{1}{2}I+\MD{0}{i}](\nu\times \nabla u_i)=\nu\times \nabla[\frac{1}{2}I+\KD{0}{i}]( u_i),
\end{align}  
and, for every scalar function $w$, \footnote{A direct application of \eqref{Divrel-average-zero} with $a=w\psi$.}
                                                                                        \begin{equation}\label{partintegral}
                                                                                        \int_{\p D_i}w\nu\cdotp\curl\psi~ds=
                                                                                         -\int_{\p D_i}\psi\cdotp (\nu\times\nabla w)~ds.
                                                                                        \end{equation}
Hence, the term under the integral of the left hand side of \eqref{First-Linsys-approximation-step1}, using \eqref{nu-time-nabla-interversion} and 
\eqref{partintegral}, gives  
\begin{equation}\label{integraltransform}
\begin{aligned}
 \int_{\p D_i} \psi\cdotp[\frac{1}{2}I+\MD{0}{i}](a_i^{[2]})~ds
                                                                &=\int_{\p D_i} \psi\cdotp\nu\times \nabla[\frac{1}{2}I+\KD{0}{i}]( u_i)~ds,\\
                                                                &=-\int_{\p D_i} (\nu\cdotp\curl\psi)[\frac{1}{2}I+\KD{0}{i}]( u_i)~ds,\\
                                                                &=-\int_{\p D_i} [\frac{1}{2}I+(\KD{0}{i})^*](\nu\cdotp\curl\psi)~u_ids.
 \end{aligned} 
\end{equation} 

Using \eqref{integraltransform} with $\psi_l$ as in \eqref{Def-Psi_l}, we get \begin{align*}
                                                                            O( (\epsilon^2+\abs{k\epsilon}^2\epsilon) \epsilon^2 )+\int_{\p D_i} [\frac{1}{2}I+(\KD{0}{i})^*](\nu\cdotp\curl\psi_l)u_i~ds
                                                                            =O(\epsilon^4)+O(\abs{k}^2\epsilon \epsilon^4)-\int_{\p D_i} \nu_i^l u_i~ds,   \end{align*}
to conclude that 
\begin{equation}
                                                                             \int_{\p D_i} \psi_l\cdotp[\frac{1}{2}I+\MD{0}{i}](a_i^{[2]})~ds=O(\epsilon^4+\abs{k}^2\epsilon^5)+\int_{\p D_i} \nu_i^l u_i~ds.                 
                                                                            \end{equation} 
\end{proof}
We recall the notations, for $l=1,2,3,$ $$[\mathcal{T}_{\p D_i}]^l:=\int_{\p D_i} \nu_i^l[\frac{1}{2}I+\KD{0}{i}]^{-1}(x-z_i)~ds.$$

\begin{lemma}\label{secondlappr} The second term of the left hand side of \eqref{First-Linsys-approximation-step1} admits the following approximations
                                     \begin{equation} 
                                        \begin{aligned}
                                        \int_{\p D_i} \psi_l\cdotp [M_{ij,_D}^k]a_j^{[1]}~ds=&[\mathcal{T}_{\p D_i}]^l~ \cdotp\nabla_x\Phi_k(z_i,z_j)\times \mathcal{A}_j+ 
                                        O(\frac{1}{\delta_{i,j}}\bigl(\frac{1}{\delta_{i,j}}+\abs{k}\bigr)^2{\epsilon^7})\\
                                         \end{aligned}
                                     \end{equation} 
and
                                    \begin{equation}\label{First-LSys-Second-Part-Appro} 
                                        \begin{aligned}
                                        \int_{\p D_i} \psi_l\cdotp [M_{ij,_D}^k]( a_j^{[2]})~ds=&[\mathcal{T}_{\p D_i}]^l~\cdotp\left(-k^2\Phi_k(z_i,z_j)I+\nabla_x\nabla_y \Phi_k(z_i,z_j)\right) \mathcal{B}_j\\
                                        &+O\biggl(\frac{1}{\delta_{i,j}}\Bigl((\frac{1}{\delta_{i,j}}+\abs{k})^3+(\frac{1}{\delta_{i,j}}+\abs{k})\Bigr)\epsilon^7\biggr).\\
 \end{aligned}
\end{equation} 

In addition, we have the approximation  
                                     \begin{align}\label{incapproximation2}
                                      \int_{\p D_i} \psi_l\cdotp\nu_i\times E^\n~ds=\int_{\p D_i} \nu_i^l{ [\frac{1}{2}I+\KD{0}{i}]^{-1}(x-z_i)}~ds\cdotp E^\n(z_i) + O(k\epsilon^4).
                                     \end{align}\end{lemma}
\begin{proof} Adding and subtracting $\nabla_x\Phi_k(z_i,y),$ we write  
                              \begin{equation}\label{Second-Part-Lemma-Appr}   
                               \begin{aligned}
                               \int_{\p D_i} \psi_l\cdotp [M_{ij,_D}^k]a_j^{[1]}~ds
                               =&\int_{\p D_i} \psi_l\cdotp (\nu_{x_i} \times\int_{\p D_j}\left(\nabla_x\Phi_k(x_i,y)-\nabla_x\Phi_k(z_i,y)\right)
                                 \times a_j^{[1]}(y) ~ds(y))~ds(x) \\
                                &+\int_{\p D_i} \psi_l\cdotp (\nu_{x_i} \times\int_{\p D_j}\left(\nabla_x\Phi_k(z_i,y)\right)
                               \times a_j^{[1]}(y) ~ds(y))~ds(x).
                               \end{aligned}
                        \end{equation} 
For the first integral of the right hand side, we get, using Holder's inequality then the Mean-value-theorem, with $\Ll^2(\p D_i)$ norm,  
                                     \begin{align*}
                                      \Biggl|\int_{\p D_i} \psi_l\cdotp (\nu_{x_i} \times\int_{\p D_j}&\left(\nabla_x\Phi_k(x_i,y)-\nabla_x\Phi_k(z_i,y)\right)
                                        \times a_j^{[1]}(y) ~ds(y))~ds(x)\Biggr|
                                        \\
                                        &\leq\norm*{\psi_l}~\norm*{\int_{\p D_j}\left(\nabla_x\Phi_k(x_i,y)-\nabla_x\Phi_k(z_i,y)\right)
                                        \times a_j^{[1]}(y)~ds(y)}.
                                     \end{align*} 
As
                                     \begin{align*}
                                     \Bigl|(\int_{\p D_j}\left(\nabla_x\Phi_k(x_i,y)-\nabla_x\Phi_k(z_i,y)\right)
                                     &\times a_j^{[1]}(y))~ds(y) \Bigr|
                                     \\
                                     &\leq  \sup_{x\in\p D_i}\abs{(x-z_i)}
                                     \norm*{a_j^{[1]}}\sup_{x\in\p D_i,y\in \p D_j}\abs{\nabla_x\nabla_x\Phi_k(x_i,y)}\epsilon\abs{\p B_i}^\frac{1}{2},
                                     \end{align*} 
                        with \eqref{doublgrad}, we get
                                        \begin{align*}
                                        \biggl|(\int_{\p D_j}\left(\nabla_x\Phi_k(x_i,y)-\nabla_x\Phi_k(z_i,y)\right)
                                        \times a_j^{[1]}(y))~ds(y)~~ds(x) \biggr|
                                        \leq \frac{e^{-\Im k~\delta_{i,j}}}{4\pi}\frac{1}{\delta_{i,j}}\bigl(\frac{1}{\delta_{i,j}}+\abs{k}\bigr)^2\epsilon^4.
                                        \end{align*} 
                                        
Finally \begin{align*}
                                                            \biggl|\int_{\p D_i} \psi_l\cdotp (\nu_{x_i} \times\int_{\p D_j}\left(\nabla_x\Phi_k(x_i,y)-\nabla_x\Phi_k(z_i,y)\right)
                                                            &\times a_j^{[1]}(y) ~ds(y))~ds(x)\biggr|
                                                            \leq\frac{e^{-\Im k~\delta_{i,j}}}{4\pi\delta_{i,j}}\bigl(\frac{1}{\delta_{i,j}}+\abs{k}\bigr)^2\epsilon^7,
                                                            \end{align*} 
          which is \begin{equation}\label{Erreur-Approx-Sys1-Lemma1}
                   \int_{\p D_i} \psi_l\cdotp (\nu_{x_i} \times\int_{\p D_j}\left(\nabla_x\Phi_k(x_i,y)-\nabla_x\Phi_k(z_i,y)\right)
                   \times a_j^{[1]}(y) ~ds(y))~ds(x)=O(\frac{1}{\delta_{i,j}}\bigl(\frac{1}{\delta_{i,j}}+\abs{k}\bigr)^2{\epsilon^7}).
                   \end{equation} 
        
For the second integral of the right hand side of \eqref{Second-Part-Lemma-Appr}, we get
                                                         \begin{align*}
                                                        \int_{\p D_i} \psi_l\cdotp (\nu_{x_i} \times\int_{\p D_j}&\nabla_x\Phi_k(z_i,y)
                                                        \times a_j^{[1]}(y) ~ds(y))~ds(x)\\
                                                        &=\int_{\p D_i} \psi_l\cdotp (\nu_{x_i} \times\int_{\p D_j}\nabla_x\left(\Phi_k(z_i,y)-\Phi_k(z_i,z_j)\right)
                                                        \times a_j^{[1]}(y) ~ds(y))~ds(x)\\
                                                        &+\int_{\p D_i} \psi_l\cdotp (\nu_{x_i} \times\nabla_x\Phi_k(z_i,z_j)\times\int_{\p D_j}
                                                        a_j^{[1]}(y) ~ds(y))~ds(x).
                                                        \end{align*} 
                                                        
Using again the Mean-value-theorem as in \eqref{Erreur-Approx-Sys1-Lemma1} for the first integral of the second member,
                                                                    we get\begin{equation}\label{Approxlemma2} 
                                                                                \begin{aligned}
                                                                                \int_{\p D_i} \psi_l\cdotp (\nu_{x_i} \times\int_{\p D_j}\nabla_x\Phi_k(z_i,y)
                                                                                \times a_j^{[1]}(y) ~ds(y))~ds(x)
                                                                                =&O(\frac{1}{\delta_{i,j}}\bigl(\frac{1}{\delta_{i,j}}+\abs{k}\bigr)^2{\epsilon^7})\\
                                                                                 +\int_{\p D_i} \psi_l\cdotp (\nu_{x_i}& \times\nabla_x\Phi_k(z_i,z_j)\times\int_{\p D_j}
                                                                                  a_j^{[1]} ~ds(y))~ds(x).
                                                                                \end{aligned}
                                                                          \end{equation} 
In addition, considering the fact that, for any vectors $a, b, c$ of $\R^3$ we have $a\cdotp(b\times c)=-c\cdotp(b\times a ),$ we write
\footnote{Recall the definition of $\mathcal{A}_j=\int_{\p D_i} (a_j^{[1]}+a_j^{[2]})~ds=\int_{\p D_i} a_j^{[1]}~ds.$}
                                                                                        \begin{align*}
                                                                                         \int_{\p D_i}\hspace{-0,2cm} \psi_l\cdotp (\nu_{x_i} \times\nabla_x\Phi_k(z_i,z_j)\times
                                                                                         {{\int_{\p D_j}\hspace{-0,2cm}a_j^{[1]}~ds}})~ds(x)
                                                                                         &=
                                                                                         \int_{\p D_i}\hspace{-0,2cm} \psi_l\cdotp 
                                                                                         \nu_{x_i} \times\nabla\left((x-z_i)\cdotp\nabla_x\Phi_k(z_i,z_j)\times \mathcal{A}_j\right)~ds(x),\\
                                                                                         &=-\int_{\p D_i}\hspace{-0,2cm} \nu_{x_i} \times\psi_l\cdotp 
                                                                                         \nabla\left((x-z_i)\cdotp\nabla_x\Phi_k(z_i,z_j)\times \mathcal{A}_j\right)~ds(x).
                                                                                        \end{align*} 

Integrating by parts, and considering \eqref{Surface-Divergence-of-psi_l}, we have
 \begin{align*}
                                                                                         \int_{\p D_i} \hspace{-0.2cm}\psi_l\cdotp (\nu_{x_i} \times\nabla_x\Phi_k(z_i,z_j)\times \mathcal{A}_j)~ds(x)
                                                                                         &=\int_{\p D_i}\hspace{-0,2cm} \Div(\nu\times\psi_l) 
                                                                                         \left((x-z_i)\cdotp\nabla_x\Phi_k(z_i,z_j)\times \mathcal{A}_j\right)~ds(x),\\
                                                                                         &=-\int_{\p D_i}\hspace{-0,2cm} \nu\cdotp\curl\psi_l 
                                                                                         \left((x-z_i)\cdotp\nabla_x\Phi_k(z_i,z_j)\times \mathcal{A}_j\right)~ds(x),\\
                                                                                         &=\Bigl(\int_{\p D_i}\hspace{-0,2cm} [\frac{1}{2}I+(\KD{0}{i})^*]^{-1}(\nu_i^l) 
                                                                                         (x-z_i)~ds\Bigr) \cdotp\nabla_x\Phi_k(z_i,z_j)\times \mathcal{A}_j.
                                                                                        \end{align*} 
                                                                                        Replacing in \eqref{Approxlemma2}, summing over $j$ gives the first approximation. For \eqref{First-LSys-Second-Part-Appro}, 
                                                                                        being $a_j^{[2]}=\nu\times\nabla u_j$ we have (see Lemma 5.11 \cite{DMM96}) 
                                                  \begin{align*}
                                                  [M_{ij,_D}^k]\nu\times\nabla u_j=\nu\times\nabla [K_{i,j}^k]u_j -k^2~\nu\times [S_{ij,_D}^k](\nu_y u_j),
                                                  \end{align*} 
                                     then \begin{equation}\label{Transformation2}
                                          \begin{aligned}
                                           \int_{\p D_i} \psi_l\cdotp [M_{ij,_D}^k](a_j^{[2]})~ds
                                           = \psi_l\cdotp(\nu\times\nabla [K_{i,j}^k]u_j -k^2~\nu\times [S_{ij,_D}^k](\nu_y u_j)).
                                          \end{aligned} 
                                          \end{equation}

The first term of the right hand side gives \footnote{Being $\int_{\p D_i}(\nu\cdotp\curl\psi) C=0,$ for any constant vector $C.$}
                                                \begin{equation}\label{Approxlemma3}  
                                                       \begin{aligned}
                                                       \int_{\p D_i} \psi_l\cdotp\nu\times\nabla [K_{i,j}^k]u_j~ds
                                                       &=-\int_{\p D_i}\nu\cdotp\curl\psi(x) \int_{\p D_j}\nu_y\cdotp\nabla_y\Phi_k(x_i,y) u_j(y)~ds(y)~ds(x) ,\\
                                                       &=-\int_{\p D_i}\nu\cdotp\curl\psi(x) \int_{\p D_j}\nu_y\cdotp\nabla_y\left(\Phi_k(x_i,y)-\Phi_k(z_i,y)\right) u_j(y)~ds(y)~ds(x).
                                                       \end{aligned}
                                                \end{equation}       
    
By Taylor formula at the first order, \footnote{Actually $\abs{\int_{0}^{1}D^3\Phi_k(tx+(1-t)z_i,y)dt\circ(x-z_i)~(x-z_i)}\leq 
                                                                                   \frac{C}{\delta_{i,j}}\bigl(\frac{1}{\delta_{i,j}}+\abs{k}\bigr)^3\epsilon^2.$}
                                                           \begin{equation}\label{gradientPhiappro}
                                                                     \begin{aligned}
                                                                     \nabla_y\left(\Phi_k(x_i,y)-\Phi_k(z_i,y)\right)=\nabla_x\nabla_y \Phi_k(z_i,y)(x-z_i) + 
                                                                                                                      O\Bigl(\frac{1}{\delta_{i,j}}(\frac{1}{\delta_{i,j}}+\abs{k})^3\epsilon^2\Bigr),
                                                                     \end{aligned}\end{equation}          
 \eqref{Approxlemma3} gives   \begin{align*}
                                    \int_{\p D_i} \psi_l\cdotp\nu\times\nabla [K_{i,j}^k]u_j~ds=
                                    &-\int_{\p D_i}\nu\cdotp\curl\psi(x) \Bigl(\int_{\p D_j}
                                    \nabla_x\nabla_y \Phi_k(z_i,y)(x-z_i)\cdotp\nu_y u_j(y)~ds(y)\Bigr)~ds(x)\\ 
                                    &+\int_{\p D_i}\nu\cdotp\curl\psi(x)
                                    \int_{\p D_j}\nu_y\cdotp O\Bigl(\frac{1}{\delta_{i,j}}(\frac{1}{\delta_{i,j}}+\abs{k})^3\epsilon^2\Bigr)u_j(y)~ds(y)~ds(x).
                                    \end{align*} 
 
Adding and subtracting $\nabla_x\nabla_y \Phi_k(z_i,z_j)$ under the integral, with the approximation \eqref{gradientPhiappro}, yields 
                                                                                 \begin{align*}
                                                                                 \int_{\p D_i} \psi_l\cdotp\nu\times\nabla [K_{i,j}^k]u_j~ds=
                                                                                  &-\int_{\p D_i}\nu\cdotp\curl\psi(x) \Bigl(\int_{\p D_j}\nabla_x\nabla_y \Phi_k(z_i,z_j)(x-z_i)\cdotp\nu_y u_j(y)~ds(x)\Bigr)~ds(y)\\ 
                                                                                  &+\int_{\p D_i}\nu\cdotp\curl\psi(x)
                                                                                     \int_{\p D_j}2\times\nu_y\cdotp O\Bigl(\frac{1}{\delta_{i,j}}(\frac{1}{\delta_{i,j}}+\abs{k})^3\epsilon^2\Bigr)u_j(y)~ds(y)~ds(x).
                                                                                \end{align*} 

In view of \eqref{Decomposition-Estimate-u-L2Div}, we have \begin{align*}
                                                            \Bigl|\int_{\p D_i}\nu\cdotp\curl\psi(x)
                                                            \int_{\p D_j}\nu_y\cdotp O\Bigl(\frac{1}{\delta_{i,j}}(\frac{1}{\delta_{i,j}}+\abs{k})^3&\epsilon^2\Bigr)u_j(y)~ds(y)~ds(x)\Bigr|\\
                                                            &\leq \norm*{\nu\cdotp\curl\psi}
                                                            \norm*{ \int_{\p D_j}\nu_j\cdotp O\bigl(\frac{1}{\delta_{i,j}}(\frac{1}{\delta_{i,j}}+\abs{k})^3\epsilon^2\bigr)u_j~ds},\\
                                                            &\leq O(\frac{1}{\delta_{i,j}}\bigl(\frac{1}{\delta_{i,j}}+\abs{k}\bigr)^3\epsilon^3)O(\epsilon)\norm*{u_j}O(\epsilon),\\
                                                            &=
                                                            O\bigl(\frac{1}{\delta_{i,j}}(\frac{1}{\delta_{i,j}}+\abs{k})^3\epsilon^7\bigr).
                                                            \end{align*} 
 
It follows, with $\left(\nabla_x\nabla_y \Phi_k(z_i,z_j)\right)^T$ standing for the transpose, 
                                                                     \begin{align*}
                                                                     \int_{\p D_i} \psi_l\cdotp\nu\times\nabla [K_{i,j}^k]u_j~ds=
                                                                     &-\int_{\p D_i}\nu\cdotp\curl\psi(x) (x-z_i) \cdotp
                                                                     \Bigl(\bigl(\nabla_x\nabla_y \Phi_k(z_i,z_j)\bigr)^T\int_{\p D_j}\nu_j u_j~ds \Bigr)~ds(x)~\\ 
                                                                     &+O\Bigl(\frac{1}{\delta_{i,j}}(\frac{1}{\delta_{i,j}}+\abs{k})^3\epsilon^7\Bigr),
                                                                     \end{align*}
       hence, being $\left(\nabla_x\nabla_y \Phi_k(z_i,z_j)\right)^T=\nabla_x\nabla_y \Phi_k(z_i,z_j),$  we get in view of \eqref{Surface-Divergence-of-psi_l} that
                                             \begin{equation}\label{To-Put-together1st-term-lin-sys-appr}
                                                   \begin{aligned}
                                                   \int_{\p D_i} \psi_l\cdotp\nu\times\nabla [K_{i,j}^k]u_j~ds=&O\Bigl(\frac{1}{\delta_{i,j}}(\frac{1}{\delta_{i,j}}+\abs{k})^3\epsilon^7\Bigr)\\
                                                   &-\int_{\p D_i} [\frac{1}{2}I+(\KD{0}{i})^*]^{-1}(-\nu_i^l) 
                                                    (x-z_i)~ds\cdotp\left(\nabla_x\nabla_y \Phi_k(z_i,z_j) \mathcal{B}_j \right).
                                                   \end{aligned}
                                             \end{equation}
Now, consider the second term of \eqref{Transformation2}, \footnote{With the following product rule $u\cdotp(v\times w)=-w\cdotp(v\times u ).$}
                                                                                  \begin{equation} \label{Second-Term-Sys}     
                                                                                   -k^2 \int_{\p D_i} \psi_l\cdotp (\nu\times [S_{ij,_D}^k](\nu_y u_j(y)))~ds
                                                                                   =
                                                                                   k^2\int_{\p D_i} (\nu_{x_i}\times\psi_l(x))\cdotp\int_{\p D_j} \Phi_k(x_i,y)(\nu_y u_j(y))~ds(y)~ds(x),
                                                                                   \end{equation}
                                                                                   we have \begin{align*}
                                                                                           \int_{\p D_j}\Phi_k(x_i,y)(\nu_y u_j(y))~ds(y)=&\int_{\p D_j}\Phi_k(z_i,y)~(\nu_y u_j(y))~ds(y)\\
                                                                                            &+\int_{\p D_j}\int_{0}^{1}\left(\nabla_x\Phi_k(t x +(1-t)z_i,y)dt)\cdotp(x-z_i) \right)(\nu_y u_j(y))~ds(y),
                                                                                           \end{align*}
 and repeating the same approximation in $y$, with the following estimate \begin{align*}
                                                                                       \Bigl|{\int_{\p D_j}\int_{0}^{1}\left(\nabla_x\Phi_k(t x +(1-t)z_i,y)dt)
                                                                                       \cdotp(x-z_i) \right)(\nu_y u_j(y))~ds(y)}\Bigr|
                                                                                       &\leq C\frac{\epsilon^2}{\delta_{i,j}}(\frac{1}{\delta_{i,j}+\abs{k}})\norm*{u_j}_\Ll^2(\p D_i),\\
                                                                                       &\leq C\frac{1}{\delta_{i,j}}\bigl(\frac{1}{\delta_{i,j}}+\abs{k}\bigr)\epsilon^4,
                                                                                       \end{align*}
we get \begin{align*}
                                                            \int_{\p D_j}&\Phi_k(x_i,y)(\nu_y u_j(y))~ds(y)=
                                                            \Phi_k(z_i,z_j) \int_{\p D_j}\nu_j u_j~ds+
                                                            2~O\Bigl( \frac{1}{\delta_{i,j}}(\frac{1}{\delta_{i,j}}+\abs{k})\epsilon^4\Bigr).
                                                            \end{align*} 

Replacing in \eqref{Second-Term-Sys}, gives  

\begin{align*}
                                                                     -k^2 \int_{\p D_i} \psi_l\cdotp (\nu\times [S_{ij,_D}^k](\nu_y u_j))~ds
                                                               &=k^2\int_{\p D_i} (\nu_{x_i}\times\psi_l(x))\cdotp \Bigl(\Phi_k(z_i,z_j)\mathcal{B}_j+
                                                                 O\Bigl( \frac{1}{\delta_{i,j}}(\frac{1}{\delta_{i,j}}+\abs{k})\epsilon^4\Bigr)\Bigr)~ds(x)
                                                                   \end{align*}
Considering the estimate \eqref{psiestim}, we have
 
  \begin{align*}
                                           -k^2 \int_{\p D_i} \psi_l\cdotp (\nu\times [S_{ij,_D}^k](\nu_y u_j))~ds =k^2\int_{\p D_i} (\nu_{x_i}\times\psi_l(x))\cdotp \Phi_k(z_i,z_j)\mathcal{B}_j~ds(x)+
                                             O\Bigl(\frac{1}{\delta_{i,j}}\bigl(\frac{1}{\delta_{i,j}}+\abs{k}\bigr)\epsilon^7\Bigr),
                                           \end{align*} and then, with \eqref{Divrel2} for the second inequality, we derive
                                           \begin{align*}
                                           -k^2 \int_{\p D_i} \psi_l\cdotp& (\nu\times [S_{ij,_D}^k](\nu_y u_j))~ds \\
                                            &=k^2\int_{\p D_i} (\nu_{x_i}\times\psi_l(x))\cdotp \Phi_k(z_i,z_j) \nabla\left((x-z_i)\cdotp \mathcal{B}_j\right)~ds(x)+
                                              O\Bigl(\frac{1}{\delta_{i,j}}\bigl(\frac{1}{\delta_{i,j}}+\abs{k}\bigr)\epsilon^7\Bigr),\\
                                            &=-k^2\int_{\p D_i} \Div(\nu_{x_i}\times\psi_l(x)) \Phi_k(z_i,z_j)
                                                                     \left((x-z_i)\cdotp \mathcal{B}_j\right)~ds(x)+O\Bigl(\frac{1}{\delta_{i,j}}\bigl(\frac{1}{\delta_{i,j}}+\abs{k}\bigr)\epsilon^7\Bigr).
                                           \end{align*}
As consequence, being $\Div(\nu \times\psi)=-\nu\cdotp\curl\psi$, we obtain \begin{equation}\label{putogether}
 \begin{aligned}
-k^2 \int_{\p D_i} \psi_l\cdotp (\nu\times [S_{ij,_D}^k](\nu_y u_j))~ds=
                                                                      &k^2\int_{\p D_i} [\frac{1}{2}I+(\KD{0}{i})^*]^{-1}(-\nu_i^l)(x)~\bigl((x-z_i)\cdotp \Phi_k(z_i,z_j) \mathcal{B}_j\bigr)~ds(x)\\
                                                                      &+O\Bigl(\frac{1}{\delta_{i,j}}\bigl(\frac{1}{\delta_{i,j}}+\abs{k}\bigr)\epsilon^7\Bigr).
                                                                      \end{aligned}
                                      \end{equation} 
                                      
It remain to put together \eqref{To-Put-together1st-term-lin-sys-appr}, \eqref{putogether} and to sum over $j$ to get the conclusion.

Concerning \eqref{incapproximation2}, doing as in \eqref{Second-Term-Sys}
                                       \begin{align*}
                                        \int_{\p D_i}\hspace{-0,2cm}\psi_l(x)\cdotp \nu_{x_i} \times E^\n(x)~ds(x)&=-\int_{\p D_i}\hspace{-0,2cm}\nu_{x_i}\times\psi_l\cdotp  E^\n(z_i)~ds(x) -
                                                                                             \int_{\p D_i}\hspace{-0,2cm}\nu_{x_i}\times\psi_l \cdotp (E^\n(x_i)-E^\n(z_i))~ds(x),\\
                                                                                       =&-\int_{\p D_i}\hspace{-0,2cm}\nu_{x_i}\times\psi_l\cdotp  E^\n(z_i)~ds(x)\\
                                                                                       &+ O(\norm*{\nu_{x_i}\times\psi_l}_{\Ll^2({\p D_i})} \norm*{ (E^\n(x_i)-E^\n(z_i))}_{\Ll^2({\p D_i})}).
                                       \end{align*} 
With the Mean value Theorem, we get         \begin{align*}
                                        \norm*{ (E^\n(x_i)-E^\n(z_i))}_{\Ll^2({\p D_i})}=\norm*{ \int_{0}^{1}\nabla E^\n(t x_i+(1-t)z_i)dt\cdotp(x-z_i)}_{\Ll^2({\p D_i})}
                                                                                        =O(k\epsilon^2), 
                                       \end{align*}
thus, considering \eqref{psiestim}, and \eqref{Divrel2} for the last identity, we end up with
                                        \begin{align*}
                                        \int_{\p D_i}\psi_l(x) \nu_{x_i} \times E^\n(x)~ds(x)&=-\int_{\p D_i}\nu_{x_i}\times\psi_l(x)\cdotp~\nabla\left((x-z_i)\cdotp  E^\n(z_i)\right)~ds(x) + O(k\epsilon^4),\\
                                                                                         &=\int_{\p D_i}-\nu_{x_i}\cdotp\curl\psi_l(x)~\left((x-z_i)\cdotp  E^\n(z_i)\right)~ds(x) + O(\abs{k}\epsilon^4).
                                       \end{align*}
Using \eqref{Surface-Divergence-of-psi_l} gives the conclusion.
\end{proof}
Finally, the approximation for the $\mathcal{B}_i$'s, with $\Vmt$ as defined in \eqref{Virtual-Mass-Tensor},
                                 \begin{align*}
                                  \mathcal{B}_i=\Vmt\sum_{(j\neq i)\geq1}^m&\left( -\nabla_x\Phi_k(z_i,z_j)\times \mathcal{A}_j+\Pi_k(z_i,z_j)\mathcal{B}_j\right)-\Vmt E^\n(z_i) \\
                                  &+O(\sum_{(j\neq i)\geq1}^{m}\frac{1}{\delta_{i,j}}\left((\frac{1}{\delta_{i,j}}+\abs{k})^3+
                                    (\frac{1}{\delta_{i,j}}+\abs{k})^2+(\frac{1}{\delta_{i,j}}+\abs{k})\right)\epsilon^7)\\
                                  &+O((1+\abs{k}+\abs{k}^2\epsilon)~\epsilon^4),
                                  \end{align*} 
holds. It suffices, for $l=1,2,3,$ to replace the approximations of Lemma \ref{Lemme-Def-Psi} and Lemma \ref{secondlappr} in \eqref{First-Linsys-approximation-step1} to conclude. 
Developing the approximation error of the above equation, as pointed in \eqref{Error-Exhaustiv-Expression}, gives \eqref{SYS2}.

\subsubsection{Justification of \eqref{SYS1}}

Let $\phi$ be any smooth enough scalar function. Multiply each side of \eqref{Equation-to-solve} 
 by $\nabla \phi$ and integrate over $\p D_i$ to get, using the relation \eqref{Div-definition}, 
 \begin{equation} \label{aresoudreexp}
 \int_{\p D_i} \phi  \Div~[\frac{1}{2}I+M_{ii,_D}^k]a~ds~+
 \sum_{(j\neq i)\geq1}^m\int_{\p D_i} \phi ~\Div [M_{ij,_D}^k]( a)~ds=-\int_{\p D_i} \phi~\Div(\nu_{x_i}\times E^\n)~ds.
 \end{equation}
As $~curl^2=\curl\curl= - \Delta +\nabla~div,$ we have
        \begin{equation}\label{wheretoreplace}
        \begin{aligned}
        &\int_{\p D_i}\phi\biggl([\frac{1}{2}I-(K_{\p D_i}^k)^*]\Div a-k^2\nu_{x_i}\cdotp [\SD{k}{i}]a\biggr)~ds\\
        &-\sum_{(j\neq i)\geq1}^m\int_{\p D_i}\phi\biggl([(K_{i,j}^k)^*]\Div a+ k^2\nu_{x_i}\cdotp [S_{ij,_D}^k]a \biggr)~ds=-\int_{\p D_i} \phi~\nu_i\cdotp\curl E^\n~ds.
        \end{aligned} 
       \end{equation}
        Let now $\phi$ be the solution to the following integral equation \begin{equation}
[-\frac{1}{2}I+\KD{0}{i}](\phi)(x)=(x-z_i), \label{Definition-Of-phi}
                                                                      \end{equation}
 then, as result of \eqref{estiv1}, $\phi$ satisfies the following estimate
\begin{align}\label{phi-Estimation}
 \norm*{\phi}_{\Ll^2(\p D_i)}\leq C_{\KB{0}{i}} \epsilon^2.
\end{align} 
The tensor  $\Polt$ is defined in \eqref{Polarization-Tensor}. The justification of \eqref{SYS1} is a direct consequence of the following expansions.
 \begin{lemma} \label{lemmainter1}
 With the previous notation we have the following three approximations
               \begin{align}\label{SYS1term1}
               \int_{\p D_i}\phi\biggl([\frac{1}{2}I-(K_{\p D_i}^k)^*]\Div a-&k^2\nu_{x_i}\cdotp [\SD{k}{i}]a\biggr)~ds=\mathcal{A}_i+O((\epsilon+1)\abs{k}^2\epsilon^4).
               \end{align}
               \begin{equation}\label{SYS1term2}
                              \begin{aligned}
                              \int_{\p D_i}\hspace{-0,2cm}\phi\Bigl([(K_{i,j}^k)^*]\Div a+ k^2\nu_{x_i}\cdotp [S_{ij,_D}^k]a \Bigr)~ds
                             =&\Polt \Bigl(\Pi_k(z_i,z_j)\mathcal{A}_j-k^2\nabla_x\Phi_k(z_i,z_j)\times\mathcal{B}_j\Bigr) \\
                              +O\biggl(\Bigl( (\frac{1}{\delta_{i,j}}+\abs{k})^3+&(\frac{1}{\delta_{i,j}}+\abs{k})+
                              \abs{k}^2(\frac{1}{\delta_{i,j}}+\abs{k})^2\Bigr)\frac{{\epsilon}^7}{\delta_{i,j}}\biggr),
                              \end{aligned}
               \end{equation}
               \begin{align}\label{curlEi}
                \int_{\p D_i} \phi~\nu_i\cdotp\curl E^\n~ds=\Polt_{\p D_i} \curl E^\n(z_i)+O(\abs{k}^2\epsilon^4).
                \end{align}
  \end{lemma}
\begin{proof} We have for \eqref{SYS1term1}\begin{equation}\label{Sys1-DominatingTerm} \begin{aligned}
                                               \int_{\p D_i}&\phi\biggl([\frac{1}{2}I-(\KD{k}{i})^*]\Div a+ k^2\nu_{x_i}\cdotp [\SD{k}{i}]a\biggr)~ds\\
                                                &=\int_{\p D_i}\phi\biggl([\frac{1}{2}I-(\KD{0}{i})^*]\Div a+[(\KD{0}{i})^*-(\KD{k}{i})^*]\Div a +
                                                k^2\nu_{x_i}\cdotp [\SD{k}{i}]a\biggr)~ds,
\end{aligned}
\end{equation}
Using \eqref{knusdii} and \eqref{Fnormcompact2}, the right-hand side gives 
                                                             \begin{align*}
                                                             &\int_{\p D_i}\phi[\frac{1}{2}I-(\KD{0}{i})^*]\Div a~ds+~ O( \abs{k}^2 \epsilon^2 \norm*{\Div b}_{\Ll^2_0(\p D_i)}\norm*{\phi}_{\Ll^2(\p D_i)} )\\
                                                              &+O( \abs{k}^2\epsilon \norm*{b}_{\Ll^2(\p D_i)}\norm*{\phi}_{\Ll^2(\p D_i)}).      
\end{align*} 
With \eqref{phi-Estimation} and \eqref{Estimates-Of-The-Density-a}, it becomes  
\begin{align*}
\int_{\p D_i}\phi[\frac{1}{2}I-(\KD{0}{i})^*]\Div a~ds+O((\epsilon+1)\abs{k}^2\epsilon^4).
\end{align*} 
Hence as $\int_{\p D_i}\phi[\frac{1}{2}I-(\KD{0}{i})^*]\Div a~ds=\int_{\p D_i}[\frac{1}{2}I-\KD{0}{i}]\phi\Div a~ds,$ with the definition \eqref{Definition-Of-phi} the right-hand side of \eqref{Sys1-DominatingTerm} ends up to be
\begin{align}
 -\int_{\p D_i}(x-z_i)\Div a(x)~ds(x)+O((\epsilon+1)\abs{k}^2\epsilon^4)=\mathcal{A}_i+O((\epsilon+1)\abs{k}^2\epsilon^4).
\end{align}
Concerning \eqref{SYS1term2}, we obtain, after developing the first term of the first member in Taylor series,
\begin{align*}
\int_{\p D_i}\phi~[(K_{i,j}^k)^*](\Div a)~ds=& \int_{\p D_i}\phi(x)~\nu\cdotp\int_{\p D_j}\left( \nabla_y\Phi_k(x,z_j)+\nabla_y\nabla_x\Phi_k(x,z_j)\cdotp(y-z_j)\right)\Div a(y)~ds(y)~ds(x)\\
                                                                  +\int_{\p D_i}\phi(x)~\nu\cdotp\int_{\p D_j}&\left(\int_{0}^{1} D^3_{y,y,x}\Phi_k(x,ty+(1-t)z_j)\circ(y-z_j)\cdotp(y-z_j)\right)\Div a(y)~ds(y)~ds(x),                                
\end{align*}
which gives as we did it in \eqref{gradientPhiappro}\footnote{Note that $\int_{\p D_j}\nabla_x\Phi_k(x,z_j)\Div a=0.$}  
\begin{align*}
\int_{\p D_i}\phi~[(K_{i,j}^k)^*](\Div a)=& \int_{\p D_i}\phi(x)~\nu\cdotp\int_{\p D_j}\left(\nabla_y\nabla_x\Phi_k(x,z_j)\cdotp(y-z_j)\right)\Div a(y)~ds(y)~ds(x)\\
&+O( \frac{1}{\delta_{i,j}}\bigl(\frac{1}{\delta_{i,j}}+\abs{k}\bigr)^3\epsilon^2\epsilon\norm*{\Div a}\epsilon\norm*{\phi}).
\end{align*} 
Repeating the same computations for $x\in\p D_i,$ we get, in consideration of \eqref{phi-Estimation}
\begin{align*}
\int_{\p D_i}\phi~[(K_{i,j}^k)^*](\Div a)&=\int_{\p D_i}\phi(x)~\nu\cdotp\int_{\p D_j}\left(\nabla_y\nabla_x\Phi_k(z_i,z_j)\cdotp(y-z_j)\right)\Div a(y)~ds(y)~ds(x)\\
&+2~O( \frac{1}{\delta_{i,j}}\bigl(\frac{1}{\delta_{i,j}}+\abs{k}\bigr)^3\epsilon^7).
\end{align*} 
With the notation \eqref{Polarization-Tensor}, we get \footnote{Recall that $\int_{\p D_j}(y-z_j)\Div a(y)~ds(y)=-\int_{\p D_j} a(y)~ds(y)=-\mathcal{A}_j.$}
\begin{equation}  \label{1stlemmaint1}
\begin{aligned}
\int_{\p D_i}\phi~[(K_{i,j}^k)^*](\Div a)~ds&=-\Polt \left(\nabla_y\nabla_x\Phi_k(z_i,z_j)\mathcal{A}_j\right)+O( \frac{1}{\delta_{i,j}}\bigl(\frac{1}{\delta_{i,j}}+\abs{k}\bigr)^3\epsilon^7).
\end{aligned}
\end{equation} 
Concerning the second term of the first member of \eqref{SYS1term2}, we have in view of \eqref{Decomposition-Dnsity-a} \cref{Decomposition-density-estimate}
\begin{equation}\label{SYS1term2appro}
\begin{aligned}
\int_{\p D_i}\phi~k^2\nu_i\cdotp [S_{ij,_D}^k](a_j)~ds=\int_{\p D_i}\phi~k^2\nu_i\cdotp [S_{ij,_D}^k]\Bigl(a_j^{[1]}+a_j^{[2]}\Bigr)~ds,\\                                             
\end{aligned}
\end{equation}        
and then 
\begin{align*}
\int_{\p D_i}\phi~k^2\nu_i\cdotp [S_{ij,_D}^k](a_j^{[1]})~ds=& \int_{\p D_i}\phi(x)~k^2\nu_{x_i}\cdotp \int_{\p D_j}\left(\Phi_k(x,y)-\Phi(z_i,z_j)\right)a_j^{[1]}(y)~ds(y)~ds(x)\\
&+\int_{\p D_i}\phi(x)~k^2\nu_{x_i}\cdotp \int_{\p D_j}\Phi_k(z_i,z_j)a_j^{[1]}(y)~ds(y)~ds(x),
\end{align*} 
Using Mean-value-theorem, we get for the right-hand side of the above equation
\begin{align*}
\int_{\p D_i}\phi~k^2\nu_i\cdotp \Bigl(\int_{\p D_j}O\bigl((\frac{1}{\delta_{i,j}}+\abs{k})\frac{1}{\delta_{i,j}}\epsilon\bigr) a_j^{[1]}(y)~ds(y)\Bigr)~ds+ k^2 \Phi_k(z_i,z_j)\int_{\p D_i}\phi~\nu_i\cdotp \Bigl(\int_{\p D_j}a_j^{[1]}(y)~ds(y)\Bigr)~ds,
\end{align*} 
then considering the estimates \eqref{Decomposition-Estimate-a_1-L2} and \eqref{phi-Estimation}, with Holder's inequality give  
\begin{equation}\label{approximationSij1}
\int_{\p D_i}\phi~k^2\nu_i\cdotp [S_{ij,_D}^k] (a_j^{[1]})~ds =
 \Polt k^2 \Phi_k(z_i,z_j)\mathcal{A}_j+O\bigl((\frac{1}{\delta_{i,j}}+\abs{k})\frac{1}{\delta_{i,j}}\epsilon^7\bigr).
\end{equation}
The second term of the second member of \eqref{SYS1term2appro} gives, again considering \eqref{Decomposition-Dnsity-a},
\begin{align*}
\int_{\p D_i}\phi~k^2\nu_i\cdotp [S_{ij,_D}^k]a_j^{[2]}~ds=\int_{\p D_i}\phi(x)~k^2\nu_{x_i}\cdotp \int_{\p D_j}\Phi_k(x,y)\nu\times \nabla u_j ~ds(y)~ds(x),
\end{align*} 
which, by integrating by party, gives 
\begin{align*} 
\int_{\p D_i}\phi~k^2\nu\cdotp [S_{ij,_D}^k]a_j^{[2]}~ds=\int_{\p D_i}\phi(x)~k^2\nu_{x_i}\cdotp \Bigl(-\int_{\p D_j}u_j(y)\nu_{y_i}\times \nabla_y\Phi_k(x,y) ~ds(y) \Bigr)~ds(x).
\end{align*} 
Now, doing a first order approximation, we have
\begin{align*} 
\int_{\p D_i}\phi~k^2\nu\cdotp [S_{ij,_D}^k]a_j^{[2]}~ds=&\int_{\p D_i}\phi(x)~k^2\nu_{x_i}\cdotp \Bigr(-\int_{\p D_j}u_j\nu_j ~ds\Bigl)\times \nabla_y\Phi_k(x,z_j)~ds(x)\\  
&+\int_{\p D_i}\phi(x)~k^2\nu_{x_i}\cdotp \Bigl(-\int_{\p D_j}u_j(y)\nu_{y_i}\times \Bigl(\nabla_y\Phi_k(x,y)-\nabla_y\Phi_k(x,z_j)\Bigr) ~ds(y)\Bigr)~ds(x),   
\end{align*} 
and similarly to \eqref{Erreur-Approx-Sys1-Lemma1} the right-hand side is equal to  
\begin{align*}
\int_{\p D_i}\phi(x)~k^2\nu_{x_i}\cdotp \Bigl(-\int_{\p D_j}u_j\nu_j ~ds\Bigr)\times \nabla_y\Phi_k(x,z_j)~ds(x)\\
+\int_{\p D_i}\phi(x)~k^2\nu_{x_i}&\cdotp \biggl(-\int_{\p D_j}u_j\nu_j\times
O\Bigl(\frac{1}{\delta_{i,j}}(\frac{1}{\delta_{i,j}}
+\abs{k})^2{\epsilon}\Bigr) ~ds\biggr)~ds(x),
\end{align*} 
which, in view of the estimates \eqref{Decomposition-Estimate-u-L2Div} and \eqref{phi-Estimation}, becomes    
\begin{align*}
\int_{\p D_i}\phi(x)~k^2\nu_{x_i}\cdotp \Bigl(-\int_{\p D_j}u_j\nu_j~ds\Bigr)\times \nabla_y\Phi_k(x,z_j)~ds(x)
+O\Bigl(\frac{\abs{k}^2}{\delta_{i,j}}(\frac{1}{\delta_{i,j}}
+\abs{k})^2{\epsilon}^7\Bigr).                                                                                                 
\end{align*} 
Repeating the same calculation, for  $x\in \p D_i,$ gives
                                                                                                \begin{align*}
                                                                                                \int_{\p D_i}\phi~k^2\nu\cdotp [S_{ij,_D}^k]a_j^{[2]}~ds=&
                                                                                                \int_{\p D_i}\phi~k^2\nu_i\cdotp \Bigl(-\int_{\p D_j}u_j\nu_j ~ds\Bigr)\times \nabla_y\Phi_k(z_i,z_j)~ds\\
                                                                                                &+O\Bigl(\frac{\abs{k}^2}{\delta_{i,j}}(\frac{1}{\delta_{i,j}}
                                                                                                  +\abs{k})^2{\epsilon}^7\Bigr).
                                                                                                \end{align*} 

Hence, being $V\times U=-U\times V$ for any vectors $U,V$, and $\nabla_y\Phi_k(x,y)=-\nabla_x\Phi_k(x,y)$ we get with the notations \eqref{Polarization-Tensor}
                                                                                                 \begin{align*}
                                                                                                \int_{\p D_i}\phi~k^2\nu\cdotp &[S_{ij,_D}^k]a_j^{[2]}~ds=
                                                                                                -\Polt k^2\nabla_x\Phi_k(z_i,z_j)\times\mathcal{B}_j +
                                                                                                O\left(\abs{k}^2\frac{1}{\delta_{i,j}}(\frac{1}{\delta_{i,j}}
                                                                                                  +\abs{k})^2{\epsilon}^7\right).
                                                                                                \end{align*}
The last approximation of Lemma \ref{lemmainter1} being obvious, we end the proof of \eqref{SYS1} by taking the sum over $j$ of the two first approximations and replacing in \eqref{wheretoreplace}.
\end{proof}
\section{Invertibility of the linear system}\label{Section-invertibility-algebraic-system}
 For $\mu^+$ and $\mu^-$ defined as in \eqref{mudefinition} and ${\mathcal{E}}=({\mathcal{E}})_{i=1}^{2m}$ defined as
\begin{equation}
{\mathcal{E}}_i= \left\{\begin{aligned} & E^\n(z_i),~i\in\{1,...,m\},\\
                       &\curl E^\n(z_{i-m}),~i\in\{m+1,...,2m\},
\end{aligned}\right.
\end{equation}
 we have the following proposition.
 \begin{proposition}\label{propositioncondinvsysl}
 Under the condition
       \begin{align}
         C_{Li}:=1-C_{Ls}\frac{\mu^+\epsilon^3}{\delta^3}>0 ,
        \end{align} for some constant  $C_{Ls},$ \footnote{The constant $C_{Ls}$ is provided in (\ref{constant-Cls}).}
the following linear system is invertible
\begin{equation}\label{SYSTOT}
 \begin{aligned}
\widehat{\mathcal{B}}_i&=\Vmt\sum_{(j\neq i)\geq1}^m\left(~ -\nabla_x\Phi_k(z_i,z_j)\times \widehat{\mathcal{A}}_j~+ \Pi_k(z_i,z_j)\widehat{\mathcal{B}}_j\right)-\Vmt E^\n(z_i), \\
 \widehat{\mathcal{A}}_i&=-\Polt \sum_{(j\neq i)\geq1}^m \left(\Pi_k(z_i,z_j)\widehat{\mathcal{A}}_j-k^2\nabla\Phi_k(z_i,z_j)\times\widehat{\mathcal{B}}_j\right) -\Polt\curl E^\n(z_i),
\end{aligned}
\end{equation} and the solution satisfies the following estimate 
\begin{equation}\label{Estimate-Solutions-of-LinSys}
 \begin{aligned}
\left(\sum_{i=1}^m\bigl(\abs*{\widehat{\mathcal{A}}_i}^2+\abs*{\widehat{\mathcal{B}}_i}^2\bigr)\right)^\frac{1}{2} :=\left(\bigl<\widehat{\mathcal{B}},\widehat{\mathcal{B}}\bigr>_{\mathbb{C}^{3\times m}}
                                                    +\bigl<\widehat{\mathcal{A}},\widehat{\mathcal{A}}\bigr>_{\mathbb{C}^{3\times m}}\right)^\frac{1}{2}\leq 
                                                   \frac{1}{C_{Li}\mu^-} \epsilon^3\bigl<\widehat{\mathcal{E}},\widehat{\mathcal{E}}\bigr>_{\mathbb{C}^{3\times 2m}}^\frac{1}{2}.
\end{aligned}
\end{equation} 
Further, if the condition \eqref{NeumannSeries-Cond} is satisfied, then the system could be inverted using Neumann series with the following estimate 
\begin{equation}\label{Estimate-Each-Solution-LinSys}
\begin{aligned}
\abs{\mathcal{A}_i}\leq \frac{1}{C_{L^2_i}{\mu^-}} \epsilon^3 \abs{\mathcal{E}_i},\,~
\abs{\mathcal{B}_i}\leq  \frac{1}{C_{L^2_i}{\mu^-}} \epsilon^3 \abs{\mathcal{E}_{i+m}}.
\end{aligned}
\end{equation}
\end{proposition} 

To prove this result, we need to introduce some notations. Let $(\widehat{\mathcal{C}}_i)_{i\in\{1,...,2m\}}$ be defined as
\begin{equation*}
\widehat{\mathcal{C}}_i= \left\{\begin{aligned}
&\Vmt^{-1}\widehat{\mathcal{B}}_i,~i\in\{1,...,m\},\\
&-\bigl[\mathcal{P}_{\p D_{i-m}}\bigr]^{-1}\widehat{\mathcal{A}}_{i-m},~i\in\{m+1,...,2m\},
\end{aligned}\right.
\end{equation*} 
and let $\mathcal{Q}$ be the following diagonal Bloc matrix 
\begin{equation}
Q_i:=\left\{\begin{aligned}
 &\Vmt, &&\text{for }i\in\{1,...,m\},\\
 &-\bigl[\mathcal{P}_{\p D_{i-m}}\bigr]&& \text{for }i\in\{m+1,...,2m\},
  \end{aligned}\right.
 \end{equation} 
 with $\mathcal{Q}_{1,1}=Diag(Q_i)_{i=1}^m,$ and $\mathcal{Q}_{2,2}=Diag(Q_{i+m})_{i=1}^m.$
 
 Consider \begin{equation}
         \widehat{\theta}_{i,j}\widehat{\mathcal{C}}_j:=\nabla_x\Phi_k(z_i,z_j)\times \widehat{\mathcal{C}}_j=\left[\begin{aligned}
                                                             &0                         &&-\partial_3\Phi_k(z_i,z_j)&&\partial_2\Phi_k(z_i,z_j)\\
                                                             &\partial_3\Phi_k(z_i,z_j)&&0                         &&-\partial_1\Phi_k(z_i,z_j)\\
                                                             &-\partial_2\Phi_k(z_i,z_j)&&\partial_1\Phi_k(z_i,z_j) &&0
                                                            \end{aligned}\right] \widehat{\mathcal{C}}_j,
            \end{equation}
                                                         
 and define the following bloc matrix \begin{equation*}
                                  \Sigma^k:=\biggl[\begin{aligned}
                                                   &\Sigma^k_{1,1}&&0_{\mathbb{C}^m\times \mathbb{C}^m}\\
                                                   &0_{\mathbb{C}^m\times \mathbb{C}^m}&&\Sigma^k_{2,2}
                                                  \end{aligned}\biggr]=(\sigma_{i,j}^k)_{i,j=1}^{2m},\,~\,~ \varTheta^k:=\biggl[\begin{aligned}
                                                                                                                             &0_{\mathbb{C}^m\times \mathbb{C}^m}&&\varTheta^k_{1,2}\\
                                                                                                                             &\varTheta^k_{2,1}&& 0_{\mathbb{C}^m\times \mathbb{C}^m}
                                                                                                                             \end{aligned}\biggr]=(\theta_{i,j}^k)_{i,j=1}^{2m},   
                                       \end{equation*} where \begin{equation}
                                     {\sigma^k}_{i,j}:= \left\{\begin{aligned}
                                                               &-\Pi_k(z_i,z_j),&&~i\neq j, i,j\in\{1,...,m\},\\
                                                               &-\Pi_k(z_{i-m},z_{j-m}),&&~i\neq j,~i,j\in\{m+1,...,2m\},\\
                                                               &0, &&\text{\em otherwise},
                                                              \end{aligned}\right.
                                                              \end{equation}
                                                          and \begin{equation}
                                     \theta_{i,j}^k= \left\{\begin{aligned}
                                                                &\widehat{\theta}_{i,j-m}, &&~i\in\{1,...,m\},~j\in\{1+m,...,2m\},\\
                                                               &k^2\widehat{\theta}_{i-m,j} &&~j\in\{1,...,m\},~i\in\{1+m,...,2m\},\\
                                                               &0,&&\text{\em otherwise}.
                                                              \end{aligned}\right.
                                                              \end{equation} 
With these notations, solving the system \eqref{SYSTOT} is equivalent to solve the equation
                                                              \begin{align} \label{Linear-System-Transform}
                                                               \widehat{\mathcal{C}}+\Sigma^k\mathcal{Q}\widehat{\mathcal{C}}+\varTheta^k\mathcal{Q}\widehat{\mathcal{C}}=\mathcal{E}.
                                                              \end{align}
                                                              
If we multiply both sides of the last system by $\mathcal{Q}\widehat{\mathcal{C}}$ we get 
                                                     \begin{equation}\label{EstimateSystem}
                                                              \begin{aligned}
                                                              \bigl<\widehat{\mathcal{C}}, \mathcal{Q}\widehat{\mathcal{C}}\bigr>_{\mathbb{C}^{3\times 2m}}+
                                                              \bigl<\Sigma^k\mathcal{Q}\widehat{\mathcal{C}}, \mathcal{Q}\widehat{\mathcal{C}}\bigr>_{\mathbb{C}^{3\times 2m}}+
                                                              \bigl<\varTheta^k\mathcal{Q}\widehat{\mathcal{C}}, \mathcal{Q}\widehat{\mathcal{C}}\bigr>_{\mathbb{C}^{3\times 2m}}
                                                              =\bigl<\mathcal{E}, \mathcal{Q}\widehat{\mathcal{C}}\bigr>_{\mathbb{C}^{3\times 2m}}
                                                              \end{aligned} 
                                                      \end{equation}        
where $\bigl<\cdotp,\cdotp\bigr>_{\mathbb{C}^{3\times 2m}}$ stands for the usual scalar product in $\mathbb{C}^{3\times 2m}$.

 Adding and subtracting $\bigl<\Sigma^0\mathcal{Q}\widehat{\mathcal{C}}, \mathcal{Q}\widehat{\mathcal{C}}\bigr>_{\mathbb{C}^{3\times 2m}}$ gives
                                                        \begin{equation}\label{Systemintermdiare}     
                                                             \begin{aligned}
                                                              \bigl<\widehat{\mathcal{C}}, \mathcal{Q}\widehat{\mathcal{C}}\bigr>_{\mathbb{C}^{3\times 2m}}+
                                                              \bigl<\left(\Sigma^k-\Sigma^0\right)\mathcal{Q}\widehat{\mathcal{C}}, \mathcal{Q}\widehat{\mathcal{C}}\bigr>_{\mathbb{C}^{3\times 2m}}
                                                              +&\bigl<\Sigma^0\mathcal{Q}\widehat{\mathcal{C}}, \mathcal{Q}\widehat{\mathcal{C}}\bigr>_{\mathbb{C}^{3\times 2m}}\\
                                                              +&\bigl<\varTheta^k\mathcal{Q}\widehat{\mathcal{C}}, \mathcal{Q}\widehat{\mathcal{C}}\bigr>_{\mathbb{C}^{3\times 2m}}
                                                              =\bigl<\mathcal{E}, \mathcal{Q}\widehat{\mathcal{C}}\bigr>_{\mathbb{C}^{3\times 2m}},
                                                              \end{aligned}
                                                        \end{equation}
 Let $\chi_{\p \widehat{\Omega}}$ denotes  the characteristic function on $\p \widehat{\Omega}=\cup_{i=1}^m \p B_{\delta/4}(z_i)$ where $B_r(z):=B(z,r)$ denotes a ball of center $z$ and radius $r,$ and 
 $\bigl<\cdotp,\cdotp\bigr>_{\Ll^2(\p \widehat{\Omega})}$ denotes the usual scalar product of $\Ll^2(\p \Omega)$.  

\begin{lemma} For $\mathcal{U}:=\sum_{i=1}^{m}\nu_{xi}\cdotp \overline{Q_i\widehat{\mathcal{C}}_i} \chi_{{\p B_{\delta/4}^{z_i}}}$, and 
$\mathcal{V}:=\sum_{i=m+1}^{2m}\nu_{x_{i-m}}\cdotp \overline{Q_i\widehat{\mathcal{C}}_i} \chi_{{\p B_{\delta/4}^{z_{i-m}}}}$ with $\nu_{xi}$ being 
                           the outward unit normal vector to $\p B_{\delta/4}^{z_i},$ we have
                           \begin{equation}\label{DominatingTerm1}
                           \Sigma^0_{1,1}\mathcal{Q}_{1,1}\widehat{\mathcal{C}}_1\cdotp \mathcal{Q}_{1,1}\widehat{\mathcal{C}}_1=\frac{48^2 }{\pi^2\delta^6}
                                               \Bigl(\bigl<S_{\p \widehat{\Omega}}^0~\mathcal{U},\mathcal{U}\bigr>_{\Ll^2(\p \widehat{\Omega})}-
                                                             \sum_{i=1}^{m}\bigl<S_{{\p B_{\delta/4}^{z_i}}}^0~\mathcal{U},\mathcal{U}\bigr>_{\Ll^2({\p B_{\delta/4}^{z_i}})}\Bigr),
                            \end{equation}
                            \begin{equation}\label{DominatingTerm2}
                           \Sigma^0_{2,2}\mathcal{Q}_{2,2}\widehat{\mathcal{C}}_2\cdotp \mathcal{Q}_{2,2}\widehat{\mathcal{C}}_2=\frac{48^2 }{\pi^2\delta^6}
                                               \Bigl(\bigl<S_{\p \widehat{\Omega}}^0\mathcal{V},\mathcal{V}\bigr>_{\Ll^2(\p \widehat{\Omega})}-
                                                  \hspace{-0,2cm}\sum_{i=m+1}^{2m}\bigl<S_{{\p B_{\delta/4}^{z_{i-m}}}}^0~\mathcal{V},\mathcal{V}\bigr>_{\Ll^2({\p B_{\delta/4}^{z_{i-m}}})}\Bigr),
                            \end{equation}
                            \begin{align}\label{Differestimate}
                            \abs*{\bigl<\left(\Sigma^k-\Sigma^0\right)\mathcal{Q}\widehat{\mathcal{C}}, \mathcal{Q}\widehat{\mathcal{C}}\bigr>_{\mathbb{C}^{3\times 2m}}}\leq
                            \frac{63^\frac{1}{2} \abs{k}^2m^{\frac{2}{3}}}{4\pi\delta}\sum_{i=1}^{2m}\abs*{Q_i\widehat{\mathcal{C}}_i}^2
                            =\frac{63^\frac{1}{2} \abs{k}^2D(\Omega)^{\frac{2}{3}}}{4\pi\delta^3}\bigl<\mathcal{Q}\widehat{\mathcal{C}},\mathcal{Q}\widehat{\mathcal{C}}\bigr>_{\mathbb{C}^{3\times 2m}},
                            \end{align}
                            and
                            \begin{align}\label{Differestimate2}
                             \abs*{\bigl<\varTheta^k\mathcal{Q}\widehat{\mathcal{C}},\mathcal{Q}\widehat{\mathcal{C}}\bigr>_{\mathbb{C}^{3\times 2m}}}
                             \leq 
                             \frac{1}{8\pi} \frac{(1+\abs{k}^2)C_{k,_{D(\Omega)}}}{\delta^3} \bigl<\mathcal{Q}\widehat{\mathcal{C}},\mathcal{Q}\widehat{\mathcal{C}}\bigr>_{\mathbb{C}^{3\times 2m}}.
                            \end{align} where $C_{k,_{D(\Omega)}}:=\max(C_0D(\Omega)^\frac{1}{3}, \abs{k}C_0 \bigl(D(\Omega)^\frac{2}{3}+D(\Omega)^\frac{1}{3}\bigr))$.
                            \end{lemma}
                            
\begin{proof} To prove \eqref{DominatingTerm1}, using Mean-value-theorem for harmonic function, for $j\neq i$, we have
         \begin{align*}
          \nabla_x\nabla_y \Phi_0(z_i,z_j)=\frac{3\times 4^3 }{4\pi\delta^3}\int_{B(z_j,\frac{\delta}{4})}\nabla_x\nabla_y \Phi_0(z_i,y)dy
                                            =\frac{48 }{\pi\delta^3}\nabla_x\int_{ B(z_j,\frac{\delta}{4})} \nabla_y\Phi_0(z_i,y)~dy
         \end{align*} 
then using Gauss divergence theorem,
                                  \begin{align*}
                                  \nabla_x\nabla_y \Phi_0(z_i,z_j)=\frac{48 }{\pi\delta^3}\nabla_x\int_{\p B_{\delta/4}^{z_i}} \Phi_0(z_i,y)\nu_y~ds(y).
                                  \end{align*} 
Repeating the same for $z_i$ we get \begin{align}\label{Mean-Value-Theorem-Trasform}
                                    \nabla_x\nabla_y \Phi_0(z_i,z_j)=\frac{48^2 }{\pi^2\delta^6}\int_{\p B_{\delta/4}^{z_i}}\int_{\p B_{\delta/4}^{z_i}} \Phi_0(x,y)\nu_x\nu_y^T~ds(y)~ds(x),
                                    \end{align} 
 and from $\bigl<\Sigma^0_{1,1}\mathcal{Q}_{1,1}\widehat{\mathcal{C}}_1,\mathcal{Q}_{1,1}\widehat{\mathcal{C}}_1\bigr>_{\mathbb{C}^{3\times m}}=
 \sum_{i=1}^{m}\sum_{1\leq j\neq i}^{m}(\nabla_x\nabla_y \Phi_0(z_i,z_j)Q_j\widehat{\mathcal{C}}_j)\cdotp \overline{Q_i\widehat{\mathcal{C}}_i},$
using \eqref{Mean-Value-Theorem-Trasform}, we get 
                                 \begin{equation}\label{Linearsystsig0}
                                \begin{aligned}
                                \bigl<\Sigma^0_{1,1}\mathcal{Q}_{1,1}\widehat{\mathcal{C}}_1,\mathcal{Q}_{1,1}\widehat{\mathcal{C}}_1\bigr>=
                                \frac{48^2 }{\pi^2\delta^6}\sum_{i=1}^{m}\Bigl(\sum_{1\leq j\neq i}^{m} \int_{\p B_{\delta/4}^{z_i}}\int_{\p B_{\delta/4}^{z_i}} 
                                \hspace{-0,3cm}\Phi_0(x,y)\nu_x\nu_y^T~ds(y)~ds(x)~Q_j\widehat{\mathcal{C}}_j\Bigr)\cdotp \overline{Q_i\widehat{\mathcal{C}}_i},\\
                                =
                                \frac{48^2 }{\pi^2\delta^6}\sum_{i=1}^{m} \int_{\p B_{\delta/4}^{z_i}}\Bigl(\sum_{1\leq j\neq i}^{m}\int_{\p B_{\delta/4}^{z_i}} 
                                \Phi_0(x,y)\nu_y\cdotp Q_j\widehat{\mathcal{C}}_j~ds(y)\Bigr)\nu_x\cdotp \overline{Q_i\widehat{\mathcal{C}}_i}~ds(x).\\
                               \end{aligned}
                              \end{equation} 
                              
Adding and subtracting $\sum_{i=1}^{m}\int_{\p B_{\delta/4}^{z_i}}\int_{\p B_{\delta/4}^{z_i}} 
                                \Phi_0(x,y)\nu_y\cdotp Q_j\widehat{\mathcal{C}}_i~ds(y)\nu_x\cdotp \overline{Q_i\widehat{\mathcal{C}}_i}~ds(x),$ gives 
                                                                                 \begin{equation}
                                                                                     \begin{aligned}
                                                                                     \bigl<\Sigma^0_{1,1}\mathcal{Q}_{1,1}\widehat{\mathcal{C}}_1,\mathcal{Q}_{1,1}\widehat{\mathcal{C}}_1\bigr>
                                                                                     =&\frac{48^2 }{\pi^2\delta^6}\sum_{i=1}^{m} \int_{\p B_{\delta/4}^{z_i}}\Bigl(\sum_{j=1}^{m}\int_{\p B_{\delta/4}^{z_i}}
                                                                                     \Phi_0(x,y)\nu_y\cdotp Q_j\widehat{\mathcal{C}}_j~ds(y)\Bigr)\nu_x\cdotp \overline{Q_i\widehat{\mathcal{C}}_i}ds(x),\\
                                                                                     &-\frac{48^2 }{\pi^2\delta^6} \sum_{i=1}^{m}\int_{\p B_{\delta/4}^{z_i}}\int_{\p B_{\delta/4}^{z_i}} 
                                                                                     \Phi_0(x,y)\nu_y\cdotp Q_j\widehat{\mathcal{C}}_i~ds(y)\nu_x\cdotp \overline{Q_i\widehat{\mathcal{C}}_i}ds(x),
                                                                                     \end{aligned}
                                                                                 \end{equation}
                                                                                 
and then \eqref{Linearsystsig0} becomes, \begin{align*}
                                 \bigl<\Sigma^0_{1,1}\mathcal{Q}_{1,1}\widehat{\mathcal{C}}_1,\mathcal{Q}_{1,1}\widehat{\mathcal{C}}_1\bigr>_{\mathbb{C}^{3\times m}} 
                                 =& 
                                 \frac{48^2 }{\pi^2\delta^6}\int_{\p \widehat{\Omega}:=\cup_{i=1}^{2m} \p B_{\delta/4}^{z_i}}\int_{\cup_{j=1}^{2m} \p B_{\delta/4}^{z_i}} 
                                 \Phi_0(x,y)\mathcal{U}(y)~ds(y)~\overline{\mathcal{U}}(x)~ds(x)\\
                                 &-\frac{48^2 }{\pi^2\delta^6} \sum_{i=1}^{2m}\int_{\p B_{\delta/4}^{z_i}}\int_{\p B_{\delta/4}^{z_i}} 
                                 \Phi_0(x,y)\nu_y\cdotp Q_j\widehat{\mathcal{C}}_i~ds(y)\nu_x\cdotp \overline{Q_i\widehat{\mathcal{C}}_i}~ds(x).
                                 \end{align*} 
The same arguments remain valid for \eqref{DominatingTerm2}.\\
Concerning \eqref{Differestimate}, we have \begin{align}\label{Mean-Vvalue-Double-Gradient}
                                           \abs{(-k^2\Phi_k(z_i,z_j)I+\nabla_x\nabla_y\Phi_k(z_i,z_j)  -\nabla_x\nabla_y\Phi_0(z_i,z_j))}\leq \frac{6\abs{k}^2}{4\pi \delta_{i,j}}.
                                          \end{align} 
              
Indeed, recalling \eqref{Mean-Val-Thm} and \eqref{meanvthe}  
                                                                                     \begin{align*}
                                                                                      \nabla_y(\Phi_k(x,y)-\Phi_0(x,y))= \frac{-(ik)^2}{4\pi}\int_{0}^{1}\frac{le^{ikl\abs{x-y}}}{\abs{x-y}}(x-y)dl
                                                                                     \end{align*} 
then we have also \begin{align*}
                                                                                 \nabla_x\nabla_y(\Phi_k(x,y)-\Phi_0(x,y))
                                                                                 =\frac{-(ik)^3}{4\pi}\int_{0}^{1}l^2e^{ikl\abs{x-y}}dl\frac{(x-y)(x-y)^T}{\abs{x-y}^2}
                                                                                 &\\
                                                                                 &\hspace{-6cm}\frac{-(ik)^2}{4\pi}\int_{0}^{1}le^{ikl\abs{x-y}}dl
                                                                                 \left(\frac{I}{\abs{x-y}}+\frac{(x-y)(x-y)^T}{\abs{x-y}^3}\right).
                                                                                 \end{align*}

Integrating by part the first term of the right-hand side, gives  
                              \begin{equation} 
                                   \begin{aligned}
                                   \nabla_x\nabla_y(\Phi_k(x,y)-\Phi_0(x,y))
                                   =&\Bigl(k^2\Phi_k(x,y)-\frac{k^2}{4\pi}\int_{0}^{1}2l \frac{e^{ikl\abs{x-y}}}{\abs{x-y}}dl\Bigr)\frac{(x-y)(x-y)^T}{\abs{x-y}^2}\\
                                   &+\frac{-(ik)^2}{4\pi}\int_{0}^{1}le^{ikl\abs{x-y}}dl
                                   \Bigl(\frac{I}{\abs{x-y}}+\frac{(x-y)(x-y)^T}{\abs{x-y}^3}\Bigr).
                                   \end{aligned}
                                \end{equation}

Hence, we get\begin{align*}
                           \abs{\Pi_k(z_i,z_j)-\Pi_0(z_i,z_j)}
                           &= \abs{k^2\Phi_k(z_i,z_j)I-\nabla_x\nabla_y\Phi_k(z_i,z_j)+\nabla_x\nabla_y\Phi_0(z_i,z_j)},\\
                           &\leq \abs{k^2\Phi_k(z_i,z_j)}+\abs{\nabla_x\nabla_y\Phi_k(z_i,z_j)-\nabla_x\nabla_y\Phi_0(z_i,z_j)}\leq \frac{6\abs{k}^2}{4\pi \delta_{i,j}}.
                           \end{align*}
 Now, as \begin{align}
            \bigl<\Bigl(\Sigma^k_{1,1}-\Sigma^0_{1,1}\Bigr)\mathcal{Q}_{1,1}\widehat{\mathcal{C}}_1, \mathcal{Q}_{1,1}\widehat{\mathcal{C}}_1\bigr>_{\mathbb{C}^{3\times m}}
            =\sum_{i=1}^{m}\Bigl(\sum_{1\leq j\neq i}^{m}(\Pi_k(z_i,z_j)-\Pi_0(z_i,z_j))Q_j\widehat{\mathcal{C}}_j\Bigr)\cdotp \overline{Q_i\widehat{\mathcal{C}}_i}
           \end{align} 
 using Holder's inequality, for the inner sum, we get 
                                                      \begin{align*}
                                                      \abs{\bigl<\Bigl(\Sigma^k_{1,1}-\Sigma^0_{1,1}\Bigr)\mathcal{Q}_{1,1}\widehat{\mathcal{C}}_1, \mathcal{Q}_{1,1}\widehat{\mathcal{C}}_1\bigr>}
                                                      \leq\sum_{i=1}^{m} \Bigl(\sum_{1\leq j\neq i}^{m}\abs{\Pi_k(z_i,z_j)-\Pi_0(z_i,z_j)}^2\Bigr)^\frac{1}{2}
                                                      \Bigl(\sum_{1\leq j\neq i}^{m}\abs{Q_j\widehat{\mathcal{C}}_j}^2\Bigr)^\frac{1}{2} \abs{Q_i\widehat{\mathcal{C}}_i},
                                                      \end{align*}
 which gives in view of \eqref{Mean-Vvalue-Double-Gradient} \begin{align*}
                                                      \abs{\bigl<\Bigl(\Sigma^k_{1,1}-\Sigma^0_{1,1}\Bigr)\mathcal{Q}_{1,1}\widehat{\mathcal{C}}_1, \mathcal{Q}_{1,1}\widehat{\mathcal{C}}_1\bigr>_{\mathbb{C}^{3\times m}}}
                                                     \leq \sum_{i=1}^{m} \bigl(\sum_{1\leq j\neq i}^{m}(\frac{3\abs{k}^2}{2\pi \delta_{i,j}})^2\bigr)^\frac{1}{2}
                                                      \Bigl(\sum_{j=1}^{m}\abs{Q_j\widehat{\mathcal{C}}_j}^2\Bigr)^\frac{1}{2} \abs{Q_i\widehat{\mathcal{C}}_i}.
                                                     \end{align*} 
Repeating Holder's inequality for the outer sum, we obtain

\begin{equation}\label{suiteestim}\begin{aligned}
                                                         \abs{\bigl<\Bigl(\Sigma^k_{1,1}-\Sigma^0_{1,1}\Bigr)\mathcal{Q}_{1,1}\widehat{\mathcal{C}}_1, \mathcal{Q}_{1,1}\widehat{\mathcal{C}}_1\bigr>_{\mathbb{C}^{3\times m}}}
                                                      \leq&\Bigl(\sum_{i=1}^{m} \sum_{1\leq j\neq i}^{m}( \frac{3\abs{k}^2}{2\pi \delta_{i,j}})^2
                                                      \sum_{j=1}^{m}\abs{Q_j\widehat{\mathcal{C}}_j}^2\Bigr)^\frac{1}{2} \Bigl(\sum_{i=1}^{m}\abs{Q_i\widehat{\mathcal{C}}_i}^2\Bigr)^\frac{1}{2},\\
                                                      \leq& \biggl(\sum_{i=1}^{m} \sum_{1\leq j\neq i}^{m}( \frac{6\abs{k}^2}{4\pi \delta_{i,j}})^2\biggr)^\frac{1}{2} 
\Bigl(\sum_{j=1}^{m}\abs{Q_j\widehat{\mathcal{C}}_j}^2\Bigr).
                                                      \end{aligned} \end{equation} 
                                                      
The inner sum gives, as we did it in \eqref{Exhaustiv-counting},~ 
                                                                       $$\sum_{1\leq j\neq i}^{m}(\frac{6\abs{k}^2}{4\pi \delta_{i,j}})^2
                                                                       \leq\sum_{l= 2}^{m^\frac{1}{3}} 7l^2 ~\frac{3^2\abs{k}^4}{4^2\pi^2 l^2\delta^2}
                                                                       = 7 m^\frac{1}{3}\frac{3^2\abs{k}^4}{4^2\pi^2\delta^2},$$
                                                                                      and then
                                                                                       \begin{equation*}
                                                                                                \bigl(\sum_{i=1}^{m} \sum_{1\leq j\neq i}^{m}(\frac{3\abs{k}^2}{2\pi \delta_{i,j}} )^2\bigr)^\frac{1}{2}
                                                                                                \leq m^\frac{1}{2}\bigl(7 m^\frac{1}{3}\frac{3^2\abs{k}^4}{4^2\pi^2\delta^2}\bigr)^\frac{1}{2}
                                                                                                \leq \frac{63^\frac{1}{2}\abs{k}^2m^{\frac{2}{3}}}{4\pi\delta}.
                                                                                                \end{equation*}
                                                                                                
Repeating the same calculation for $\abs{\bigl<\left(\Sigma^k_{2,2}-\Sigma^0_{2,2}\right)\mathcal{Q}_{2,2}\widehat{\mathcal{C}}_2, \mathcal{Q}_{2,2}\widehat{\mathcal{C}}_2\bigr>_{\mathbb{C}^{3\times m}}}$ 
leads to the conclusion.
For the last assertion (\ref{Differestimate2}), we proceed as follows:

\begin{equation}\label{Teta-k-estimates}
                        \begin{aligned}
                        \bigl<\varTheta^k\mathcal{Q}\widehat{\mathcal{C}},\mathcal{Q}\widehat{\mathcal{C}}\bigr>_{\mathbb{C}^{3\times 2m}}
                        =& 
                        \sum_{i=1}^m\bigl(\sum_{j\neq i}^m \nabla \Phi_k(z_i,z_j)\times Q_{j+m}\widehat{\mathcal{C}}_{j+m}\bigr)\cdotp \overline{Q_i\widehat{\mathcal{C}}_i}\\
                        &+\sum_{i=1}^m\bigl(\sum_{j\neq i}^m k^2\nabla \Phi_k(z_{i},z_{j}) Q_{j}\widehat{\mathcal{C}}_{j}\bigr)\cdotp \overline{Q}_{i+m}\overline{\widehat{\mathcal{C\,}}}_{i+m},
                             \end{aligned}
                             \end{equation}
                             
The first term of the right hand side of \eqref{Teta-k-estimates} is smaller then
                                               \begin{align*}
                                               \frac{1}{4\pi}\sum_{i=1}^m \sum_{j\neq i}^m   \frac{\abs*{Q_{j+m}\widehat{\mathcal{C}}_{j+m}}}{\delta_{i,j}}(\frac{1}{\delta_{i,j}}+\abs{k}) \abs*{ Q_i\widehat{\mathcal{C}}_i},
                                               \end{align*} 
which is,
                                              \begin{align*}
                                               \frac{1}{4\pi}\sum_{i=1}^m \sum_{j\neq i}^m   \frac{\abs*{Q_{j+m}\widehat{\mathcal{C}}_{j+m}}}{\delta_{i,j}}\frac{ \abs*{ Q_i\widehat{\mathcal{C}}_i}}{\delta_{i,j}}+
                                               \frac{1}{4\pi}\sum_{i=1}^m \sum_{j\neq i}^m   
                                                                                \frac{\abs{k}^\frac{1}{2}\abs*{Q_{j+m}\widehat{\mathcal{C}}_{j+m}}}{\delta_{i,j}^\frac{1}{2}} 
                                                                                \frac{\abs{k}^\frac{1}{2}\abs*{ Q_i\widehat{\mathcal{C}}_i}}{\delta_{i,j}^\frac{1}{2}},
                                              \end{align*}
and do not exceed \footnote{Comes from $2ab\leq a^2+b^2$ for every real numbers $a$, $b$.}
                                              \begin{align*}
                                               \frac{1}{8\pi}\sum_{i=1}^m \sum_{j\neq i}^m \Bigl(\frac{\abs*{Q_{j+m}\widehat{\mathcal{C}}_{j+m}}^2}{\delta_{i,j}^2}+
                                                                                                  \frac{ \abs*{ Q_i\widehat{\mathcal{C}}_i}^2}{\delta_{i,j}^2}\Bigr)+
                                               \frac{1}{8\pi}\sum_{i=1}^m \sum_{j\neq i}^m  \Bigl(\frac{\abs{k}\abs*{Q_{j+m}\widehat{\mathcal{C}}_{j+m}}^2}{\delta_{i,j}} 
                                                                                                  +\frac{\abs{k}\abs*{ Q_i\widehat{\mathcal{C}}_i}^2}{\delta_{i,j}}\Bigr),
                                              \end{align*}   
which in its turn, is not greater than\footnote{being $\sum_{i=1}^m \sum_{j\neq i}^m a_{i,j}=\sum_{j=1}^m \sum_{i\neq j }^m a_{i,j},$ for every real numbers $a_{i,j} $.} 
                                              \begin{align*}
                                              \frac{1}{8\pi}\sum_{j=1}^m \abs*{Q_{j+m}\widehat{\mathcal{C}}_{j+m}}^2 \sum_{i\neq j}^m\frac{1}{\delta_{i,j}^2} +
                                                                                               \frac{1}{8\pi}\sum_{i=1}^m \abs*{ Q_i\widehat{\mathcal{C}}_i}^2 \sum_{j\neq i}^m\frac{1}{\delta_{i,j}^2}\\
                                               +\frac{\abs{k}}{8\pi}\sum_{j=1}^m \abs*{Q_{j+m}\widehat{\mathcal{C}}_{j+m}}^2 \sum_{i\neq j}^m\frac{1}{\delta_{i,j}} +
                                                                                                \frac{\abs{k}}{8\pi} \sum_{i=1}^m \abs*{ Q_i\widehat{\mathcal{C}}_i}^2 \sum_{j\neq i}^m\frac{1}{\delta_{i,j}}.
                                              \end{align*}
We have, as done in \eqref{Exhaustiv-counting}, for $m\leq{D(\Omega)}/{\delta^3}$,
                                       \begin{align*}
                                        \sum_{j\neq i}^m  
                                                \frac{1}{\delta_{i,j}^2}=\sum_{l=2}^{({m}/{2})^\frac{1}{3}}\frac{C_0l^2}{l^2\delta^2}
                                                =\frac{C_0{(\frac{m}{2})^\frac{1}{3}}}{\delta^2}\leq \frac{C_0 D(\Omega)^\frac{1}{3}}{\delta^3}=\frac{C_{_{D(\Omega)}}^{(1)}}{\delta^3}.
                                       \end{align*} 
We get, in a similar way,
                                       \begin{align*}
                                        \sum_{j\neq i}^m  
                                                \frac{1}{\delta_{i,j}}=\sum_{l=2}^{({m}/{2})^\frac{1}{3}} 
                                                \frac{C_0l^2}{l\delta^2}
                                                =\frac{C_0{(\frac{m}{2})^\frac{1}{3}\bigl((\frac{m}{2})^\frac{1}{3}+1\bigr)}}{2\delta}\leq 
                                                \frac{C_0 \bigl(D(\Omega)^\frac{2}{3}+D(\Omega)^\frac{1}{3}\bigr)}{2\delta^3}=\frac{C_{_{D(\Omega)}}^{(2)}}{\delta^3}.   
                                       \end{align*} 
                                       
Then, for $C_{k,_{D(\Omega)}}:=\max(C_{_{D(\Omega)}}^{(1)},\abs{k}C_{_{D(\Omega)}}^{(2)})$, we obtain
                                     \begin{align*}
                                        \biggl|{\sum_{i=1}^m\sum_{j\neq i}^m \nabla \Phi_k(z_i,z_j)\times Q_{i+m}\widehat{\mathcal{C}}_{i+m}\cdotp \overline{Q_i\widehat{\mathcal{C}}_i}}\biggr|
                                               &\leq \frac{1}{8\pi} \frac{C_{k,_{D(\Omega)}}}{\delta^3} \sum_{i= 1}^{2m}\abs{Q_i\widehat{\mathcal{C}}_i}.
                                       \end{align*} 
                                       
Repeating the same argument for the second term of the right-hand side of \eqref{Teta-k-estimates} gives the conclusion.
\end{proof}

\begin{proof}(of Proposition \ref{propositioncondinvsysl})
If we consider \eqref{Systemintermdiare}, in view of \eqref{DominatingTerm1} and \eqref{DominatingTerm2}, we get 
                                                   \begin{align*}
                                                   \bigl<\widehat{\mathcal{C}}, \mathcal{Q}\widehat{\mathcal{C}}\bigr>_{\mathbb{C}^{3\times 2m}}&+
                                                              \bigl<\varTheta^k\mathcal{Q}\widehat{\mathcal{C}}, \mathcal{Q}\widehat{\mathcal{C}}\bigr>_{\mathbb{C}^{3\times 2m}}+
                                                                \bigl<\left(\Sigma^k-\Sigma^0\right)\mathcal{Q}\widehat{\mathcal{C}}, \mathcal{Q}\widehat{\mathcal{C}}\bigr>_{\mathbb{C}^{3\times 2m}}\\
                                                              +&\frac{48^2 }{\pi^2\delta^6}\Bigl(\bigl<S_{\p \widehat{\Omega}}^0~\mathcal{U},\mathcal{U}\bigr>_{\Ll^2(\p \widehat{\Omega})}-
                                                                \sum_{i=1}^{m}\bigl<S_{{\p B_{\delta/4}^{z_i}}}^0~\mathcal{U},\mathcal{U}\bigr>_{\Ll^2({\p B_{\delta/4}^{z_i}})}\Bigr)\\
                                                              +&\frac{48^2 }{\pi^2\delta^6}\left(\bigl<S_{\p \widehat{\Omega}}^0\mathcal{V},\mathcal{V}\bigr>_{\Ll^2(\p \widehat{\Omega})}-
                                                                 \sum_{i=m+1}^{2m}\bigl<S_{{\p B_{\delta/4}^{z_{i-m}}}}^0~\mathcal{V},\mathcal{V}\bigr>_{\Ll^2({\p B_{\delta/4}^{z_{i-m}}})}\right)
                                                                 =\bigl<\mathcal{E}, \mathcal{Q}\widehat{\mathcal{C}}\bigr>_{\mathbb{C}^{3\times 2m}}.
                                                   \end{align*}
Using the fact that $\bigl<S_{\p \widehat{\Omega}}^0~\phi,\phi\bigr>_{\Ll^2(\p \widehat{\Omega})}\geq 0,$ for every $\phi\in\Ll^2(\p \Omega)$ we get 
                                                   \begin{align*}
                                                   \bigl<\widehat{\mathcal{C}}, \mathcal{Q}\widehat{\mathcal{C}}\bigr>_{\mathbb{C}^{3\times 2m}}+
                                                              \bigl<\varTheta^k\mathcal{Q}\widehat{\mathcal{C}}, \mathcal{Q}\widehat{\mathcal{C}}\bigr>_{\mathbb{C}^{3\times 2m}}+
                                                             &\bigl<\left(\Sigma^k-\Sigma^0\right)\mathcal{Q}\widehat{\mathcal{C}}, \mathcal{Q}\widehat{\mathcal{C}}\bigr>_{\mathbb{C}^{3\times 2m}}\\
                                                              &\hspace{-5cm} -\frac{48^2 }{\pi^2\delta^6}\Bigl(\sum_{i=1}^{m}
                                                              \bigl<S_{{\p B_{\delta/4}^{z_i}}}^0~\mathcal{U},\mathcal{U}\bigr>_{\Ll^2({\p B_{\delta/4}^{z_i}})}+
                                                              \sum_{i=m+1}^{2m}\bigl<S_{{\p B_{\delta/4}^{z_{i-m}}}}^0~\mathcal{V},\mathcal{V}\bigr>_{\Ll^2({\p B_{\delta/4}^{z_{i-m}}})}\Bigr)\leq
                                                              \bigl<\mathcal{E}, \mathcal{Q}\widehat{\mathcal{C}}\bigr>_{\mathbb{C}^{3\times 2m}}.
                                                   \end{align*} 
By the definition of $\mathcal{Q}$ and the inequalities \eqref{Tensor-Inequalities}, we have 
                                                                     \begin{align}\label{Sys-Linear-inmind}
                                                                     \bigl<\mathcal{Q}\widehat{\mathcal{C}}, \mathcal{Q}\widehat{\mathcal{C}}\bigr>_{\mathbb{C}^{3\times 2m}}
                                                                     \leq \mu^+\epsilon^3 \bigl<\widehat{\mathcal{C}}, \mathcal{Q}\widehat{\mathcal{C}}\bigr>_{\mathbb{C}^{3\times 2m}}.
\end{align} 

Using \eqref{Differestimate} and \eqref{Differestimate2}, we arrive at 
                                          \begin{equation}\label{Last-Drift}
                                                   \begin{aligned}
                                                   &\Bigl(1-\bigl[ \frac{(1+\abs{k}^2)C_{k,_{D(\Omega)}}}{8\pi} +
                                                   \frac{63^\frac{1}{2} \abs{k}^2D(\Omega)^{\frac{2}{3}}}{4\pi}\bigr]\frac{\mu^+\epsilon^3}{\delta^3} \Bigr)
                                                                                           \bigl<\widehat{\mathcal{C}},\mathcal{Q}\widehat{\mathcal{C}}\bigr>_{\mathbb{C}^{3\times 2m}}\\
                                                   &- \frac{48^2 }{\pi^2\delta^6}\Bigl(\sum_{i=1}^{m}
                                                   \bigl<S_{{\p B_{\delta/4}^{z_i}}}^0~\mathcal{U},\mathcal{U}\bigr>_{\Ll^2({\p B_{\delta/4}^{z_i}})}+
                                                   \sum_{i=m+1}^{2m}\hspace{-1mm}\bigl<S_{{\p B_{\delta/4}^{z_{i-m}}}}^0~\mathcal{V},\mathcal{V}\bigr>_{\Ll^2({\p B_{\delta/4}^{z_{i-m}}})}\Bigr)
                                                   \leq \bigl<\mathcal{E}, \mathcal{Q}\widehat{\mathcal{C}}\bigr>_{\mathbb{C}^{3\times 2m}}.
                                                   \end{aligned}
                                          \end{equation}  

Further, the following scaling inequality holds
                                                \begin{align*}
                                                 \bigl<S_{{\p B_{\delta/4}^{z_{i}}}}^0~\mathcal{U},\mathcal{U}\bigr>_{\Ll^2({\p B_{\delta/4}^{z_{i}}})}
                                                 &=
                                                 \int_{B_{\delta/4}^{z_{i}}}\int_{B_{\delta/4}^{z_{i}}} 
                                                 (\Phi_0(x,y)\nu_x\cdotp{Q_i\widehat{\mathcal{C}}_i}~ds(y))\nu_y \cdotp \overline{Q_i\widehat{\mathcal{C}}_i}~ds(x),\\
                                                 &=
                                                 \int_{\p B_{1}^{z_{i}}}\Bigl(\int_{\p B_1^{z_i}} \frac{1}{(\frac{\delta}{4})}~\Phi_0(s,t)~ \nu_t \cdotp {Q_i\widehat{\mathcal{C}}_i} (\frac{\delta}{4})^2~ds(t)\Bigr)
                                                 \nu_s~ \cdotp \overline{Q_i\widehat{\mathcal{C}}_i} (\frac{\delta}{4})^2~ds(s),\\
                                                 &\leq
                                                 (\frac{\delta^3}{4^3}) \norm*{S_{_{B^1}}}\abs{Q_i\widehat{\mathcal{C}}_i} \abs{\p B_1^0}= (\frac{4\pi\delta^3}{4^3}) \norm*{S_{_{B^1}}}\abs{Q_i\widehat{\mathcal{C}}_i}^2 ,
                                                 \end{align*} 
which remains valid for $\bigl<S_{{\p B_{\delta/4}^{z_{i-m}}}}^0~\mathcal{V},\mathcal{V}\bigr>_{\Ll^2({\p B_{\delta/4}^{z_{i-m}}})},$
                   gives in \eqref{Last-Drift} \begin{align*}
                                               \Bigl(1-\bigl[ \frac{(1+\abs{k}^2)C_{k,_{D(\Omega)}}}{8\pi} +
                                               \frac{63^\frac{1}{2} \abs{k}^2D(\Omega)^{\frac{2}{3}}}{4\pi}\bigr]& \frac{\mu^+\epsilon^3}{\delta^3} \Bigr)
                                               \bigl<\widehat{\mathcal{C}}, \mathcal{Q}\widehat{\mathcal{C}}\bigr>_{\mathbb{C}^{3\times 2m}}\\
                                               &-\frac{12^2 \norm*{S_{_{B^1}}} }{\pi\delta^3}\sum_{i=1}^{2m}\abs{Q_i\widehat{\mathcal{C}}_i}^2\leq  
                                               \bigl<\mathcal{E}, \mathcal{Q}\widehat{\mathcal{C}}\bigr>_{\mathbb{C}^{3\times 2m}}.
                                               \end{align*} 
                                               
As $\bigl<\mathcal{E}, \mathcal{Q}\widehat{\mathcal{C}}\bigr>_{\mathbb{C}^{3\times 2m}}\leq 
                                               \bigl<\widehat{\mathcal{C}}, \mathcal{Q}\widehat{\mathcal{C}}\bigr>_{\mathbb{C}^{3\times 2m}}^\frac{1}{2}\bigl<\mathcal{E}, 
                                               \mathcal{Q}\mathcal{E}\bigr>_{\mathbb{C}^{3\times 2m}}^\frac{1}{2}$
    we get, with \eqref{Sys-Linear-inmind} \begin{align*}
                                            \Bigl(1-{\bigl[ \frac{(1+\abs{k}^2)C_{k,_{D(\Omega)}}}{8\pi} +\frac{12^2\norm*{S_{_{B^1}}} }{\pi}+
                                            \frac{63^\frac{1}{2} \abs{k}^2D(\Omega)^{\frac{2}{3}}}{4\pi}\bigr]}\frac{\mu^+\epsilon^3}{\delta^3} \Bigr)&\bigl<\widehat{\mathcal{C}}, 
                                            \mathcal{Q}\widehat{\mathcal{C}}\bigr>_{\mathbb{C}^{3\times 2m}}\\
                                           &\leq\bigl<\widehat{\mathcal{C}}, \mathcal{Q}\widehat{\mathcal{C}}\bigr>_{\mathbb{C}^{3\times 2m}}^\frac{1}{2}
                                             \bigl<\mathcal{E}, \mathcal{Q}\mathcal{E}\bigr>_{\mathbb{C}^{3\times 2m}}^\frac{1}{2},
                                            \end{align*} 
which is precisely   \begin{align}\label{Transforme-Into-Estima}
                                            \left(1- C_{Ls}\frac{\mu^+\epsilon^3}{\delta^3} \right)\bigl<\widehat{\mathcal{C}}, \mathcal{Q}\widehat{\mathcal{C}}\bigr>_{\mathbb{C}^{3\times 2m}}^\frac{1}{2}\leq 
                                             \bigl<\mathcal{E}, \mathcal{Q}\mathcal{E}\bigr>_{\mathbb{C}^{3\times 2m}}^\frac{1}{2}
                                            \end{align} 
where we set 
\begin{equation}
 \label{constant-Cls}
 C_{Ls}:={\bigl[ \frac{(1+\abs{k}^2)C_0 \Bigl((1+\frac{\abs{k}}{2}) D(\Omega)^\frac{1}{3}+\frac{\abs{k}}{2}D(\Omega)^\frac{2}{3}\Bigr) }{8\pi} +\frac{12^2\norm*{S_{_{B^1}}} }{\pi}+
                                            \frac{63^\frac{1}{2} \abs{k}^2D(\Omega)^{\frac{2}{3}}}{4\pi}\bigr]} 
\end{equation}
recalling the constants $C_{k,_{D(\Omega)}}$.
                                            
                                            Then a sufficient condition for the solvability of \eqref{SYSTOT} is given by $C_{Ls}\frac{\mu^+\epsilon^3}{\delta^3}< 1$. Further, if the previous condition is satisfied, then from\eqref{Transforme-Into-Estima}, considering \eqref{Tensor-Inequalities}, we get
                                                     \begin{align}\label{Estimates-For-C}
                                                      \left(1- C_{Ls}\frac{\mu^+\epsilon^3}{\delta^3} \right)\mu^-\epsilon^3\bigl<\widehat{\mathcal{C}},\widehat{\mathcal{C}}\bigr>_{\mathbb{C}^{3\times 2m}}^\frac{1}{2}\leq 
                                             \mu^+\epsilon^3\bigl<\mathcal{E},\mathcal{E}\bigr>_{\mathbb{C}^{3\times 2m}}^\frac{1}{2},
                                                     \end{align} 
and the definition of $\mathcal{C},$ yields inverting \eqref{Tensor-Inequalities} 
                                           \begin{align*}
                                            \bigl<\widehat{\mathcal{C}},\widehat{\mathcal{C}}\bigr>_{\mathbb{C}^{3\times 2m}}^\frac{1}{2}
                                            =&   \left(\bigl<\mathcal{Q}_{1,1}^{-1}\widehat{\mathcal{B}},\mathcal{Q}_{1,1}^{-1}\widehat{\mathcal{B}}\bigr>_{\mathbb{C}^{3\times m}}
                                                     +\bigl<\mathcal{Q}_{2,2}^{-1}\widehat{\mathcal{A}},\mathcal{Q}_{2,2}^{-1}\widehat{\mathcal{A}}\bigr>_{\mathbb{C}^{3\times m}}\right)^\frac{1}{2},\\
                                          \geq&\frac{\mu^+}{\epsilon^3} \left(\bigl<\widehat{\mathcal{B}},\widehat{\mathcal{B}}\bigr>_{\mathbb{C}^{3\times m}}
                                                    +\bigl<\widehat{\mathcal{A}},\widehat{\mathcal{A}}\bigr>_{\mathbb{C}^{3\times m}}\right)^\frac{1}{2}.
                                           \end{align*} 

Replacing in \eqref{Estimates-For-C}, we obtain \begin{align*}
                                            \left(1- C_{Ls}\frac{\mu^+\epsilon^3}{\delta^3} \right)\mu^-
                                            \mu^+ \left(\bigl<\widehat{\mathcal{B}},\widehat{\mathcal{B}}\bigr>_{\mathbb{C}^{3\times m}}
                                                    +\bigl<\widehat{\mathcal{A}},\widehat{\mathcal{A}}\bigr>_{\mathbb{C}^{3\times m}}\right)^\frac{1}{2}\leq 
                                                    \mu^+\epsilon^3\bigl<\mathcal{E},\mathcal{E}\bigr>_{\mathbb{C}^{3\times 2m}}^\frac{1}{2}.
                                            \end{align*} 

Concerning \eqref{Estimate-Each-Solution-LinSys}, it suffice to observe that 
                                            \begin{align*}
                                             \abs{\Sigma^k}&\leq 2\sum_{i=1}^m\abs{\Pi_k(z_i,z_j)}\leq \frac{1}{2\pi}\sum_{i=1}^m\Bigl(\frac{\abs{k}^2}{\delta_{i,j}}+\frac{3}{\delta_{i,j}}\bigl(\frac{1}{\delta_{i,j}}+\abs{k}\bigr)^2\Bigr),\\
                                             \abs{\varTheta^k}&\leq2\sum_{i=1}^m\abs{\nabla\Phi_k(z_i,z_j)}\leq \frac{1}{2\pi}\sum_{i=1}^m\frac{1}{\delta_{i,j}}\bigl(\frac{1}{\delta_{i,j}}+\abs{k}\bigr)
                                            \end{align*} 
which, summing as in \eqref{Exhaustiv-counting} gives, 
                                            \begin{align*}
                                             \abs{\Sigma^k\mathcal{Q}}+\abs{\varTheta^k\mathcal{Q}}\leq
                                             \abs{\Sigma^k}+\abs{\varTheta^k}\leq 4 \mu^+ \left(\frac{\ln{m^\frac{1}{3}}}{\delta^3}+\frac{2k m^\frac{1}{3}}{\delta^2} +\frac{m^\frac{2}{3}}{2\delta}k^2\right)\epsilon^3
                                            \end{align*} with this, it comes that, for 
                                            $$C_{L^2_i}={1-4\mu^+\left(\frac{\ln{m^\frac{1}{3}}}{\delta^3}+\frac{2k m^\frac{1}{3}}{\delta^2} +\frac{m^\frac{2}{3}}{2\delta}k^2\right)\epsilon^3}$$ 
                                            \begin{align}
                                            \abs{\mathcal{C}_i}\leq \frac{1}{C_{L^2_i}}\mathcal{E}_i.
                                            \end{align}

\end{proof}

\section{End of the proof of Theorem \ref{letheorem}}\label{Section-end-of-the-proof}
With the notations of the previous section, the linear system (\eqref{SYS2},\eqref{SYS1}) becomes 
                                                              \begin{align}\label{Systemdebase} 
                                                               \mathcal{C}+\Sigma^k\mathcal{Q} \mathcal{C}+\varTheta^k\mathcal{Q} {\mathcal{C}}=\mathcal{E}+\Varepsilon(\epsilon,\delta,&\abs{k},m)\epsilon^3,
                                                              \end{align} 
                  with $({\mathcal{C}}_i)_{i\in\{1,...,2m\}}$ defined as 
                                                              \begin{equation}
                                         {\mathcal{C}}_i:= \left\{\begin{aligned}
                                                               &\Vmt^{-1}{\mathcal{B}}_i,~ &&~i\in\{1,...,m\},\\
                                                               &-\bigl[\mathcal{P}_{\p D_{i-m}}\bigr]^{-1}{\mathcal{A}}_{i-m},~ &&~i\in\{m+1,...,2m\},
                                                              \end{aligned}\right.
                                                              \end{equation}
                                                         and ${\Varepsilon}(\epsilon,\delta,\abs{k},m)=({\Varepsilon_i}(\epsilon,\delta,\abs{k},m))_{i=1}^m$ with
                                                              \begin{equation}\label{Definition-of-Varepsilon}
            {\Varepsilon_i}(\epsilon,\delta,\abs{k},m):= \left\{\begin{aligned}
                                                               &\frac{\epsilon^4}{\delta^4}+ \Varepsilon_{k,\delta,m}  \epsilon^4+(1+\abs{k})\epsilon
                                                               ,~ &&~i\in\{1,...,m\},\\
                                                               &\frac{\epsilon^4}{\delta^4}+ \abs{k}\Varepsilon_{k,\delta,m}  \epsilon^4+\abs{k}^2\epsilon,
                                                               ~ &&~i\in\{m+1,...,2m\}.
                                                              \end{aligned}\right.
                                                              \end{equation} 
The difference between \eqref{Systemdebase} and \eqref{Linear-System-Transform} implies 
\begin{align}
                                                                                     (\mathcal{C}-\widehat{\mathcal{C}})+
                                                                                     \Sigma^k\mathcal{Q} (\mathcal{C}-\widehat{\mathcal{C}})+\varTheta^k\mathcal{Q} (\mathcal{C}-\widehat{\mathcal{C}})
                                                                                     =\Varepsilon(\epsilon,\delta,&\abs{k},m),
                                                                                    \end{align} 
           which gives, with the estimates \eqref{Estimate-Solutions-of-LinSys} \begin{align}\label{Estimation-of-the-Approximation-AhatA}
                                                                                \sum_{i=1}^m
                                                                                \Bigl(\abs*{\widehat{\mathcal{A}}_i-\mathcal{A}_i}^2+
                                                                                \abs*{\widehat{\mathcal{B}}_i-\mathcal{B}_i}^2\Bigr)\leq \frac{1}{C_{Li}\mu^-} \Bigl(\frac{\epsilon^4}{\delta^4}
                                                                                +(1+\abs{k})\Varepsilon_{k,\delta,m} \epsilon^4+\max(\abs{k}^2,1+\abs{k} )\epsilon\Bigr)^2
                                                                                ~2m~\epsilon^6,
                                                                                \end{align} and, with \eqref{Estimate-Each-Solution-LinSys}
                                                                                \begin{equation}\label{Estimate-Each-Approximation-AhatA}
                                                                                \left\{\begin{aligned}
                                                                                \abs*{\widehat{\mathcal{A}}_i-\mathcal{A}_i}
                                                                                \leq \frac{1}{C_{L_i^2}\mu^-} \Bigl(\frac{\epsilon^4}{\delta^4}
                                                                                +(1+\abs{k})\Varepsilon_{k,\delta,m} \epsilon^4+\max(\abs{k}^2,1+\abs{k})\epsilon\Bigr)^2
                                                                                ~\epsilon^3,\\
                                                                                \abs*{\widehat{\mathcal{B}}_i-\mathcal{B}_i}
                                                                                \leq \frac{1}{C_{L_i^2}\mu^-}\Bigl(\frac{\epsilon^4}{\delta^4}
                                                                                +(1+\abs{k})\Varepsilon_{k,\delta,m} \epsilon^4+\max(\abs{k}^2,1+\abs{k})\epsilon\Bigr)^2
                                                                                ~\epsilon^3.
                                                                                \end{aligned}\right.
                                                                                \end{equation}
                                                                                We set $$O^\Varepsilon(\frac{\epsilon^4}{\delta^4}):=O\Bigl(\frac{\epsilon^4}{\delta^4}
                                                                                +(1+\abs{k})\Varepsilon_{k,\delta,m} \epsilon^4+\max(\abs{k}^2,1+\abs{k})\epsilon\Bigr).$$
\begin{lemma}
 We have the following asymptotic approximation for the far field,    
    \begin{equation*}
        \begin{aligned}
        E^\infty(\tau)=&\frac{ik}{4\pi}\sum_{i=1}^m e^{-ik\tau.z_i}\tau\times 
        \bigl(\widehat{\mathcal{A}}_i-ik\tau\times \widehat{\mathcal{B}}_i\bigr) +O\left(   ~ (\abs{k}^3+\abs{k}^2)~ m\epsilon^4\right)\\
        &+\frac{\abs{k}}{2\pi}~\frac{\max(1,\abs{k})}{C_{Li}\mu^-}~O^\Varepsilon\Bigl(\frac{\epsilon^4}{\delta^4}\Bigr)m\epsilon^3
        \end{aligned}
   \end{equation*}
    and the following one for the scattered field
 \begin{align*}
  E^\s(x)=&\sum_{i=1}^m \left(  \nabla\Phi_k(x,z_i)\times\widehat{\mathcal{A}}_i+
                       \curl \curl (\Phi_k(x,z_i) \widehat{\mathcal{B}}_i)\right)+\frac{(C_{L_i^2}\mu^-)^{-1}}{\mu^+}\times O^\Varepsilon(\frac{\epsilon^4}{\delta^4})\\
                       &+O\bigl((C_{L_i^2}\mu^-)^{-1}\frac{\epsilon^7}{\delta^7}+\Varepsilon(\delta^6,\abs{k}+\abs{k}^2+\abs{k}^3)\epsilon^7+\Varepsilon(\delta^5,\abs{k}^2)\epsilon^7\bigr).
\end{align*}
\end{lemma}
\begin{proof}
 Recalling the approximation of the far field in Proposition \ref{Field-Approximation-First-Step},  we have 
                                       \begin{equation}\label{Far-Field-Approximation2}\begin{aligned}
                                         E^\infty(\tau)=&\frac{ik}{4\pi}\sum_{i=1}^m e^{-ik\tau.z_i}\tau\times 
                             \bigl((\mathcal{A}_i-\widehat{\mathcal{A}}_i)-ik\tau\times (\mathcal{B}_i-\widehat{\mathcal{B}_i}\bigr)\\
                              &+\frac{ik}{4\pi}\sum_{i=1}^m e^{-ik\tau.z_i}\tau\times 
                             \bigl(\widehat{\mathcal{A}}_i-ik\tau\times \widehat{\mathcal{B}}_i\bigr) +O\left( (\abs{k}^3+\abs{k}^2)~ m\epsilon^4\right).
                                       \end{aligned}
                                       \end{equation} 
For the first term of the right hand side, we have
                                          \begin{align*}
                                                \abs*{\frac{ik}{4\pi}
                                                \sum_{i=1}^m e^{-ik\tau.z_i}\tau\times 
                                                \bigl((\mathcal{A}_i-\widehat{\mathcal{A}}_i)-ik\tau\times (\mathcal{B}_i-\widehat{\mathcal{B}_i})\bigr)}&\\
                                               &\hspace{-3cm}\leq2~ \frac{\abs{k}}{4\pi}\max(1,\abs{k})~m^\frac{1}{2}
                                               \Bigl(\sum_{i=1}^m \bigl(\abs{\mathcal{A}_i-\widehat{\mathcal{A}}_i}^2+\abs{\mathcal{B}_i-\widehat{\mathcal{B}_i}}^2\bigr)\Bigr)^\frac{1}{2},\\
                                               &\hspace{-3cm}\leq \frac{\abs{k}}{2\pi}\frac{\max(1,\abs{k})}{C_{L_i\mu^-}}~O^\Varepsilon\Bigl(\frac{\epsilon^4}{\delta^4}\Bigr)m\epsilon^3.
                                          \end{align*} 

With this estimate, \eqref{Far-Field-Approximation2} becomes
                                 \begin{equation}         
                                          \begin{aligned}
                                            E^\infty(\tau)=&\frac{ik}{4\pi}\sum_{i=1}^m e^{-ik\tau.z_i}\tau\times 
                                                            \bigl(\widehat{\mathcal{A}}_i-ik\tau\times \widehat{\mathcal{B}}_i\bigr) +O\left(   ~ (\abs{k}^3+\abs{k}^2)~ m\epsilon^4\right)\\
                                                           &+\frac{\abs{k}}{2\pi}\frac{\max(1,\abs{k})}{C_{L_i\mu^-}}~O^\Varepsilon\Bigl(\frac{\epsilon^4}{\delta^4}\Bigr)m\epsilon^3.
                                          \end{aligned}
                                 \end{equation}
Again, in view of Proposition \ref{Field-Approximation-First-Step}, we have 
                                            \begin{equation}\label{Field-Approximation-LinSys-Without-Error}
                                            \begin{aligned}
                                            E^\s(x)=\sum_{i=1}^m& \left(  \nabla\Phi_k(x,z_i)\times\bigl(\mathcal{A}_i-\widehat{\mathcal{A}}_i\bigr)+  
                                                    \curl \curl (\Phi_k(x,z_i) \bigl(\mathcal{B}_i-\widehat{\mathcal{B}}_i\bigr))\right)\\
                                                    &+\sum_{i=1}^m \left(  \nabla\Phi_k(x,z_i)\times\widehat{\mathcal{A}}_i+  
                                                    \curl \curl (\Phi_k(x,z_i) \widehat{\mathcal{B}}_i)\right)\\
                                                   &+O\bigl(\frac{\epsilon^4}{\delta^4}+\Varepsilon_{k,\delta,m}~\epsilon^4 \bigr).
                                           \end{aligned} 
                                           \end{equation}
 Let $i_0$ be as in \eqref{Min-indices-Satisfying-d(x,D_i)},
 from the representation of the linear system we have \footnote{Notice that $-\varPi_k(x,y)=\nabla_y\times\nabla_x\times \Phi_k(x,y)\; I$.} 
                                                     \begin{align*}
                                                      \sum_{(i\neq i_0)\geq1}^m \biggl(  \nabla\Phi_k(z_{i_0},z_i)\times\bigl(\mathcal{A}_i-\widehat{\mathcal{A}}_i\bigr)+&  
                                                    \curl \curl (\Phi_k(z_{i_0},z_i) \bigl(\mathcal{B}_i-\widehat{\mathcal{B}}_i\bigr))\biggr)\\
                                                    =&
                                                    [\mathcal{T}_{\p D_{i_0}}]^{-1}\bigl(\mathcal{B}_{i_0}-\widehat{\mathcal{B}}_{i_0}\bigr)+{\Varepsilon_{i_0}}(\epsilon,\delta,\abs{k},m),
                                                      \end{align*}
                                                    hence, adding and subtracting the last identity to \eqref{Field-Approximation-LinSys-Without-Error} gives 
                        \begin{equation}\label{Field-Approximation-LinSys-Without-Error2}
                       \begin{aligned}
                        E^\s(x)=&\biggl(  \nabla\Phi_k(x,z_{i_0})\times\bigl(\mathcal{A}_{i_0}-\widehat{\mathcal{A}}_{i_0}\bigr)+  
                                                    \curl \curl (\Phi_k(x,z_{i_0}) \bigl(\mathcal{B}_{i_0}-\widehat{\mathcal{B}}_{i_0}\bigr))\biggr)\\
                       &+\sum_{(i\neq i_0)\geq1}^m  \Bigl[(\nabla\Phi_k(x,z_i)-\nabla\Phi_k(z_{i_0},z_i))\times\bigl(\mathcal{A}_{i_0}-\widehat{\mathcal{A}}_{i_0}\bigr)\\
                       &\hspace{1,5cm}+ \curl \curl \Bigl(\bigl(\Phi_k(x,z_i)-\Phi_k(z_{i_0},z_i)\bigr)\bigl(\mathcal{B}_i-\widehat{\mathcal{B}}_i\bigr)\Bigr)\Bigr]\\
                       &+[\mathcal{T}_{\p D_{i_0}}]^{-1}\bigl(\mathcal{B}_{i_0}-\widehat{\mathcal{B}}_{i_0}\bigr)+{\Varepsilon_{i_0}}(\epsilon,\delta,\abs{k},m)\\
                       &+\sum_{i=1}^m \left(  \nabla\Phi_k(x,z_i)\times\widehat{\mathcal{A}}_i+  
                       \curl \curl (\Phi_k(x,z_i) \widehat{\mathcal{B}}_i)\right)\\
                       &+O\bigl(\frac{\epsilon^4}{\delta^4}+\Varepsilon_{k,\delta,m}~\epsilon^4 \bigr).
                       \end{aligned} 
                       \end{equation}
 For $x\in\p \Omega,$  $\frac{1}{d_{x,i_0}}=\frac{1}{\delta}$ and then the first term of the right hand side of \eqref{Field-Approximation-LinSys-Without-Error2} 
 is smaller then \begin{align*}
                 \biggl(\frac{1}{\delta}\bigl(\frac{1}{\delta}+\abs{k}\bigr)\abs*{\mathcal{A}_{i_0}-\widehat{\mathcal{A}}_{i_0}}
                 +\Bigl(\frac{\abs{k}^2}{\delta}+\frac{1}{\delta}\bigl(\frac{1}{\delta}+\abs{k}\bigr)^2\Bigr)
                 \abs*{\mathcal{B}_{i_0}-\widehat{\mathcal{B}}_{i_0}}\biggr)
                 \end{align*}
                 which gives, considering \eqref{Estimate-Each-Approximation-AhatA} 
                                                                          \begin{equation*}
                                                                          (C_{L_i^2}\mu^-)^{-1} \Bigl(\frac{1}{\delta^3}+\frac{1+2\abs{k}}{\delta^2}+\frac{\abs{k}+2\abs{k}^2}{\delta}\Bigr)
                                                                           \Bigl(\frac{\epsilon^4}{\delta^4}
                                                                                +(1+\abs{k})\Varepsilon_{k,\delta,m} \epsilon^4+\max(\abs{k}^2,1+\abs{k})\epsilon\Bigr)^2
                                                                                ~\epsilon^3,
                                                                           \end{equation*} which is \begin{equation}\label{First-Term-Error-Field-Approximation}
                                                                                                     O\bigl((C_{L_i^2}\mu^-)^{-1}\frac{\epsilon^7}{\delta^7}+\Varepsilon(\delta^6,\abs{k})\epsilon^7+
                                                                                                                    \Varepsilon(\delta^5,\abs{k}^2)\epsilon^7\bigr).
                                                                                                    \end{equation}
   The second term of \eqref{Field-Approximation-LinSys-Without-Error2} is exactly \begin{align*}
                                                                            \sum_{(i\neq i_0)\geq1}^m \biggl[&\bigl(\int_{[0,1]}\nabla_x\nabla_x\Phi_k(tx+(1-t)z_{i_0},z_i)~dt\cdotp(x-z_{i_0})\bigr)\times
                                                                                            \bigl(\mathcal{A}_{i_0}-\widehat{\mathcal{A}}_{i_0}\bigr)\\
                                                                                   &+ \bigl(\int_{[0,1]}\nabla_x\varPi_k(tx+(1-t)z_{i_0},z_i)~dt\cdotp(x-z_{i_0})\bigr)
                                                                                      \bigl(\mathcal{B}_i-\widehat{\mathcal{B}}_i\bigr)\biggr], 
                                                                                  \end{align*}
     which turns out to be not greater then \begin{align*}
                                            \sum_{(i\neq i_0)\geq1}^m \biggl[\frac{1}{\delta_{i_0,i}}\bigl(\frac{1}{\delta_{i_0,i}}+\abs{k}\bigr)^2\delta\abs*{\mathcal{A}_{i_0}-\widehat{\mathcal{A}}_{i_0}}
                                            + \Bigl(\frac{\abs{k}^2}{\delta_{i_0,i}}\bigl(\frac{1}{\delta_{i_0,i}}+\abs{k}\bigr)+
                                                                            \frac{1}{\delta_{i_0,i}}\bigl(\frac{1}{\delta_{i_0,i}}+\abs{k}\bigr)^3\Bigr)\delta
                                                                                      \abs{\mathcal{B}_i-\widehat{\mathcal{B}}_i}\biggr],
                                           \end{align*} and, due to \eqref{Estimate-Each-Approximation-AhatA}, not exceed 
                                           \begin{align*}
                                           \sum_{(i\neq i_0)\geq1}^m \biggl[\frac{1}{\delta_{i_0,i}}\bigl(\frac{1}{\delta_{i_0,i}}+\abs{k}\bigr)^2\delta
                                            + \Bigl(\frac{\abs{k}^2}{\delta_{i_0,i}}\bigl(\frac{1}{\delta_{i_0,i}}+\abs{k}\bigr)+
                                                                            \frac{1}{\delta_{i_0,i}}\bigl(\frac{1}{\delta_{i_0,i}}+\abs{k}\bigr)^3\Bigr)\delta\biggr]\\
                                           \times (C_{L_i^2}\mu^-)^{-1} \Bigl(\frac{\epsilon^4}{\delta^4}
                                                                                +(1+\abs{k})\Varepsilon_{k,\delta,m} \epsilon^4+\max(\abs{k}^2,1+\abs{k})\epsilon\Bigr)^2
                                                                                ~2\epsilon^3,
                                           \end{align*} 
     hence summing as in \eqref{Exhaustiv-counting} gives \begin{align*}
                                                           O\Bigl( \frac{1}{\delta^4}+\frac{(1+\abs{k})\ln(m^{1/3})}{\delta^3}+
                                                           \frac{(\abs{k}+\abs{k}^2)m^{1/3}}{\delta^2}+\frac{(\abs{k}^3+\abs{k}^2)m^{2/3}}{\delta}\Bigr)\delta\\
                                                           \times (C_{L_i^2}\mu^-)^{-1} \Bigl(\frac{\epsilon^4}{\delta^4}
                                                                                +(1+\abs{k})\Varepsilon_{k,\delta,m} \epsilon^4+\max(\abs{k}^2,1+\abs{k})\epsilon\Bigr)^2
                                                                                ~2\epsilon^3,
                                                          \end{align*} which is, for $m=O(1/\delta^3),$ the analogue of \eqref{First-Term-Error-Field-Approximation} with the following additional term 
                                                          \begin{align}\label{Additional-Term-Error-Field-Approximation}
                                                            \Varepsilon(\delta^6,\abs{k}^2+\abs{k}^3)\epsilon^7+\Varepsilon(\delta^5,\abs{k}^2)\epsilon^7.
                                                          \end{align}  
The third term of (\ref{Field-Approximation-LinSys-Without-Error2}), due to \eqref{Tensor-Inequalities} and \eqref{Estimate-Each-Approximation-AhatA}, is bounded by
\begin{align}
\frac{(C_{L_i^2}\mu^-)^{-1}}{\mu^+} \Bigl(\frac{\epsilon^4}{\delta^4}+(1+\abs{k})\Varepsilon_{k,\delta,m} \epsilon^4+\max(\abs{k}^2,1+\abs{k})\epsilon\Bigr)^2
=\frac{(C_{L_i^2}\mu^-)^{-1}}{\mu^+}\times O^\Varepsilon(\frac{\epsilon^4}{\delta^4}),
\end{align} compiling this last error  with \eqref{First-Term-Error-Field-Approximation} and the additional term \eqref{Additional-Term-Error-Field-Approximation} gives us the result.
\end{proof}


%

\end{document}